\documentclass[reqno]{amsart}
\usepackage{textcomp}
\usepackage{lmodern}

\usepackage[dvipsnames]{xcolor}
\usepackage{amsmath}
\usepackage{amsthm}
\usepackage{hyperref}
\usepackage{amsfonts,graphics,amsthm,amsfonts,amscd,latexsym}
\usepackage{tikz-cd}
\usepackage{epsfig}
\usepackage{flafter}
\usepackage{longtable}
\usepackage{mathtools}
\usepackage{comment}
\usepackage{makecell}
\usepackage[all]{xy}

\hypersetup{
    colorlinks=true,    
    linkcolor=blue,          
    citecolor=blue,      
    filecolor=blue,      
    urlcolor=blue           
}

\usepackage{tikz}
\usetikzlibrary{graphs,positioning,arrows,shapes.misc,decorations.pathmorphing}

\tikzset{
    >=stealth,
    every picture/.style={thick},
    graphs/every graph/.style={empty nodes},
}

\tikzstyle{vertex}=[
    draw,
    circle,
    fill=black,
    inner sep=1pt,
    minimum width=5pt,
]
\usepackage[position=top]{subfig}
\usepackage{amssymb}
\usepackage{color}

\calclayout

\usetikzlibrary{decorations.pathmorphing}
\tikzstyle{printersafe}=[decoration={snake,amplitude=0pt}]

\newcommand{\Cl}{\operatorname{Cl}}

\newcommand{\mult}{\operatorname{mult}}

\newcommand{\rank}{\operatorname{rank}}

\newcommand{\ddivv}{\operatorname{div}}
\newcommand{\ord}{\operatorname{ord}}

\newcommand{\pp}{\mathbb{P}}

\newcommand{\qq}{\mathbb{Q}}
\newcommand{\zz}{\mathbb{Z}}

\newcommand{\cc}{\mathbb{C}}

\newcommand{\oo}{\mathcal{O}}

\def\O#1.{\mathcal {O}_{#1}}			
\def\pr #1.{\mathbb P^{#1}}				
\def\af #1.{\mathbb A^{#1}}			
\def\ses#1.#2.#3.{0\to #1\to #2\to #3 \to 0}	
\def\xrar#1.{\xrightarrow{#1}}			
\def\K#1.{K_{#1}}						
\def\bA#1.{\mathbf{A}_{#1}}			
\def\bM#1.{\mathbf{M}_{#1}}				
\def\bL#1.{\mathbf{L}_{#1}}				
\def\bB#1.{\mathbf{B}_{#1}}				
\def\bK#1.{\mathbf{K}_{#1}}			
\def\subs#1.{_{#1}}					
\def\sups#1.{^{#1}}

\tikzset{
  bu panel/.style={
    x=1mm,y=1mm,line width=0.35pt,
    every node/.style={inner sep=0pt,outer sep=0pt,font=\normalsize}
  },
  bu exceptional/.style={
    line width=0.4pt,line cap=butt,
    dash pattern=on 1.6pt off 1.6pt
  }
}

\newcommand{\BUFrame}{
  \path[use as bounding box] (-19,-15.8) rectangle (19,15.8);
  \draw (-13.3,11.8) -- (13.3,11.8);
  \draw (-13.3,-11.8) -- (13.3,-11.8);
  \draw (-12,-13.2) -- (-12,13.2);
  \draw (12,-13.2) -- (12,13.2);
  \node at (0,13.9) {$s_\infty$};
  \node at (0,-13.8) {$s_0$};
  \node at (-14.3,2.1) {$f_0$};
  \node at (15,9) {$f_\infty$};
}

\newcommand{\BUTPanel}{
  \vcenter{\hbox{
    \begin{tikzpicture}[bu panel]
      \BUFrame
      \fill (-7.2,-11.8) circle[radius=0.62mm];
      \fill (6.5,-11.8) circle[radius=0.62mm];
      \fill (12,0) circle[radius=0.62mm];
      \node at (-7.2,-8.8) {$p_1$};
      \node at (6.5,-8.8) {$p_k$};
      \node at (0,-8.8) {$\cdots$};
      \node at (14.2,0) {$q$};
    \end{tikzpicture}
  }}
}

\newcommand{\BUXPanel}{
  \vcenter{\hbox{
    \begin{tikzpicture}[bu panel]
      \BUFrame
      \foreach \dx in {0,11.8}{
        \begin{scope}[xshift=\dx mm]
          \draw[bu exceptional,dash phase=2.82pt]
            (-8,-4)
            .. controls (-5.793524,-6.603911)
                    and (-4.860191,-9.603911) .. (-5.2,-13);
        \end{scope}
      }
      \draw[bu exceptional,dash phase=2pt]
        (6,7.1)
        .. controls (6.574084,2.821710)
                and (9.007417,0.455044) .. (13.3,0);
      \node at (-8,-8.2) {$E_1$};
      \node at (3.8,-8.2) {$E_k$};
      \node at (-1,-8.8) {$\cdots$};
      \node at (4,2.7) {$E_q$};
    \end{tikzpicture}
  }}
}

\newcommand{\BlowupDiagram}{
  \xymatrix@C=28mm@R=4mm@M=0mm{
    \BUTPanel & \BUXPanel \ar@<1mm>[l]_{\textstyle\pi} \\
    \hbox{\large $T\simeq\mathbb F_2$} & \hbox{\large $X$}
  }
}

\usepackage{tikz}
\usetikzlibrary{matrix,arrows,decorations.pathmorphing}

  \newtheorem{theorem}{Theorem}[section]
  \newtheorem{lemma}[theorem]{Lemma}
  \newtheorem{proposition}[theorem]{Proposition}
  \newtheorem{corollary}[theorem]{Corollary}

  \newtheorem{notation}[theorem]{Notation}
  \newtheorem{definition}[theorem]{Definition}
  \newtheorem{example}[theorem]{Example}

\newtheorem{remark}[theorem]{Remark}

\theoremstyle{remark}

\numberwithin{equation}{section}

\usepackage[all]{xy}

\begin{document}

\thanks{E.A.S. was supported by the State Secretariat for Education, Research and Innovation (SERI) of Switzerland REF-1131-552105. J.M. was partially supported by NSF research grant DMS-2443425.}

\title[Cluster type smooth Fano threefolds]{Cluster type smooth Fano threefolds}

\author[E. Alves da Silva]{Eduardo Alves da Silva}
\address{Departement Mathematik und Informatik,  Universität Basel, Spiegelgasse 1, 4051 Basel, Switzerland.
}
\email{eduardo.alvesdasilva@unibas.ch}

\author[F.~Figueroa]{Fernando Figueroa}
\address{Department of Mathematics, Northwestern University, Evanston, IL 60208, USA
}
\email{fernando.figueroa@northwestern.edu}

\author[J.~Moraga]{Joaqu\'in Moraga}
\address{UCLA Mathematics Department, Box 951555, Los Angeles, CA 90095-1555, USA
}
\email{jmoraga@math.ucla.edu}

\author[J.I.~Y\'a\~nez]{Jos\'e Ignacio Y\'a\~nez}
\address{Departamento de Matem\'atica, Universidad T\'ecnica Federico Santa Mar\'ia, Avenida Espa\~na 1680, Valpara\'iso, Chile}
\email{jose.yaneze@usm.cl}


\begin{abstract}
We prove that most rational smooth Fano threefolds
are compactifications of $\mathbb{G}_m^3$.
More precisely, we study which smooth Fano threefolds are cluster type, i.e., for which smooth Fano threefolds $X$ one can find a boundary $B\in |-K_X|$
and an embedding $\mathbb{G}_m^3\hookrightarrow X\setminus B$. We prove that many important classes of smooth Fano threefolds are cluster type. 
For instance, a smooth Fano threefold which is the blow-up of either $\mathbb{P}^3$ or $\mathbb{Q}^3$ and is general in its deformation space is of cluster type.
As a consequence, we show that among the $105$ families
parametrizing all smooth Fano threefolds, the general element of precisely $87$ families is of cluster type. We prove that a general prime smooth Fano threefold of genus $7$ is not of cluster type. This is the only family whose general element is rational but not of cluster type.
Finally, we prove that cluster type projective klt varieties are cylindrical
and we give a criterion for such a cylinder to be anticanonically polar.
Using this, we reduce the cylindricity problem for smooth Fano threefolds to the families \textnumero\,1.3, \textnumero\,1.5, and \textnumero\,1.6.
\end{abstract} 

\maketitle

\setcounter{tocdepth}{1} 
\enlargethispage{4pt}
\tableofcontents

\section{Introduction}

Fano varieties are projective varieties whose anticanonical divisor is ample, and are one of the three building blocks of smooth projective varieties. Smooth Fano surfaces, or {\em del Pezzo} surfaces, form $10$ deformation families: $\pp^1\times\pp^1$ and the blow-ups of $\pp^2$ at $k\in\{0,\dots,8\}$ points in general position. In dimension three, some Fano varieties are irrational. The Picard rank one classification is due to Iskovskikh~\cite{Isk77,Isk78}, and the higher Picard rank classification to Mori and Mukai~\cite{MM81,MM03}. Iskovskikh and Prokhorov~\cite{IP99} give a comprehensive account. Crucial tools include the Kobayashi--Ochiai index bounds~\cite{KO73} and good anticanonical members~\cite{Sho79}. In~\cite{KMM92}, Koll\'ar, Miyaoka, and Mori proved that in each dimension there are finitely many families of smooth Fano varieties. Thus a full classification is possible in principle, although the higher-dimensional geometry remains poorly understood.

Cluster algebras and their Grassmannian examples originate in~\cite{FZ02,Scott06}. The geometric cluster type condition does not specify a seed. Recently, there has been some substantial interest in the connections between Fano geometry and cluster algebras~\cite{MM24,Mor24a,Mor24b,MY24,JM24,HMY25, Cor25, ELY25,EFMS25,MZ25,Mor26a}.
It is expected that {\em many} Fano varieties carry an affine chart that {\em resembles} the spectrum of a cluster algebra.
More precisely, it is expected that many Fano varieties are of cluster type (see Definition~\ref{def:cluster-type}).
In~\cite{JM24}, Ji and the third author proved that in each dimension $n$ there are finitely many families of smooth Fano varieties which are cluster type. 
Thus, in principle, we should be able to classify cluster type smooth Fano varieties in each dimension.
The geometry of cluster type varieties is more constrained than that of rational varieties, so it is natural to expect a better understanding of these varieties in higher dimensions.
A smooth surface $X$ is of cluster type if and only if the linear system $|-K_X|$ admits a reduced anticanonical cycle of rational curves, including an irreducible nodal rational curve~\cite{GHK15}.
In dimension two, a del Pezzo surface of degree at least two is cluster type
and a general del Pezzo surface of degree one is of cluster type~\cite{ALP23}. 
In dimension three, there are $105$ families of smooth Fano threefolds. 
It is known that the general member of $88$ of these families is a rational smooth Fano threefold~\cite{Con82,AB92}. 
In~\cite{LMV24}, Loginov, Vasilkov, and the third author proved that a general rational smooth Fano threefold is indeed log rational. 
This means that there exists $B\in |-K_X|$ such that the pair $(X,B)$ is crepant birationally equivalent to $(\pp^3,H_0+H_1+H_2+H_3)$ where the $H_i$'s are the coordinate hyperplanes.
One can think of log rationality as a weaker notion than cluster type. In~\cite{AdSFM26}, the authors studied cluster type boundaries on the three-dimensional projective space $\pp^3$. 
For instance, an slc quartic $B$ of coregularity zero with a point of multiplicity three gives a cluster type pair $(\pp^3,B)$ (see~\cite[Proposition 5.2]{AdSFM26}). 
The motivating problem for this project was to understand which smooth Fano blow-ups of $\pp^3$ are cluster type. If $X={\rm Bl}_C(\pp^3)$ and $C\subset B$, where $(\pp^3,B)$ is of cluster type, then $X$ is of cluster type (see Lemma~\ref{lem:blow-up-of-cluster-type}). Thus, for blow-ups $X$ of $\pp^3$, we reduce the problem of proving that $X$ is of cluster type to solving an interpolation problem for singular quartic hypersurfaces in $\pp^3$.
We solve these interpolation problems, for each deformation class of $X$, proving the following theorem.

\begin{theorem}\label{theorem:blow-ups-P^3}
Let $X$ be a smooth Fano threefold which is a blow-up of $\pp^3$.
If $X$ is general in its deformation space, then $X$ is of cluster type.
\end{theorem} 

The case of blow-ups of $\pp^3$ will be studied in Section~\ref{sec:p3}. 

Smooth quadrics are known to be cluster type varieties. 
For instance, if $\mathbb{Q}^n\subset \mathbb{P}^{n+1}$ is a smooth quadric, then the log 
Calabi--Yau pair $(\mathbb{Q}^n,H_2+\dots+H_{n+1})$ 
is of cluster type when the $H_i$'s are the coordinate hyperplanes. 
Moreover, a general hypersurface in a toric variety which has degree two in Cox coordinates is a cluster type variety~\cite{ELY25}. 
Solving interpolation problems, similar to the case of the projective space, we can also deduce that smooth Fano threefolds which are blow-ups of a three-dimensional quadric are also of cluster type.

\begin{theorem}\label{theorem:blow-up-Q^3}
Let $X$ be a smooth Fano threefold which is a blow-up
of $\mathbb{Q}^3$. 
If $X$ is general in its deformation space, then $X$ is of cluster type. 
\end{theorem} 

The case of blow-ups of $\qq^3$ will be studied in Section~\ref{sec:q3}.
Finally, we address the other possible isomorphism classes of smooth Fano threefolds. 
Many rational smooth Fano threefolds are either blow-ups of the projective space $\mathbb{P}^3$ or a three-dimensional quadric $\mathbb{Q}^3$.
In order to address the other cases, we use several different approaches. 
In some cases, we will use the complexity approach
to detect the cluster type property initiated in~\cite{ELY25}.
In other cases, we will study del Pezzo fibrations 
or conic fibration structures on the smooth Fano variety or one of its blow-ups.
In~\cite{MY24}, the third and fourth authors developed techniques to detect whether del Pezzo fibrations
or conic fibrations are cluster type (see Proposition~\ref{prop:fibration-to-smooth-curve} and Proposition~\ref{prop:fibration-to-smooth-surface}).
Using all these different tools, we prove the following theorem.

\begin{theorem}\label{theorem:cluster-type-Fano-threefolds}
Among the $105$ families parametrizing all smooth Fano threefolds, the general element of at least $87$ families is of cluster type.
\end{theorem} 

The cases which are not blow-ups of $\pp^3$ or $\qq^3$ are dealt with in Section~\ref{sec:main-series} and Section~\ref{sec:other}. 
In Section~\ref{sec:main-series} we study smooth Fanos in the main series and in Section~\ref{sec:other} we study several other smooth Fano threefolds.
In some of these cases, we prove that some natural boundary does not lead to a cluster type pair.
For instance, in the case of a double cover $X$ of $\pp^1\times \pp^2$ over a surface of degree $(2,2)$, a natural boundary $B\in |-K_X|$ has coregularity zero but $(X,B)$ is not cluster type (see Proposition~\ref{prop:double-cover-p1xp2}).  
However, in Proposition~\ref{prop:double-cover-p1xp2-cluster-type}, we are able to obtain a different cluster type boundary for this double cover.

\subsection{Applications to cylindrical varieties} In this subsection, we explain some applications of the cluster type condition to cylindrical varieties. 
A {\em cylinder} in a variety $X$ is an open subset $U\subseteq X$ isomorphic to $\mathbb{A}^1\times Z$ where $Z$ is an affine variety. 
In such a case, we say that $X$ is a {\em cylindrical} variety.
Cylindrical varieties have attracted a lot of attention in recent years (see, for instance~\cite{KZ01,SY02,DK15,CPW16,Saw24,KKW24,KKW25,KW25}). For example, there is an interesting connection of cylinders and additive group actions~\cite{KPZ13}.
We refer the reader to~\cite{CPPZ21} for a complete survey on the topic of cylindricity.

Every smooth rational projective surface contains a cylinder (see, e.g.,~\cite{KPZ11}).
Thus, every smooth del Pezzo surface is cylindrical. 
If $X$ is a cylindrical rationally connected variety, then the cylinder $\mathbb{A}^1\times Z$ must satisfy that $Z$ is rationally connected. 
A smooth projective model of the surface $Z$ is therefore rationally connected, hence rational. Thus a cylindrical rationally connected threefold is rational. 
In particular, a cylindrical smooth Fano threefold must be rational~\cite{Zha06}. 
Thus, there are examples of smooth Fano threefolds which are not cylindrical.
It is then natural to understand which smooth rational Fano varieties are cylindrical.
It is not known whether every rational smooth Fano threefold is cylindrical (see~\cite[Question 1.3]{CPPZ21}).
Some rational smooth Fano threefolds are conjectured not to be cylindrical~\cite[Conjecture 3.9 and Conjecture 3.13]{CPPZ21}. It is expected that most rational smooth Fano threefolds are cylindrical. 
However, the cases in which this statement is proven are limited. 
Cylindricity is known for del Pezzo threefolds and some special blow-ups of $\pp^3$ and $\qq^3$ (see~\cite[Examples 3.14-17]{CPPZ21}). In this direction, we show that cluster type varieties with mild singularities are indeed cylindrical.

\begin{theorem}\label{theorem:ct-implies-cylindrical}
Let $X$ be a projective cluster type variety
with klt singularities. Then $X$ is a cylindrical variety.
\end{theorem} 

The idea of the proof of Theorem~\ref{theorem:ct-implies-cylindrical} is simple 
and we give a short heuristic here.
The cluster type variety $X$ admits a birational map 
$\mathbb{G}_m^n \dashrightarrow X$ which is an open 
embedding on a big open subset of the algebraic torus. 
Let $B:={\rm Poles}(\Omega_n)$ be the cluster type boundary.
Since $X$ has mild singularities, the boundary $B$ must be non-trivial. For each prime component $P$ of $B$, the valuation ${\rm ord}_P$
determines a one-parameter subgroup
$\lambda_P\colon\mathbb G_m\to\mathbb G_m^n$.
The limits $\lim_{t\to0}g\lambda_P(t)$, for general
$g\in\mathbb G_m^n$, sweep out a dense open subset of $P$.
The previous statement can be proved by looking at toric models of the cluster type pair $(X,B)$ (see Proposition~\ref{prop:cluster-type-bir}).
General translates of a suitable one-dimensional subtorus have limits in an open subset of a chosen boundary component $P$. After shrinking the parameter space, they give a cylinder in $X\setminus\operatorname{Supp}(B-P)$. In Theorem~\ref{thm:anticanonical-polar-cylinder}, we will explain some criteria for the construction of this cylinder to yield a polar anticanonical cylinder (see Section~\ref{sec:cylindrical} for the definition).

Combining Theorem~\ref{theorem:ct-implies-cylindrical} with our main result on cluster type smooth Fano threefolds, Theorem~\ref{theorem:cluster-type-Fano-threefolds}, we obtain the following result.

\begin{corollary}\label{cor:cylindrical-smooth-Fano-threefolds}
Among the $105$ families parametrizing all smooth Fano threefolds, 
the general element of at least $87$ families admits a cylinder. 
\end{corollary}

It is worth mentioning that log rationality is not expected to imply cylindricity.
Indeed, in~\cite{LMV24}, Loginov, Vasilkov, and the third named author proved that a general rational smooth Fano threefold is indeed log rational.
Thus, a general smooth Fano threefold of Picard rank one, index one, and genus $7$ is conjectured to give an example of a log rational variety which is not cylindrical. 
It is still unclear how to detect whether a log rational variety is cylindrical.

\subsection{The Mori-Mukai family \textnumero\,1.6} 
Due to the Mori-Mukai classification, there are precisely 88 families of smooth Fano threefolds whose general element is rational.
By Theorem~\ref{theorem:cluster-type-Fano-threefolds} the general element of 87 of these families is of cluster type. 
The remaining family is the Mori-Mukai family \textnumero\,1.6 
of prime smooth Fano threefolds of genus $7$. 
In this case, we will prove that a general prime smooth Fano threefold of genus $7$ is not of cluster type. 
In this direction, we will develop new techniques to prove that a pair is not of cluster type. 
More precisely, we will prove the following lemma that allows us to construct singular movable linear systems on cluster type varieties.

\begin{lemma}\label{lem:movable-intro}
Let $X$ be a klt Fano variety. 
Let $(X,B)$ be a cluster type pair with ${\rm Cl}(X)\simeq \zz[B]$. 
Then, there exists a movable linear system $\mathcal{M}\subseteq |-mK_X|$ and a log canonical place $v$ of $(X,B)$ with 
\[
v(\mathcal{M})>mv(B). 
\]
Furthermore, we can choose the center of $v$ in $X$ to be a closed point.
\end{lemma}

The previous lemma will be proved in Section~\ref{sec:MM-1-6}.
Roughly speaking, the previous lemma says that on a cluster type pair we can always find a log canonical place $v$ of $(X,B)$ and some movable linear system whose general element is more singular than $B$ with respect to $v$. Thus, if the inequality $v(\mathcal{M})\leq mv(B)$ holds for any movable linear system $\mathcal{M}$ and divisorial valuation $v$, then $(X,B)$ is not of cluster type. A priori, this gives a good way to disprove the cluster type condition for a variety $X$.
However, there are two issues: first, one needs to perform the computation $v(\mathcal{M})\leq mv(B)$ uniformly for all movable systems and valuations, second, one needs to perform the computation for all possible $B\in |-K_X|$ that could lead to a cluster type boundary.
In Lemma~\ref{lem:criteria-under-adjunction}, we will prove that it suffices to show the inequality $v(\mathcal{M})\leq mv(B)$ after possibly performing adjunction to a subvariety.
This will help us to reduce potentially complicated computations
to more suitable surface computations. 

In Section~\ref{sec:MM-1-6}, we will classify the singularities of all the possible slc divisors $B\in |-K_X|$ of coregularity zero.
Here, $X$ is a general smooth prime Fano threefold of genus $7$. 
If $B\in |-K_X|$ is a normal boundary of coregularity zero, then we will show that it has a unique $T_{3,3,4}$ cuspidal singularity.
This will be achieved by studying the Tjurina numbers of the singularities of $B$ in Subsection~\ref{subsec:normal}. 
If $B\in |-K_X|$ is a non-normal boundary of coregularity zero, then we will show that the non-normal locus is a single line $\ell\subset X$.
This will be achieved by studying spinor duality of the spinor tenfold containing $X$ in Subsection~\ref{subsec:non-normal}.
Once we have classified the possible log canonical centers of $(X,B)$, we perform a dlt modification and use Lemma~\ref{lem:criteria-under-adjunction} to show $v(\mathcal{M})\leq mv(B)$ via a sequence of surface computations. 
Putting the previous ideas together, we obtain a proof of the following theorem. 

\begin{theorem}\label{introthm:1.6-not-cluster}
A general prime smooth Fano threefold of genus $7$ is not of cluster type.
\end{theorem} 

Note that Theorem~\ref{theorem:cluster-type-Fano-threefolds} together with Theorem~\ref{introthm:1.6-not-cluster} finishes the problem of determining which general smooth Fano threefolds are cluster type. Indeed, there are precisely $87$ families of smooth Fano threefolds whose general element is of cluster type. 
In the final subsection, we will explain what happens for special members of each family.

\subsection{Special cases of smooth Fano threefolds} In this last subsection, we explain the cases of smooth Fano threefolds which are not yet known to be cluster type or cylindrical.
Due to the previous subsection, we mostly focus on special elements in each family. The main theorem is the most complete form of the statement about the cluster type property for rational smooth Fano threefolds.

\begin{theorem}\label{thm:complete-cluster-type}
Let $X$ be a rational smooth Fano threefold. 
Then $X$ is of cluster type unless:
\begin{enumerate}
\item[(i)] the general element of its deformation family is non-rational; 
\item[(ii)] $X$ is an element of the family \textnumero\,1.6 of prime smooth Fano threefolds of genus $7$;
\item[(iii)] $X$ is a special element of the family \textnumero\,2.4 of blow-ups of $\pp^3$ along the intersection of two cubics; or 
\item[(iv)] $X$ is a special element of the family \textnumero\,10.1 of Fano threefolds of the form $\pp^1\times S$ where $S$ is a del Pezzo surface of degree one.
\end{enumerate} 
\end{theorem} 

We recall that a general element of the Mori-Mukai families \textnumero\,1.3 and \textnumero\,1.5 is irrational, however, it is not known whether special elements of these families are rational. These are precisely the varieties that appear in Theorem~\ref{thm:complete-cluster-type}.(i). 
In the case of Theorem~\ref{thm:complete-cluster-type}.(ii), we know that a general prime smooth Fano threefold with volume $12$ is not cluster type by Theorem~\ref{thm:1-6-not-cluster-type}. However, it is unclear whether special elements of such family are cluster type. The proof in Section~\ref{sec:MM-1-6} crucially uses the fact that such a general smooth Fano threefold only contains ordinary lines and that some suitable singularities in $B\in |-K_X|$ do not appear for such a general $X$. A general blow-up of $\pp^3$ along the intersection of two cubics is of cluster type (see Lemma~\ref{lem:curves-in-cubic-surfaces}). However, some special blow-ups of $\pp^3$ along the intersection of two cubics are not of cluster type (see Section~\ref{sec:special-2.4}). It is still unclear what is the locus of such blow-ups that are not cluster type. 
Finally, the case of products $X\simeq \pp^1\times S$ can be completely understood. The variety $X$ is of cluster type if and only if ${\rm coreg}(S)=0$, which is an open condition for smooth del Pezzo surfaces of degree one. Indeed, if ${\rm coreg}(S)=0$, then $S$ is of cluster type and so $X$ is of cluster type due to Lemma~\ref{lem:product-cluster-type}. On the other hand, if $X$ has coregularity zero, then the image of any fibration from $X$ has coregularity zero as well (see, e.g.,~\cite{FFMP22}). We refer the reader to~\cite{Mor22} for the notion of coregularity of surfaces. In summary, for the item (iv) the cluster type condition is totally understood while for the families (i),(ii), and (iii) there are still some special members for which the cluster type property remains open.

Recall that any blow-up of $\pp^3$ along the intersection of two cubics is cylindrical (see~\cite[Example 3.15]{CPPZ21}). Furthermore, every smooth del Pezzo surface is cylindrical as well (see~\cite[Theorem 2.10]{CPPZ21}). 
Combining these two previous facts with Theorem~\ref{theorem:ct-implies-cylindrical}
we obtain the following statement about cylindricity of rational smooth Fano threefolds.

\begin{theorem}\label{thm:complete-cylindric}
Let $X$ be a rational smooth Fano threefold.
Then $X$ is cylindrical unless: 
\begin{enumerate}
\item the general element of its deformation family is non-rational, namely in family \textnumero\,1.3 or \textnumero\,1.5; or
\item $X$ belongs to family \textnumero\,1.6.
\end{enumerate} 
\end{theorem} 

We emphasize that the techniques developed in Section~\ref{sec:MM-1-6} only allow us to prove that a general element of the Mori-Mukai family \textnumero\,1.6 is not of cluster type. We still do not know if the general element of such family is cylindrical or not. However, it is expected that such general element is not cylindrical.

The paper contains two appendices. 
In Appendix~\ref{appendix-A}, we explain how AI was used in this paper. In Appendix~\ref{appendix-B}, for the convenience of the reader, we present a table with the list of all smooth Fano threefolds and make explicit where in the paper its cluster type is proved. 

\section{Preliminaries}

In this section, we recall the definition of cluster type varieties and give criteria for a pair to be of cluster type. We work over the field $\mathbb{C}$ of complex numbers. 

\subsection{Calabi--Yau pairs} 
For the notions of singularities of pairs we refer the reader to~\cite{K13}. We recall some notions related to log Calabi--Yau pairs.  

\begin{definition} 
{\em 
A {\em log Calabi--Yau pair} $(X,B)$ is a couple where $X$ is a normal projective variety and $B$ is an effective $\qq$-divisor such that $K_X+B\sim_\qq 0$ and $(X,B)$ has log canonical singularities.
The {\em coregularity} of $(X,B)$ is the smallest dimension of a log canonical center in a dlt modification of $(X,B)$.
For a klt pair, we set its coregularity equal to the dimension of the underlying variety. 
The {\em index} of $(X,B)$ is the smallest positive integer $i$ for which $i(K_X+B)\sim 0$.
}
\end{definition} 

\subsection{Cluster type pairs} 
In this subsection, we recall the definition of cluster type pairs and prove some structural properties of these pairs. 

\begin{definition}\label{def:codim-one-embedding}
{\em
Let $\phi\colon X \dashrightarrow Y$ be a rational map. 
We say that $\phi$ is an {\em (open) embedding in codimension one} if there exists a closed subset
$Z\subset X$ of codimension at least two for which 
$\phi|_{X\setminus Z}\colon X\setminus Z \dashrightarrow Y$ is indeed a morphism 
which is an (open) embedding.
}
\end{definition}

\begin{notation}\label{not:volume-form}
{\em 
Let $\mathbb{G}_m^n={\rm Spec}(\cc[x_1^\pm,\dots,x_n^\pm])$ be an $n$-dimensional algebraic torus over the complex numbers.
We will denote by
\[
\Omega_n:=\frac{dx_1\wedge \dots \wedge dx_n}{x_1\cdots x_n} 
\]
the volume form of the algebraic torus. 
}
\end{notation} 

\begin{definition}\label{def:cluster-type}
{\em 
Let $X$ be a normal variety of dimension $n$. 
We say that $X$ is of {\em cluster type} 
if there exists an open embedding in codimension one
$\phi\colon \mathbb{G}_m^n \dashrightarrow X$
such that $\Omega_n$ has no zeros on $X$.
If the chosen $\phi$ satisfies the previous condition, then the effective divisor $B={\rm Poles}(\Omega_n)$ is said to be a {\em cluster type boundary}. 
We may also say that $X$ is a {\em cluster type variety}
and that the pair $(X,B)$ is a {\em cluster type pair}.
In the previous setting, we may say that $B\subset X$ or $\phi\colon \mathbb{G}_m^n\dashrightarrow X$ gives a {\em cluster type structure} to $X$. 
}
\end{definition} 

\begin{remark}
{\em 
In ~\cite[Theorem 1.3.(4)]{EFM24}, the authors prove that a $\qq$-factorial klt Fano variety $X$ is of cluster type
if and only if there exists $B\in |-K_X|$ 
and an open embedding $\mathbb{G}_m^n\hookrightarrow X\setminus B$.
}
\end{remark} 

The following proposition gives an equivalent characterization of cluster type pairs.
This follows directly from Definition~\ref{def:cluster-type} (compare with~\cite[Definition 2.6 and Proposition 5.9]{Mor26a}). Here ${\rm Ex}(\phi)$ is the complement of the largest source open set on which $\phi$ is an open immersion, including its indeterminacy locus.

\begin{proposition}\label{prop:cluster-type-bir}
Let $(X,B)$ be a pair. Then, the pair 
$(X,B)$ is of cluster type if and only if there is a crepant birational map 
\[
\phi\colon (\pp^n,H_0+\dots+H_n)\dashrightarrow (X,B)
\]
such that ${\rm Ex}(\phi)\cap \mathbb{G}_m^n$ has codimension at least two in $\mathbb{G}_m^n$.
\end{proposition} 

Below, we may write $\Sigma^n$ for the sum of the coordinate hyperplanes of $\pp^n$.
The following three lemmas are already in the literature, see, e.g.,~\cite{JM24,AdSFM26,Mor26a}. As the proofs are short, we give a brief argument for the sake of completeness.

\begin{lemma}\label{lem:blow-up-of-cluster-type}
Let $(X,B)$ be a cluster type pair.
Let $\psi\colon Y\dashrightarrow X$ be a birational map whose inverse contracts no divisors and which extracts only log canonical places or canonical places of $(X,B)$. 
Write $\psi^*(K_X+B)=K_Y+B_Y$ and let $(Y,B_Y)$ be the induced log Calabi--Yau pair on $Y$.
Then, the pair $(Y,B_Y)$ is of cluster type.
\end{lemma} 

\begin{proof}
Since $(X,B)$ is of cluster type, by Proposition~\ref{prop:cluster-type-bir}, we have a crepant birational map $\phi\colon (\pp^n,\Sigma^n)\dashrightarrow (X,B)$ which is small in $\mathbb{G}_m^n$.
Since the map $(Y,B_Y)\dashrightarrow (X,B)$
is a crepant birational map, then we have an induced 
crepant birational map $\phi'\colon (\pp^n,\Sigma^n)\dashrightarrow (Y,B_Y)$.
Since $\psi$ only extracts log canonical places and canonical places of $(X,B)$, then the divisor $B_Y$ is effective and so $(Y,B_Y)$ is indeed a log Calabi--Yau pair.
Since $X\dashrightarrow Y$ does not contract divisors,
then the restriction of $\phi'$ to the algebraic torus is an open embedding in codimension one. Thus, we can apply Proposition~\ref{prop:cluster-type-bir} again to conclude that $(Y,B_Y)$ is of cluster type.
\end{proof} 

\begin{lemma}\label{lem:cluster-type-descends}
Let $(X,B)$ be a log Calabi--Yau pair.
Let $\psi\colon Y\dashrightarrow X$ be a projective birational map
that only extracts log canonical places of $(X,B)$.
Let $\psi^*(K_X+B)=K_Y+B_Y$.
If $(Y,B_Y)$ is of cluster type, then so is $(X,B)$.
\end{lemma} 

\begin{proof}
Assume that $(Y,B_Y)$ is of cluster type.
Therefore by Proposition~\ref{prop:cluster-type-bir}, we have a crepant birational map 
$\phi\colon (\pp^n,\Sigma^n)\dashrightarrow (Y,B_Y)$
which is an open embedding in codimension one in the algebraic torus.
By assumption, the restriction of $\psi$ to the complement of the boundary $\psi|_{Y\setminus B_Y} \colon Y\setminus B_Y \dashrightarrow X\setminus B$ is    an open immersion in codimension one.
Therefore, we conclude that the composition
$\psi\circ \phi\colon (\pp^n,\Sigma^n)\dashrightarrow (X,B)$ is a crepant birational map which is an open embedding in codimension one in the algebraic torus.
By Proposition~\ref{prop:cluster-type-bir}, we conclude that $(X,B)$ is a cluster type pair.
\end{proof}

\begin{lemma}\label{lem:product-cluster-type}
Let $(X_1,B_1)$ and $(X_2,B_2)$ be two cluster type pairs.
Let $\pi_i \colon X_1\times X_2 \rightarrow X_i$ be the 
projection onto the $i$-th component. 
Then, the pair $(X_1\times X_2, \pi_1^*B_1+\pi_2^*B_2)$ 
is of cluster type. 
\end{lemma} 

\begin{proof}
Since $(X_i,B_i)$ is a cluster type pair for each $i\in \{1,2\}$, then we can find a crepant birational map
$\phi_i\colon (\pp^{n_i},\Sigma^{n_i})\dashrightarrow (X_i,B_i)$ which is an open embedding in codimension one in the algebraic torus.
Thus, we obtain a crepant birational map
\[
\phi_1\times \phi_2 \colon 
(\pp^{n_1},\Sigma^{n_1})\times 
(\pp^{n_2},\Sigma^{n_2}) 
\dashrightarrow 
(X_1,B_1)\times (X_2,B_2) 
\]
which is an open embedding in codimension one when restricted to the algebraic torus $\mathbb{G}_m^{n_1}\times \mathbb{G}_m^{n_2}$. 
Any two toric log Calabi--Yau pairs of the same dimension are crepant birational to each other (see~\cite[Theorem 8.2.15]{CLS11}).
Furthermore, such crepant birational map can be chosen to be torus-equivariant so it is indeed an isomorphism on the algebraic torus. 
Let 
\[
\psi\colon (\pp^{n_1+n_2},\Sigma^{n_1+n_2}) 
\dashrightarrow 
(\pp^{n_1},\Sigma^{n_1}) \times 
(\pp^{n_2},\Sigma^{n_2})
\]
be the toric birational map. 
By construction, we have that 
$(\phi_1\times\phi_2)\circ \psi$ is a crepant birational map which is an open embedding in codimension one in the algebraic torus.
We conclude by Proposition~\ref{prop:cluster-type-bir} that the pair $(X_1,B_1)\times (X_2,B_2)$ is of cluster type.
\end{proof} 

\subsection{Criteria for cluster type pairs}
In this subsection, we collect some criteria for a pair to be of cluster type. The first criterion is related to singularities in the projective space while the last two criteria are about fibrations from threefolds to surfaces and curves, respectively. First, we recall the following theorem due to Gross, Hacking, and Keel~\cite[Proposition 1.3]{GHK15}.

\begin{theorem}\label{thm:cluster-type-smooth-surf}
Let $X$ be a smooth projective surface. Then a pair $(X,B)$ is of cluster type if and only if $B$ is a reduced anticanonical cycle of rational curves. 
\end{theorem} 

The following theorem was proved by the authors in~\cite[Proposition 5.2]{AdSFM26}.

\begin{proposition}\label{prop:criterion-sing-p3}
Let $B\subset \pp^3$ be an slc quartic.
If $(\pp^3,B)$ has coregularity zero and
there is a point $p\in \pp^3$ with ${\rm mult}_p(B)=3$, 
then the pair $(\pp^3,B)$ is of cluster type. 
\end{proposition} 

We use~\cite[Theorems 2 and 6]{ELY25}.
Below we use a special case of the complexity of a log Calabi--Yau pair $(X,B)$. Write $B=\sum_{i=1}^k b_iB_i$ where the $B_i$'s are the prime components.
We define the {\em complexity} of $(X,B)$ to be $\dim X + {\rm rank \ WDiv}_{\rm alg}(X)-\sum_{i=1}^k b_i$.

\begin{theorem}\label{thm:complexity-one-is-cluster-type}
Let $(X,B)$ be an index-one log Calabi--Yau pair of complexity at most one.
Then, the pair $(X,B)$ is of cluster type.
\end{theorem} 

We now give two criteria that allow to detect whether a conic fibration from a threefold or a del Pezzo fibration from a threefold are cluster type.
In order to do so, we will need to recall the concept of strict conic fibrations.

\begin{definition}\label{def:strict-conic-fibration}
{\em 
Let $f\colon X\rightarrow Z$ be a fibration. 
Let $(X,B)$ be a log pair. We say that $f\colon (X,B)\rightarrow Z$ is a {\em strict conic fibration} 
if the following conditions are satisfied:
\begin{enumerate} 
\item the morphism $f$ is a conic fibration of relative Picard rank one; 
\item the pair $(X,B)$ is log Calabi--Yau over $Z$; 
\item the divisor $\lfloor B\rfloor$ has two prime components $S_0$ and $S_\infty$ which are horizontal over $Z$; and 
\item the divisors $S_0$ and $S_\infty$ are disjoint.
\end{enumerate} 
}
\end{definition} 

The following theorem is proved in~\cite[Theorem 1.5]{MY24}.

\begin{theorem}\label{thm:strict-conic}
Let $f\colon (X,B)\rightarrow (Z,B_Z)$ be a crepant strict conic fibration between Calabi--Yau pairs. Then, the pair $(X,B)$ is of cluster type if and only if $(Z,B_Z)$ is of cluster type.
\end{theorem} 

We will use Theorem~\ref{thm:strict-conic} to prove a criterion for conic fibrations over smooth projective surfaces to be of cluster type.

\begin{proposition}\label{prop:fibration-to-smooth-surface}
Let $f\colon X\rightarrow S$ be a fibration from a projective threefold to a surface. 
Let $(X,B)$ be a log Calabi--Yau pair for which the following conditions are satisfied:
\begin{enumerate}
\item $S$ is a smooth surface; 
\item the pair $(X,B)$ has index one and coregularity zero; and
\item the boundary divisor $B$ has two prime components that are horizontal over $S$.
\end{enumerate} 
Then, the pair $(X,B)$ is of cluster type.
\end{proposition} 

\begin{proof}
First, we may replace $(X,B)$ with a $\qq$-factorial dlt modification. By Lemma~\ref{lem:cluster-type-descends}, it suffices to prove that this $\qq$-factorial dlt modification is of cluster type. 
components. 
Let $(S,B_S)$ be the log Calabi--Yau pair obtained
by the canonical bundle formula of $(X,B)$ to $S$. 
Then, the pair $(S,B_S)$ is a log Calabi--Yau pair of index one and coregularity zero. The previous statement can be proved by performing adjunction to $B_1$ and then pushing-forward to $S$ (see, e.g.,~\cite[Theorem 5.4]{Fil18}). 
Since $(S,B_S)$ is a log Calabi--Yau pair of index one, we have $B_S\in |-K_S|$. 
Since $S$ is smooth and $(S,B_S)$ has coregularity zero, we conclude that $B_S$ is a reduced anticanonical cycle of rational curves.
Therefore, by Theorem~\ref{thm:cluster-type-smooth-surf}, we conclude that the pair $(S,B_S)$ is of cluster type.
We can run a $K_X$-MMP over $S$ that terminates with a Mori fiber space $X\dashrightarrow X'\rightarrow S'$ over $S$.
We know that $X'\rightarrow S'$ is a Mori fiber space
and $\phi\colon S'\rightarrow S$ is a birational morphism. 
Let $B'$ be the push-forward of $B$ to $X'$.
Analogously, we write $B'_i$ for the push-forward of 
$B_i$ in $X'$. 
If we perform the canonical bundle formula for $(X',B')\rightarrow S'$, we obtain a pair $(S',B_{S'})$ which is log Calabi--Yau of index one and coregularity zero. Furthermore, the pair $(S',B_{S'})$ is obtained from $(S,B_S)$ via log pull-back, i.e., we have $K_{S'}+B_{S'}=\phi^*(K_S+B_S)$.
Thus, by Lemma~\ref{lem:blow-up-of-cluster-type} we conclude that the pair $(S',B_{S'})$ is a cluster type pair. 
By~\cite[Lemma 2.12]{MM24}, applied over $S'$, we have a commutative diagram of crepant maps between log Calabi--Yau pairs: 
\[
\xymatrix{ 
(X',B')\ar[d]^-{f'} & (X'',B'')\ar[d]^-{f''}\ar@{-->}[l]_-{\phi'} \\ 
(S',B_{S'}) & (S'',B_{S''})\ar[l]_-{\phi'_S}
}
\]
such that the following conditions are satisfied: 
\begin{enumerate}
\item[(i)] $X''$ is $\qq$-factorial;
\item[(ii)] $\phi'$ is a crepant birational map only extracting lc places of $(X',B')$; 
\item[(iii)] $\phi'_S$ is an isomorphism at the generic point of $S'$; and 
\item[(iv)] $f''^{-1}(B_{S''}) \subseteq B''$. 
\end{enumerate} 
Since $f'$ is a Mori fiber space, then we know that $f''$ is also a Mori fiber space by the last clause of~\cite[Lemma 2.12]{MM24}.
By construction, the boundary divisor $B''$ has two prime components $B''_1$ and $B''_2$ which are horizontal over $S''$.
Since $(S'',B_{S''})$ is a log Calabi--Yau pair, we can apply Lemma~\ref{lem:blow-up-of-cluster-type}  to conclude it is of cluster type. 

We argue that $f''$ is a strict conic fibration.
Note that conditions (1)-(3) of Definition~\ref{def:strict-conic-fibration} are satisfied, so it suffices to show that condition (4) of Definition~\ref{def:strict-conic-fibration} holds.
Assume that $B''_1$ and $B''_2$ intersect along the codimension two subset $Z$. 
The image $S_Z$ of $Z$ in $S''$ is a log canonical center of $(S'',B_{S''})$.
Therefore, it is contained in a component $\Gamma_1$ of $B_{S''}$.
By condition (iv) above, we would have that 
${f''}^{-1}(S_Z)\subset B''$ and so $B''$ has a vertical component $B''_3$ that contains $Z$. 
This means that the curve $Z$ is contained 
in three $\qq$-Cartier prime components of $B''$.
This contradicts~\cite[Theorem 2]{MS21}. This implies that $B''_1$ and $B''_2$ do not intersect on $X''$. 
Therefore, condition (4) of Definition~\ref{def:strict-conic-fibration} is satisfied and then $f''\colon (X'',B'')\rightarrow (S'',B_{S''})$ is a strict conic fibration. 
By Theorem~\ref{thm:strict-conic}, we conclude that $(X'',B'')$ is of cluster type. 
By construction, the crepant birational map 
$(X'',B'')\dashrightarrow (X,B)$ only extracts log canonical places of $(X,B)$.
Thus, the pair $(X,B)$ is of cluster type by Lemma~\ref{lem:cluster-type-descends}.
\end{proof} 

Now, we turn to study del Pezzo fibrations over smooth rational curves and give a criterion for these threefolds to be of cluster type.

\begin{proposition}\label{prop:fibration-to-smooth-curve}
Let $f\colon X\rightarrow C$ be a Fano type fibration from a projective 
threefold to a curve. 
Let $(X,B)$ be a log Calabi--Yau pair 
of index one satisfying the following conditions: 
\begin{enumerate} 
\item the curve $C$ is a smooth rational curve;
\item the pair induced on $C$ via the canonical bundle formula
is isomorphic to $(\pp^1,\{0\}+\{\infty\})$; 
\item $B$ has a unique horizontal prime component which restricts to a singular irreducible nodal curve on general fibers; and
\item the general fiber of $X\rightarrow C$ is a smooth del Pezzo of degree at least $5$.
\end{enumerate} 
Then, the pair $(X,B)$ is a cluster type pair.
\end{proposition} 

\begin{proof}
First, we reduce to the case in which $f^{-1}(\{0\}+\{\infty\})\subset \lfloor B\rfloor$. 
Indeed, since $\{0\}$ and $\{\infty\}$ are log canonical centers of the pair induced by the canonical bundle formula, we know that $(X,B)$ has two log canonical centers that map to $\{0\}$ and $\{\infty\}$ via $f$. 
By~\cite[Lemma 2.12]{MM24},  we can find a commutative diagram 
\[
\xymatrix{
(X,B)\ar[d]^-{f} & (Y,B_Y)\ar[ld]^-{f_Y}\ar@{-->}[l] \\
C 
}
\]
where $f_{Y}^{-1}(\{0\}+\{\infty\}) \subset B_Y$
and 
$(Y,B_Y)\dashrightarrow (X,B)$ is a crepant birational map that only extracts log canonical places of $(X,B)$.
Furthermore, the map $Y\dashrightarrow X$ is an isomorphism over the generic point of $C$.
If $(Y,B_Y)$ is of cluster type, then $(X,B)$ is of cluster type by Lemma~\ref{lem:cluster-type-descends}. 
Hence, by replacing $(X,B)$ with $(Y,B_Y)$ we may simply assume that the preimage of $\{0\}$ and $\{\infty\}$ in $X$ are contained in $\lfloor B\rfloor$.
We can also assume that $X$ is $\qq$-factorial (see~\cite[Lemma 2.12.(1)]{MM24}).

By the assumptions of the proposition, there is a one-dimensional log canonical center $Z$ of $(X,B)$ which dominates $C$ with a birational section. By condition (3), the restriction of $Z$ to a general fiber $F$ of $f$ is a nodal point of $B|_F \subset F$.
We blow up $Z$ and obtain $X'\rightarrow X$.
Let $(X',B'+E)$ be the pair obtained by log pull-back of $(X,B)$ to $X'$. 
Here, $E$ is the exceptional divisor of $X'\rightarrow X$.
By~\cite[Theorem 1]{Mor19}, we may assume that $X'$ is $\qq$-factorial. 
Let $f'\colon (X',B'+E)\rightarrow C$ be the induced fibration and we let $F'$ be the general fiber of $f'$. 
Let $(F',B_F'+E_F)$ be the log Calabi--Yau pair induced on the general fiber of $f'$.
Here, $B_F'$ is the strict transform of $B_F$ in $F'$
and $E_F$ is the exceptional divisor. 
Note that $E_F^2=-1$ and ${B_F'}^2\geq 1$. 
We run a $(K_{X'}+B')$-MMP over $C$.
Note that $K_{X'}+B' \sim_{\qq,C} -E$ and so the divisor $E$ is not contracted by this MMP.
On the generic surface fiber, $B'$ has positive self-intersection, which increases under the surface MMP, so it is not contracted. The divisor $E$ cannot be contracted, since a divisorial contraction of $E$ would require $E$-negative curves, whereas the MMP is $(-E)$-negative.
Furthermore the MMP must terminate with a Mori fiber space over $C$.
All steps are crepant for the log Calabi--Yau pair. Each contracted divisor has coefficient zero or one, so Lemma~\ref{lem:blow-up-of-cluster-type} applies to lift the cluster type property back along the MMP. On the other hand, Lemma~\ref{lem:cluster-type-descends} then descends it to the original pair. 
Let $X'\dashrightarrow X''\rightarrow Z \rightarrow C$ be the birational map $X'\dashrightarrow X''$ induced by the MMP and $X''\rightarrow Z$ be the Mori fiber space. Let $(X'',B''+E'')$ be the log pair structure induced on $X''$.
We will proceed in two different cases, depending whether $Z$ is a surface or a curve. 

First, assume that $Z$ is a curve.
In this case, we have $Z\simeq C\simeq \pp^1$. 
Moreover, the variety $X''$ has Picard rank two.
Furthermore, the boundary $B''+E''$ must have $4$ components; the two horizontal components over $Z$
and the pre-images of $\{0\}$ and $\{\infty\}$. 
Thus, the complexity of the log Calabi--Yau pair $(X'',B''+E'')$ is at most one
and so by Theorem~\ref{thm:complexity-one-is-cluster-type} we conclude that $(X'',B''+E'')$ is a cluster type pair. 
Therefore, the pair $(X,B)$ is of cluster type by Lemma~\ref{lem:cluster-type-descends}. 

Now, assume that $Z$ is a surface.
In this case both $E^{\prime\prime}$ and $B^{\prime\prime}$ are horizontal over $Z$. The former is positive on the fibers of the MFS, and the latter has positive self-intersection on the generic surface fiber. Moreover, a vertical curve has non-positive self-intersection. 
Let $(Z,B_Z)$ be the log Calabi--Yau pair obtained by the canonical bundle formula applied to $(X'',B''+E'')\rightarrow Z$.
As $B''+E''$ has prime components which dominate $Z$, we conclude that $(Z,B_Z)$ has index one and coregularity zero. Let $f_Z\colon Z \rightarrow C$ be the induced fibration. 
We have that $f_Z^{-1}(\{0\}+\{\infty\})\subset B_Z$.
We run a $K_Z$-MMP over $C$ that leads to a birational morphism $Z\rightarrow Z'$ and terminates with a Mori fiber space $Z'\rightarrow C$. 
Let $B_{Z'}$ be the push-forward of $B_Z$ in $Z'$. 
By construction, the variety $Z'$ is $\qq$-factorial and has Picard rank two. 
On the other hand, the boundary divisor $B_{Z'}$ has at least three components; the fiber over $\{0\}$ and $\{\infty\}$ and another horizontal component over $C$.
We conclude that $(Z',B_{Z'})$ has complexity at most one so by Theorem~\ref{thm:complexity-one-is-cluster-type} it is of cluster type.
Then, the pair $(Z,B_Z)$ is of cluster type by Lemma~\ref{lem:blow-up-of-cluster-type}. 
Now, the boundary divisor $B''+E''$ has two horizontal components over $Z$ and the pair $(Z,B_Z)$ is of cluster type. The same proof of Proposition~\ref{prop:fibration-to-smooth-surface} then implies that $(X'',B''+E'')$ is of cluster type.
Finally, applying Lemma~\ref{lem:cluster-type-descends} one more time, we conclude that $(X,B)$ is of cluster type. 
\end{proof} 

\section{Blow-ups of the three-dimensional projective space}\label{sec:p3}

In this section, we study smooth Fano threefolds which are blow-ups of $\pp^3$. The main theorem is the following.

\begin{theorem}\label{thm:blow-up-P3}
Let $X$ be a smooth Fano threefold which is the blow-up of $\mathbb{P}^3$. Then, the following statements hold:
\begin{enumerate} 
\item If $X$ is a smooth Fano threefold in the families \textnumero\,2.9, \textnumero\,2.12, \textnumero\,2.15, \textnumero\,2.19, \textnumero\,2.22, \textnumero\,2.25, \textnumero\,2.27, \textnumero\,2.28, \textnumero\,2.30, \textnumero\,3.6, \textnumero\,3.11, \textnumero\,3.12, \textnumero\,3.14, \textnumero\,3.16, \textnumero\,3.18, \textnumero\,3.23, or \textnumero\,4.6, then $X$ is of cluster type; and 
\item If $X$ is a smooth Fano threefold in the family \textnumero\,2.4, then $X$ is of cluster type provided it is general in its deformation space.
\end{enumerate} 
\end{theorem} 

The families in Theorem~\ref{thm:blow-up-P3}.(1) are blow-ups of $\pp^3$. These are not all such families, as other blow-ups of $\pp^3$ will be treated with different methods in the next sections.

\subsection{Intersections of cubic surfaces}

In this subsection, we study curves in $\pp^3$ that are the intersections of two, three, or four cubic surfaces and the blow-ups of these curves. We start by lifting some cluster type boundaries on $\pp^3$, which are composed of a cubic surface and a plane.

\begin{lemma}\label{lem:singular-cubic-surface-An} Let $C$ be a smooth curve in $\pp^3$, such that $C$ is contained in a singular cubic with only $A_n$ singularities. Then the blow-up of $\pp^3$ along $C$ is of cluster type.
\end{lemma}

\begin{proof}
Let $B$ be the singular cubic surface containing $C$. Let $H$ be a general plane containing a single $A_n$ singularity at $p\in B$. Then, the pair $(\pp^3,B+H)$ is log Calabi--Yau.
Indeed, near $p$, the pair $(B,H|_B;p)$ is a $1$-complement of the $A_n$ singularity $(B;p)$. 
Moreover, the boundary has multiplicity
${\rm mult}_p(B+H)=3$. Therefore, by Proposition~\ref{prop:criterion-sing-p3} the pair $(\pp^3,B+H)$ is of cluster type. Let $\pi:X\rightarrow \pp^3$ be the blow-up of $\pp^3$ along $C$. As ${\rm mult}_C(B+H)=1$ and $C$ is of codimension $2$ in $\pp^3$, $\pi$ extracts only a canonical place of $(\pp^3,B+H)$. Therefore, by Lemma~\ref{lem:blow-up-of-cluster-type}, the variety $X$ is of cluster type.
\end{proof}

The following lemma shows that the general intersection of two cubic surfaces is contained in boundaries of cluster type pairs as in Lemma~\ref{lem:singular-cubic-surface-An}.

\begin{lemma}\label{lem:curves-in-cubic-surfaces}
Let $C\subset \pp^3$ be a smooth irreducible curve of degree $9$ and genus $10$, which is the intersection of two cubic surfaces. Assume that $C$ is general in its deformation space. Then, there exists a singular cubic surface with only one $A_1$ singularity that contains $C$.
\end{lemma}

\begin{proof} The space of cubic surfaces in $\pp^3$ forms a $\pp^{19}$, and the space of singular ones constitutes a divisor $\Gamma \subset \pp^{19}$. By  \cite[Chapter 1 Remark 2.3]{huy23}, the space of cubic surfaces with a single $A_1$ singularity is an open subset $\Delta \subset \Gamma$.

Since the intersection of two general cubic surfaces in $\pp^3$ is a smooth irreducible curve of genus $10$ and degree $9$, the base locus of a general pencil of cubic surfaces $\Lambda$ is such a curve.
A smooth irreducible curve of degree $9$ and genus $10$ that is the complete intersection of two cubic surfaces is projectively normal, thus $h^0(\mathcal{I}_C(3))=h^0(\mathcal{O}_{\pp^3}(3))-h^0(\mathcal{O}_C(3))=20-18=2$, therefore the only cubic surfaces containing $C$ are the ones in the pencil.
Hence we have a correspondence between smooth irreducible curves of degree $9$ and genus $10$, which are the intersection of two cubic surfaces, and general pencils of cubic surfaces $\Lambda$.

To show that $C$ is contained in a singular cubic surface with only one $A_1$ singularity is therefore equivalent to showing that a general line in $\pp^{19}$ intersects $\Delta$. Since $\Gamma$ is a divisor, every line in $\pp^{19}$ intersects $\Gamma$. As $\Delta$ is a dense open subset of $\Gamma$, $\Gamma \setminus \Delta$ has codimension at least $2$ in $\pp^{19}$, thus a general line intersects $\Delta$. Therefore, a general smooth irreducible curve of degree $9$ and genus $10$, which is the intersection of two cubic surfaces, is contained in a cubic surface with a unique $A_1$ singularity.

\end{proof}

The following example shows that in Lemma~\ref{lem:curves-in-cubic-surfaces} the hypothesis that the curve $C$ is general is required.

\begin{example}
{\em 
In $\pp^3$, consider the cubic surfaces $F=\{x^3+y^3+z^3+w^3 = 0 \}$ and $G=\{x^3+2y^3+3z^3+4w^3 = 0\}$. The intersection of $F$ and $G$ is a smooth irreducible curve of degree $9$ and genus $10$.
Any cubic surface containing $F\cap G$ must be in the pencil generated by $F$ and $G$. In this pencil, the only singular surfaces are $G-F$, $G-2F$, $G-3F$, and $G-4F$, all of which are cones over smooth plane cubics and have non-canonical singularities.
}
\end{example}

We now turn our attention to the other two cases, namely curves that are the intersection of $3$ or $4$ cubic surfaces.

\begin{lemma}\label{lem:normal-cubic-surfaces}

Let $C\subset \pp^3$ be one of the following irreducible smooth curves:
\begin{enumerate}
\item a curve of degree $7$ and genus $5$, which is the scheme-theoretic intersection of three cubic surfaces; or
\item a curve of degree $6$ and genus $3$, which is the scheme-theoretic intersection of four cubic surfaces.
\end{enumerate}
Then, any cubic surface containing $C$ is normal.
\end{lemma}

\begin{proof}
First, assume that $S$ is reducible. Then a component of $S$ containing $C$ would have degree at most $2$. B\'ezout's Theorem and the hypothesis would exclude this. Thus, from now on we assume $S$ is irreducible.

Let $S$ be an irreducible cubic surface containing $C$. If it is not normal, then it is either a cone over a singular cubic curve or a ruled surface.

In the first case, by \cite[Proposition 2.12]{GP82}, the genus and degree of the curve satisfy one of the equations $g=1+\frac{d(d-5)}{6}$ or $g=1+\frac{d(d-5)-2}{6}$, checking for our values, neither case is possible.

In the second case, by \cite[Theorem (Segre)]{GP82} there exists some integer $k$ such that the degree and genus of $C$ satisfy the equation $2g-2=(k-1)(2d-3k)-2k.$ Checking for our values, we obtain the equations $3k^2-13k+16=0$ or $3k^2-15k+22=0$, none of which has integer solutions.

\end{proof}

In the following lemma we use local coordinates to construct singular cubic surfaces.

\begin{lemma}\label{lem:singular-cubics-2.9-2.12}
Let $C\subset \pp^3$ be one of the following irreducible smooth curves:
\begin{enumerate}
\item a curve of degree $7$ and genus $5$, which is the scheme-theoretic intersection of three cubic surfaces; or
\item a curve of degree $6$ and genus $3$, which is the scheme-theoretic intersection of four cubic surfaces.
\end{enumerate}
Then, there exists a singular cubic surface with only $A_n$ singularities that contains $C$.
\end{lemma}

\begin{proof}
By Lemma~\ref{lem:normal-cubic-surfaces} any cubic surface containing $C$ is normal. Hence if it has one $A_n$ singularity, by \cite[Section 3]{BW79}
then it must have only $A_n$ singularities. 

As $C$ is the intersection of three or four cubic surfaces, we can pick three linearly independent cubic surfaces $F_1$, $F_2$ and $F_3$ containing $C$. Let $\Lambda$ be the linear system generated by these three elements. Over a general point of $C$, the conormal section of the cubics span the two-dimensional conormal fiber. For a fixed point $p$, the space $T_p\pp^3/T_pC$ is two-dimensional. Thus, a surface in $\Lambda$ being singular at $p$ is a codimension $2$ condition, i.e., there is a unique $S\in \Lambda$ singular at $p$.

We will check the quadratic term of $S$ at a general point $p$. We may use local coordinates $u=F_1$, $v=F_2$, and $w$, so that locally $C=\{u=v=0\}$, with $p=(0,0,0)$ and $w$ a parameter along $C$. Therefore $F_3=uf(u,v,w)+vg(u,v,w).$ By computing the differentials, we find that the surface with a singularity at $(0,0,t)$ must have equation $S=f(0,0,t)F_1+g(0,0,t)F_2-F_3$. If the partial derivatives $f_w(0,0,w)$ and $g_w(0,0,w)$ are both $0$ for a general point $p$, then the singular surface would be the same for a general point. This would imply that the surface is not normal, a contradiction with Lemma~\ref{lem:normal-cubic-surfaces}.
Hence at least one of $f_w(0,0,w)$ or $g_w(0,0,w)$ is nonzero. 
Therefore the quadratic term in $S$ at $p$ is $-w(f_w(0,0,0)u+g_w(0,0,0)v)+q(u,v)$. It has a nonzero $uw$ or $vw$ term and no $w^2$ term, therefore it has rank at least two. We conclude that the singularity is of type $A_n$.
\end{proof}

\subsection{Cluster type pairs of low complexity}

In this subsection, we will lift toric pairs or pairs satisfying Theorem~\ref{thm:complexity-one-is-cluster-type}, containing the blown-up curves in $\pp^3$.

\begin{lemma}\label{lem:blowup-curve-in-plane} Let $C$ be a smooth curve in $\mathbb{P}^3$, such that $C$ is contained in an irreducible and reduced quadric surface or a hyperplane. Then the blow-up of $\mathbb{P}^3$ along $C$ is of cluster type. 
\end{lemma}

\begin{proof}
Let $S$ be the quadric surface or plane containing $C$.
Consider general planes $H_1$, $H_2$ and $H_3$. Let $D$ be defined as $S+H_1+H_2$ or $S+H_1+H_2+H_3$ if $S$ is a quadric surface or a plane, respectively. Therefore, the pair $(\pp^3,D)$ is log Calabi--Yau. If $S$ is a quadric surface, then by Theorem~\ref{thm:complexity-one-is-cluster-type} the pair $(\pp^3,D)$ is of cluster type.  Otherwise $(\pp^3,D)$ is a toric pair, hence of cluster type.

Let $\pi:X\rightarrow \pp^3$ be the blow-up of $\pp^3$ along $C$. As ${\rm mult}_C(D)=1$ and $C$ is of codimension $2$ in $\pp^3$, the morphism $\pi$ extracts only a canonical place of $(\pp^3,D)$. Therefore, by Lemma~\ref{lem:blow-up-of-cluster-type}, the variety $X$ is of cluster type.
\end{proof}

\begin{lemma}\label{lem:blowup-point-curve-p3} Let $C$ be a smooth curve in $\mathbb{P}^3$, such that $C$ is contained in an irreducible and reduced quadric surface or a hyperplane. Then the blow-up of $\mathbb{P}^3$ at a point $p$ and then along the strict transform of $C$ is of cluster type. 
\end{lemma}

\begin{proof}
Let $H_1$ and $H_2$ be two general planes containing the point $p$. Let $H$ be a general plane and let $S$ be the quadric surface or plane containing $C$. Let $D$ be defined as $S$ or $S+H$ if $S$ is a quadric surface or a plane, respectively. Therefore, the pair $(\pp^3,H_1+H_2+D)$ is log Calabi--Yau. If $S$ is a quadric surface, then by Theorem~\ref{thm:complexity-one-is-cluster-type} the pair $(\pp^3,H_1+H_2+D)$ is of cluster type.  Otherwise $(\pp^3,H_1+H_2+D)$ is a toric pair, hence of cluster type.

Let $\pi_1:X\rightarrow \pp^3$ be the blow-up of $\pp^3$ at $p$, and let $(X,B_X)$ be the log pull-back of $(\pp^3,H_1+H_2+D)$. If $p\in S$, then ${\rm mult}_p(H_1+H_2+D)=3$. In the plane case choose $H$ avoiding $p$. If $p$ is the vertex of a quadric cone $S$, take $H_1$ general through $p$ and $H_2$ avoiding $p$ instead. Adjunction gives an lc boundary and multiplicity three. In all other cases the stated general planes give an lc boundary. Since $p$ is of codimension $3$ in $\pp^3$, the morphism $\pi_1$ extracts a log canonical or canonical place of $(\pp^3,H_1+H_2+D)$, according as $p\in S$ or $p\notin S$. In either case, Lemma~\ref{lem:blow-up-of-cluster-type} implies that $(X,B_X)$ is of cluster type.

Let $\pi_2:Y\rightarrow X$ be the blow-up of $C'$ the strict transform of $C$ in $X$. As ${\rm mult}_{C'}(B_X)=1$ and $C'$ is of codimension $2$ in $X$, $\pi_2$ extracts only a canonical place of $(X,B_X)$. Therefore, by Lemma~\ref{lem:blow-up-of-cluster-type}, the variety $Y$ is of cluster type.
\end{proof}

\begin{lemma}\label{lem:3.6-3.12-3.18} Let $L$ be a line in $\mathbb{P}^3$. Let $C$ be a smooth curve  disjoint from $L$, contained in an irreducible and reduced quadric surface $Q$. Then, the blow-up of $\mathbb{P}^3$ along $L$ and $C$ is of cluster type.
\end{lemma}

\begin{proof}
Let $H_1$ be a general plane containing $L$ and let $H_2$ be a general plane. Then, $(\pp^3,H_1+H_2+Q)$ is a log Calabi--Yau pair, and by Theorem~\ref{thm:complexity-one-is-cluster-type} it is of cluster type. Let $\pi:X\rightarrow \pp^3$ be the blow-up of $\pp^3$ along $L$ and $C$. The multiplicity along $C$ is one, while that along $L$ is one or two according as $L\not\subset Q$ or $L\subset Q$. Hence $\pi$ extracts only canonical or log canonical places of $(\pp^3,H_1+H_2+Q)$. Therefore, by Lemma~\ref{lem:blow-up-of-cluster-type}, the variety $X$ is of cluster type, as wanted.
\end{proof}

\begin{proposition}\label{lem:blowup-union-three-lines} Let $X$ be the blow-up of $\pp^3$ along the disjoint union of three lines. Then, the variety $X$ is of cluster type.
\end{proposition}

\begin{proof}
Let $H_1$, $H_2$ and $H_3$ be general planes, each containing one of the three disjoint lines, and let $H_4$ be a general plane. The log Calabi--Yau pair $(\pp^3, H_1+H_2+H_3+H_4)$ is toric, therefore of cluster type. Let $\pi:X\rightarrow\pp^3$ be the blow-up of $\pp^3$ along the three disjoint lines. As $H_1+H_2+H_3+H_4$ has multiplicity $1$ at each of the three lines, the map $\pi$ only extracts canonical places of $(\pp^3, H_1+H_2+H_3+H_4)$. Therefore, by Lemma~\ref{lem:blow-up-of-cluster-type}, the variety $X$ is of cluster type.
\end{proof}

\begin{proof}[Proof of Theorem~\ref{thm:blow-up-P3}]
We combine the preceding constructions.
If $X$ is in the families \textnumero\,2.15, \textnumero\,2.19, \textnumero\,2.22, \textnumero\,2.25, \textnumero\,2.27, \textnumero\,2.28 or \textnumero\,2.30, then $X$ is the blow-up of $\pp^3$ along a curve contained in a quadric surface or a plane. Therefore, by Lemma~\ref{lem:blowup-curve-in-plane}, the variety $X$ is of cluster type.

If $X$ is in the families \textnumero\,3.11, \textnumero\,3.14, \textnumero\,3.16 or \textnumero\,3.23, then $X$ is the blow-up of $\pp^3$ at a point $p$ and then along the strict transform of a curve contained in a plane or a quadric surface. Therefore by Lemma~\ref{lem:blowup-point-curve-p3}, the variety $X$ is of cluster type.

If $X$ is in \textnumero\,3.6, \textnumero\,3.12 or \textnumero\,3.18, then the variety $X$ is the blow-up of $\pp^3$ along the disjoint union of a line and a curve contained in a quadric. Therefore, by Lemma~\ref{lem:3.6-3.12-3.18}, the variety $X$ is of cluster type.

The case of elements in the family \textnumero\,4.6 is proved in Proposition~\ref{lem:blowup-union-three-lines}.

If $X$ is in the families \textnumero\,2.9 or \textnumero\,2.12, then it is the blow-up of $\pp^3$ along a curve as in Lemma~\ref{lem:singular-cubics-2.9-2.12}. Therefore the curve is contained in a singular cubic surface with only $A_n$ singularities. By Lemma~\ref{lem:singular-cubic-surface-An} the variety $X$ is of cluster type.

If $X$ is a general element in the family \textnumero\,2.4, then it is the blow-up of $\pp^3$ in a curve as in Lemma~\ref{lem:curves-in-cubic-surfaces}. Therefore, the curve is contained in a singular cubic surface with only $A_n$ singularities. By Lemma~\ref{lem:singular-cubic-surface-An}, the variety $X$ is of cluster type.
\end{proof}

\section{Blow-ups of a three-dimensional quadric}
\label{sec:q3}

In this section, we study smooth Fano threefolds which are blow-ups of a smooth quadric threefold in $\pp^4$.
Throughout this section $\qq^3$ always refers to a smooth quadric threefold in $\pp^4$.
The main theorem is the following:

\begin{theorem}\label{thm:blow-up-Q3}
Let $X$ be a smooth Fano threefold in the families \textnumero\,2.7, \textnumero\,2.13, \textnumero\,2.17, \textnumero\,2.21, \textnumero\,2.23, \textnumero\,2.26, \textnumero\,2.29, \textnumero\,2.31, 
\textnumero\,3.10, \textnumero\,3.15, 
\textnumero\,3.19, \textnumero\,3.20, \textnumero\,4.4, or \textnumero\,5.1, then $X$ is of cluster type.
\end{theorem} 

All the cases considered in the previous theorem
are blow-ups of $\mathbb{Q}^3$.
We start by presenting a Calabi--Yau pair on $\mathbb{Q}^3$ that is of cluster type. 

\begin{lemma}\label{lem:smooth-quadric-section-in-quadric}
Let $\mathbb{Q}^3$ be a smooth quadric threefold in $\pp^4$. Let $S\in |\mathcal{O}_{\mathbb{Q}^3}(2)|$ be smooth and $H$ a general hyperplane in $\pp^4$ containing a line in $S$, then the pair $(\mathbb{Q}^3,S+H|_{\mathbb{Q}^3})$ is of cluster type.
\end{lemma}

\begin{proof}
Let $\ell$ be the line in $S$, and let $R$ be the intersection of the general hyperplane through $\ell$ with $\mathbb{Q}^3$. As $H$ is general and $\mathcal{O}_\mathbb{Q}^3(-3)$ is the canonical bundle of $\mathbb{Q}^3$, the pair $(\mathbb{Q}^3,S+R)$ is an index one Calabi--Yau pair. 
Let $(R,B_R)$ be the pair obtained by adjunction of $(\mathbb{Q}^3,S+R)$ to $R$. By construction $B_R$ contains $\ell$, and the surface $R$ is isomorphic to $ \pp^1\times \pp^1$. 
Let $h$ be the hyperplane class on $S$. The residual system $\mathcal{M}\subset |h-\ell|$ cut-out by hyperplanes through $\ell$ is basepoint-free. Since $(h-\ell)\cdot \ell =2$, Bertini's theorem show that for general $H$, $B_R=\ell+C$, where $C$ is smooth and meets $\ell$ transversally at two distinct points. By choosing the ruling so $\ell$ has bidegree $(1,0)$ on $R$, the residual curve $C$ has bidegree $(1,2)$.
Thus the pair $(R,B_R)$ has coregularity zero. Therefore, the pair $(\mathbb{Q}^3,S+R)$ also has coregularity zero.
 
Let $X$ and $Y$ be the blow-ups along $\ell$ of $\mathbb{Q}^3$ and $\pp^4$ respectively. Since $\ell$ has codimension $2$ in $\mathbb{Q}^3$ and it has multiplicity $2$ in the boundary $S+R$, the log pull-back of $(\mathbb{Q}^3,S+R)$ in $X$ is $(X,S'+R'+E)$. Here, $S'$ and $R'$ are the strict transforms of $S$ and $R$, respectively, and $E$ is the exceptional divisor.
Consider the fibration $\pi:Y \rightarrow \pp^2$, where the fibers correspond to the strict transforms of planes in $\pp^4$ containing $\ell$. The residual system $\mathcal{M}$ has self intersection $(h-\ell)^2=1$, therefore $S'$ dominates $\pp^2$. As $R$ is the intersection of a hyperplane $H$ and $\mathbb{Q}^3$, the divisor $R$ intersects a plane containing $\ell$ outside of $\ell$, exactly when the plane is contained in $H$. Therefore the image of $R'$ in $\pp^2$ is a line.
We can restrict the fibration to $\pi|_X\colon X\rightarrow\pp^2$, as $S\sim 2R$, we obtain that $2R'+2E\sim S'+E$, hence $E\sim S'-2R'$. As $S'$ intersects the general fiber of $\pi|_X$, and $R$ doesn't, the divisor $E$ also intersects the general fiber. Thus, for $\pi|_X:X\rightarrow \pp^2$, the pair $R'+S'+E$ has two horizontal components $S'$ and $E$. By Proposition~\ref{prop:fibration-to-smooth-surface}, the pair $(X,R'+S'+E)$ is of cluster type. As $E$ is a log canonical place, Lemma~\ref{lem:cluster-type-descends} implies that $(\mathbb{Q}^3,R+S)$ is of cluster type as desired
\end{proof}

The next lemma shows that some of the curves that need to be blown-up on $\mathbb{Q}^3$ to obtain smooth Fano threefolds are contained in cluster type boundaries as in Lemma~\ref{lem:smooth-quadric-section-in-quadric}.

\begin{lemma}\label{lem:curves-in-quadrics}
Let $\mathbb{Q}^3$ be a smooth quadric threefold in $\pp^4$. Let $C \subset \mathbb{Q}^3$ be one of the following smooth and irreducible curves:
\begin{enumerate}
\item the intersection of two divisors from $\mathcal{O}_{\mathbb{Q}^3}(2)$,
\item a curve of degree $6$ and genus $2$,
\item an elliptic curve of degree $5$,
\item a rational normal quartic.
\end{enumerate}
Then, there is a smooth element $B\in |\mathcal{O}_{\mathbb{Q}^3}(2)|$ that contains $C$.
\end{lemma}

\begin{proof}

We first prove that in all cases the ideal $\mathcal{I}_{C/\mathbb{Q}^3}$ is generated by elements of $\mathcal{O}_{\mathbb{Q}^3}(2)$. The first case follows from the definition of the curve. 
In the other three cases $d\geq 2g+2$, where $d$ and $g$ are the degree and genus of the curve.
Riemann-Roch gives $h^0(\mathcal{O}_C(1))=d-g+1=5$. Since the curves are non degenerate in $\mathbb{P^4}$, the embedding $C\subset \mathbb{P}^4$ is linearly normal. Therefore,~\cite{SaintDonat72} implies that $\mathcal{I}_{C/\pp^4}$ is generated by quadrics.

Let $\Lambda$ be the linear system of divisors in $\mathcal{O}_{\mathbb{Q}^3}(2)$ containing $C$. As $\Lambda$ generates $\mathcal{I}_{C/\mathbb{Q}^3}$, the base locus of $\Lambda$ is exactly $C$. Therefore, by Bertini's theorem, a general element in $\Lambda$ is smooth outside of $C$. 

Fix a point $p \in C$, consider a divisor $B\in \Lambda$, and let $R\subset \pp^4$ be a quadric representing $B$. Since $T_pC\subset T_pR$ the differential of the equation of $R$ yields a linear form $\lambda_B$ on $T_p\mathbb{Q}^3/T_pC$. As $B$ is a divisor in $\mathbb{Q}^3$, the divisor $B$ being singular at $p$ is equivalent to $T_p\mathbb{Q}^3=T_pB$ or $T_p\mathbb{Q}^3\subset T_pR$, thus equivalent to $\lambda_B=0$.  As $\Lambda$ generates $\mathcal{I}_{C/\mathbb{Q}^3}$, we have the equality $\bigcap_{B\in \Lambda}T_pB=T_pC$, thus the forms $\lambda_B$ generate the two-dimensional space $(T_p\mathbb{Q}^3/T_pC)^{\vee}$.  Hence the set of divisors in $\Lambda$ singular at $p$ is of codimension $2$. Therefore the set of all $B\in \Lambda$ singular at some point of $C$ is of codimension at least one, which is what we wanted to prove.
\end{proof}

In the following propositions, we lift the cluster type boundaries that we constructed in Theorem~\ref{thm:complexity-one-is-cluster-type}.

\begin{proposition}\label{prop:hyperplanes-in-quadric}
Let $X$ be a smooth Fano threefold which is the blow-up of $\mathbb{Q}^3$, a smooth quadric threefold in $\pp^4$, along one of the following curves:
\begin{enumerate}
\item a curve of degree $4$ that is an intersection of a surface in
$|\mathcal{O}_{\mathbb{Q}^{3}}(1)|$ and a surface in
$|\mathcal{O}_{\mathbb{Q}^{3}}(2)|$,
\item a twisted cubic,
\item a conic,
\item a line.
\end{enumerate}
Then, the variety $X$ is of cluster type.
\end{proposition}

\begin{proof}
Each of those curves is contained in a hyperplane in $\pp^4$. Let $H$ be the hyperplane containing the curve $C$ and let $R$ be the intersection of $H$ and $\mathbb{Q}^3$. Therefore $R$ is a smooth quadric surface or a quadric cone, hence it has at worst canonical singularities. Thus, by picking general hyperplanes $H_1$ and $H_2$, the pair $(\mathbb{Q}^3,R+H_1\vert_{\mathbb{Q}^3}+H_2\vert_{\mathbb{Q}^3})$ is a log Calabi--Yau pair of dimension $3$. Moreover, we have that $\rho (\mathbb{Q}^3)=1$ and the boundary $R+H_1\vert_{\mathbb{Q}^3}+H_2\vert_{\mathbb{Q}^3}$ has three components. Therefore, by Theorem~\ref{thm:complexity-one-is-cluster-type}, the pair $(\mathbb{Q}^3,R+H_1\vert_{\mathbb{Q}^3}+H_2\vert_{\mathbb{Q}^3})$ is of cluster type. Let $\pi\colon X \rightarrow \mathbb{Q}^3$ be the blow-up of $\mathbb{Q}^3$ along $C$ and let  $\pi^*(K_{\mathbb{Q}^3}+R+H_1\vert_{\mathbb{Q}^3}+H_2\vert_{\mathbb{Q}^3})=K_X+B$. As $C$ has codimension $2$ in $Q$ and is contained in a single component of $R+H_1\vert_ {\mathbb{Q}^3}+H_2\vert_ {\mathbb{Q}^3}$, the morphism $\pi$ only extracts a canonical place of $(\mathbb{Q}^3,R+H_1\vert_{\mathbb{Q}^3}+H_2\vert_{\mathbb{Q}^3})$. Therefore, by Lemma~\ref{lem:blow-up-of-cluster-type}, the pair $(X,B)$ is of cluster type.
\end{proof}

\begin{proposition}\label{prop:union-lines-in-quadric}
Let $X$ be a smooth Fano threefold which is the blow-up of $\mathbb{Q}^3$, a smooth quadric threefold in $\pp^4$, along one of the following:
\begin{enumerate}
\item the disjoint union of two lines,
\item the disjoint union of two conics,
\item the disjoint union of a line and a conic.
\end{enumerate}
Then, the variety $X$ is of cluster type.
\end{proposition}

\begin{proof}
In any case $X$ is the blow-up of $\mathbb{Q}^3$ in the disjoint union of two planar curves $C_1$ and $C_2$. Let $H_1$ and $H_2$ be general hyperplanes containing $C_1$ and $C_2$ respectively, and let $H_3$ be a general hyperplane. The pair $(\mathbb{Q}^3,H_1\vert_{\mathbb{Q}^3}+H_2\vert_{\mathbb{Q}^3}+H_3\vert_{\mathbb{Q}^3})$ is log Calabi-Yau.
By Theorem~\ref{thm:complexity-one-is-cluster-type}
the previous pair is of cluster type. Let $\pi\colon X \rightarrow \mathbb{Q}^3$ be the blow-up of $\mathbb{Q}^3$ along the disjoint curves, and let $\pi^*(K_Q+H_1\vert_{\mathbb{Q}^3}+H_2\vert_{\mathbb{Q}^3}+H_3\vert_{\mathbb{Q}^3})=K_X+B$.  As each of the curves is contained in a single component of the boundary of $H_1\vert_{\mathbb{Q}^3}+H_2\vert_{\mathbb{Q}^3}+H_3\vert_{\mathbb{Q}^3}$, the morphism $\pi$ only extracts canonical places of $(\mathbb{Q}^3,H_1\vert_{\mathbb{Q}^3}+H_2\vert_{\mathbb{Q}^3}+H_3\vert_{\mathbb{Q}^3})$. Hence, by Lemma~\ref{lem:blow-up-of-cluster-type}, the pair $(X,B)$ is of cluster type.
\end{proof}

\begin{proposition}\label{prop:3.19-4.4}
Let $\mathbb{Q}^3$ be a smooth quadric threefold in $\pp^4$. Let $X$ be the blow-up of $\mathbb{Q}^3$ at two points $p_1,p_2$ whose joining line is not contained in $\mathbb Q^3$. Let $Y$ be the blow-up of $X$ along the proper transform of a conic in $\mathbb{Q}^3$ that passes through $p_1$ and $p_2$. Then, the varieties $X$ and $Y$ are of cluster type.
\end{proposition}

\begin{proof}
Let $C$ be the conic in $\mathbb{Q}^3$. Let $H_1$ and $H_2$ be general hyperplanes containing $p_1$ and $p_2$ respectively, and let $H_3$ be a general hyperplane in $\pp^4$ containing $C$. By Theorem~\ref{thm:complexity-one-is-cluster-type}, the pair $(\mathbb{Q}^3,H_1\vert_{\mathbb{Q}^3}+H_2\vert_{\mathbb{Q}^3}+H_3\vert_{\mathbb{Q}^3})$ is of cluster type. Let $\pi_1\colon X\rightarrow \mathbb{Q}^3$ be the blow-up at $p_1$ and $p_2$. Let $B_1$, $B_2$ and $B_3$ be the strict transforms in $X$ of $H_1\vert_{\mathbb{Q}^3}$, $H_2\vert_{\mathbb{Q}^3}$ and $H_3\vert_{\mathbb{Q}^3}$, respectively. As each of the points is contained in exactly two different components of the boundary, the morphism $\pi_1$ extracts only canonical places of $(\mathbb{Q}^3,H_1\vert_{\mathbb{Q}^3}+H_2\vert_{\mathbb{Q}^3}+H_3\vert_{\mathbb{Q}^3})$, i.e. $\pi^*(K_{\mathbb{Q}^3}+H_1\vert_{\mathbb{Q}^3}+H_2\vert_{\mathbb{Q}^3}+H_3\vert_{\mathbb{Q}^3})=K_X+B_1+B_2+B_3$. Thus, by Lemma~\ref{lem:blow-up-of-cluster-type}, the pair $(X,B_1+B_2+B_3)$ is of cluster type.
Let $\pi_2\colon Y\rightarrow X$ be the blow-up along the strict transform of $C$ in $X$. As $C$ is contained in $B_3$ and no other component of the boundary, the morphism $\pi_2$ extracts only canonical places of $(X,B_1+B_2+B_3)$. Therefore, by Lemma~\ref{lem:blow-up-of-cluster-type}, the variety $Y$ is of cluster type as well.
\end{proof}

\begin{proposition}\label{prop:5.1}
Let $\mathbb{Q}^3$ be a smooth quadric threefold in $\pp^4$. Let $\pi\colon X\rightarrow \mathbb{Q}^3$ be the blow-up of $\mathbb{Q}^3$ along a plane conic. Let $Y$ be the blow-up of $X$ along the disjoint union of three exceptional lines of the blow-up $\pi\colon X\rightarrow\mathbb Q^3$. Then, the variety $Y$ is of cluster type.
\end{proposition}

\begin{proof}
Let $C$ be the plane conic. Let $H_1$ and $H_2$ be two general hyperplanes containing $C$, and let $H_3$ be a general hyperplane. Let $R_1$, $R_2$ and $R_3$ be the intersections $H_1\cap \mathbb{Q}^3$, $H_2 \cap \mathbb{Q}^3$ and $H_3 \cap \mathbb{Q}^3$, respectively. The pair $(\mathbb{Q}^3,R_1+R_2+R_3)$ is Calabi--Yau, and by Theorem~\ref{thm:complexity-one-is-cluster-type} is of cluster type. As $C$ has codimension $2$ in $\mathbb{Q}^3$ and it is contained in exactly two components of the boundary, the morphism $\pi$ extracts a log canonical place of $(\mathbb{Q}^3,R_1+R_2+R_3)$. Thus, we have $\pi^*(K_{\mathbb{Q}^3}+R_1+R_2+R_3)=K_X+R_1'+R_2'+R_3'+E$, where $R_1'$, $R_2'$ and $R_3'$ are the strict transforms in $X$ of $R_1$, $R_2$ and $R_3$ respectively, and $E$ is the exceptional divisor. By Lemma~\ref{lem:blow-up-of-cluster-type}, the pair $(X,R_1'+R_2'+R_3'+E)$ is of cluster type.

Let $\pi_1\colon Y\rightarrow X$ be the blow-up along the disjoint union of three exceptional lines. As each of these lines is contained in $E$ and not in any $R_i'$, the morphism $\pi_1$ extracts only canonical places of $(X,R_1'+R_2'+R_3'+E)$. Therefore, by Lemma~\ref{lem:blow-up-of-cluster-type}, the variety $Y$ is of cluster type.
\end{proof}

We can now organize all the cases we have covered to prove the main theorem of the section.

\begin{proof}[Proof of Theorem~\ref{thm:blow-up-Q3}]
Let $X$ be in the families \textnumero\,2.7, \textnumero\,2.13, \textnumero\,2.17, or \textnumero\,2.21. Then by Lemma~\ref{lem:curves-in-quadrics} the variety $X$ is the blow-up of a curve contained in a smooth divisor in $\mathcal{O}_{\mathbb{Q}^3}(2)$, which is of cluster type by Lemma~\ref{lem:smooth-quadric-section-in-quadric}.

In the case that $X$ is contained in the families \textnumero\,2.23, \textnumero\,2.26, \textnumero\,2.29 or \textnumero\,2.31, we conclude that $X$ is of cluster type by Proposition~\ref{prop:hyperplanes-in-quadric}. 
In the case that $X$ is contained in the families \textnumero\,3.10, \textnumero\,3.15, or \textnumero\,3.20, then $X$ is of cluster type by Proposition~\ref{prop:union-lines-in-quadric}. If $X$ is contained in the families \textnumero\,3.19 or \textnumero\,4.4, then $X$ is of cluster type by Proposition~\ref{prop:3.19-4.4}. Finally, if $X$ is in the family \textnumero\,5.1, then by  Proposition~\ref{prop:5.1}, we conclude that the variety $X$ is of cluster type.
\end{proof}

\section{Smooth Fano threefolds in the main series}
\label{sec:main-series}
In this section, we study the smooth Fano threefolds in the main series; this means smooth Fano threefolds of Picard rank one and index one. 
These are smooth prime Fano threefolds of genus at most 12. For their homogeneous models see~\cite{Mukai88,Mukai89}. We prove that every smooth prime Fano threefold of genus $9$, $10$, or $12$ are cluster type. 

\subsection{Mori-Mukai family \textnumero\,1.8 and \textnumero\,1.9}

In this subsection we prove that any prime smooth Fano threefold of 
	genus 9 and 10 is of cluster type. We use the ideas of \cite{KPZ14} to 
	describe the geometry of these threefolds. 
	Let $X:= X_{2g-2}$ be a smooth prime Fano threefold 
	of genus $g \in \{9,10\}$. The anticanonical embedding induces 
	$X \subseteq \pp^{g+1}$. Let $\ell \subseteq X$ be a line,
	and let $H\sim -K_X$ be the hyperplane class. Iskovskikh in \cite{Isk89} shows that if the map $X \dashrightarrow W \subseteq \pp^{g-6}$ is the one induced by the linear system $|H - 2\ell|$, then there exists a Sarkisov 
	link
	\[
		\xymatrix{
			Y:= {\rm Bl}_\ell X \ar[d]_{\varphi} \ar@{-->}[rr]^{\alpha} & & Y' := {\rm Bl}_\Gamma W \ar[d]^{\psi}\\
			X \ar@{-->}[rr] & & W
		},
	\]
	where $\varphi$ is the blow-up of $\ell$, $\alpha$ is a flop, and $\psi$ is the blow-up of a smooth irreducible curve $\Gamma$. Let $E = {\rm Exc}(\varphi),$ $F' = \alpha_* E$, $F_\ell = \psi(F')$, and $D = {\rm Exc}(\psi)$.
	\begin{itemize}
		\item If $g=9$, then $W = \pp^3$, $F_\ell$ is an irreducible cubic surface, and $\Gamma$ is a non-hyperelliptic curve of genus $3$ and degree $7$.
		\item If $g = 10$, then $W = Q^3 \subseteq \pp^4$ is a smooth quadric, $F_\ell$ is an irreducible quadric section of $Q^3$, and $\Gamma$ is a curve of genus $2$ and degree $7$.
	\end{itemize}
	
	The normal bundle of the line $\ell$ is one of the following:
	\[
		N_{\ell/X} \simeq \begin{cases}
			\oo_{\pp^1} \oplus \oo_{\pp^1}(-1)\\
			\oo_{\pp^1}(1) \oplus \oo_{\pp^1}(-2) 
		\end{cases}
	\]
	If it is the first case, we will say that $\ell$ is an \emph{ordinary line}. If it is the latter, we say that $\ell$ is a \emph{special line}. The surface $F_\ell$ is normal if and only if the line is ordinary \cite[Theorem 2.13]{KPZ14}. If so, $F_\ell$ has at worst du Val singularities. If $F_\ell$ is not normal, then its singular locus is a double line, and its normalization is a smooth scroll of minimal degree \cite[Proposition 2.2]{KPZ14}. Virin, in \cite{Vir26} proves that every smooth prime Fano threefold of genus 9 or 10 contains a cylinder. To do so, they prove the following result.
	
	\begin{proposition}\label{prop:sing-line}
		Let $X$ be a smooth prime Fano threefold of genus 9 or 10. Then $X$ contains a line $\ell$ such that the surface $F_\ell$
		is singular.
	\end{proposition}
	
	\begin{proof}
		This follows from Corollary 4.4, Proposition 5.2 and Proposition 5.4 of \cite{Vir26}.
    \end{proof}

    To construct a boundary of cluster type on $X$, we need a more precise description of the geometry of $F_\ell$.
	
	\begin{proposition}\label{prop:mult2point}
		Let $X$ be a smooth prime Fano threefold of genus 9 or 10. Then there exists a line $\ell \subseteq X$ and a point $p\in F_\ell$ such that $\mult_p(F_\ell) = 2$. Furthermore, if 
		\[
			f = f_2 + \mathrm{higher\ order\ terms}
		\] 
		is the local equation of $F_\ell$ at $p$, then the rank of $f_2$ is $\geq 2$.
	\end{proposition}

    We prove Proposition \ref{prop:mult2point} by studying two cases: when $X$ contains a special line, and when all the lines in $X$ are ordinary.
	
	\begin{lemma}\label{lem:mult2-special-line}
		Let $X$ be a smooth prime Fano threefold of genus 9 or 10, and let $\ell\subseteq X$ be a special line. Then a general point of the singular line $\Lambda \subseteq F_\ell$ satisfies Proposition \ref{prop:mult2point}.
	\end{lemma}
	
	\begin{proof}
        Let $\nu\colon F^\nu \to F_\ell$ be the normalization of $F_\ell$, and let $M := \nu^* \oo_{F_\ell} (1).$ The normalization $F^\nu$ is one of the smooth scrolls classified~\cite[Proposition 2.2(2)]{KPZ14}. We have that $F^\nu$ is isomorphic to a Hirzebruch surface $\mathbb{F}_n$, and $M \sim \Sigma + kl$, where $\Sigma^2 = -n$, $l$ is a ruling, and $k$ and $n$ satisfy
        \[
	       k > n,\quad 2k - n = g-6.
        \]
        Hence
        \begin{itemize}
	       \item if $g=9$, then $F^\nu \cong \mathbb{F}_1$ and $M \sim \Sigma + 2l$, and
	       \item if $g=10$, then $F^\nu \cong \mathbb{F}_0$ and $M \sim \Sigma + 2l$, or $F^\nu \cong \mathbb{F}_2$ and $M\sim \Sigma + 3l.$
        \end{itemize}
        In both cases, $\nu^{-1}(\Lambda)$ is either a smooth conic or 
        a reduced degenerated conic.
        Let $C$ be a general hyperplane section of $F_\ell$.
        By Bertini's theorem and generic smoothness, the preimage of $C$ is a
        smooth irreducible curve $C' \in |M|$. Notice that
        \[
	        (K_{\mathbb{F}_n} + M)\cdot M = n -2k - 2 + 2k - n = -2,
        \]
        so adjunction gives $g(C') = 0.$

        If $g=9$, $C$ corresponds to a plane cubic, and if $g=10$, $C$ is the complete intersection of two quadrics in $\pp^3.$ In each case, $p_a(C) = 1$, and it is smooth away from its unique intersection point $p$ with $\Lambda$. For a general point on $\Lambda$, its preimage consists of two points, so $C$ has two branches at $p$. For $C$ we have
        \[
	       \delta(C) := p_a(C) - g(C') = 1.	
        \]
        Let $B_1,B_2$ be the two branches of $C$ at $p$. Then $\delta_p(C) = \delta_p(B_1) + \delta_p(B_2) + (B_1\cdot B_2)_p$. As $(B_1\cdot B_2)_p \geq 1$, then it is equal to 1, and $\delta_p(B_1) = \delta_p(B_2) = 0$.
        This implies that $C$ has an ordinary node at $p$.

        Finally, let $(x,y,z)$ be local coordinates at $p$ with $\Lambda = \{x=y=0\}$ and $H=\{z=0\}$. Then, because $F_\ell$ is singular along $\Lambda$, its local equation belongs to $(x,y)^2$, so its quadratic part $q$ is a form in $x,y$ only. 
        As $C=F_\ell\cap H$ is cut out in $H$ by $q+ \text{higher order terms}$, its tangent cone is $\{q=0\}$.
        Since $C$ has two transverse smooth branches at $p$, $q$ is nondegenerate.
        Hence $\mult_p F_\ell=2$ and $\operatorname{rank} q=2$. 
	\end{proof}
	
	\begin{lemma}\label{lem:mult2-ordinary-line}
		Let $X$ be a smooth prime Fano threefold of genus 9 or 10, and let $\ell \subseteq X$ be an ordinary line such that every line intersecting $\ell$ is ordinary. Then $F_\ell$ is smooth or has $A_n$ singularities. 
	\end{lemma}

    \begin{proof}
        Because $\ell$ is ordinary, the exceptional divisor $E \cong \mathbb{F}_1$. Let $\Sigma$ be the section with negative self-intersection, and $l$ the class of the ruling. We have $\varphi^* (-K_X)|_E = l$ and $E|_E = -\Sigma - l$. Hence
        \[
	       -K_Y|_E = \Sigma + 2l =: M.
        \]
        Let $a\Sigma + bl$ be an effective curve on $E$. Then $M(a\Sigma + bl) =a + b$. The effective cone of $E$ is generated by $\Sigma$ and $l$, so $a,b\geq 0$. This implies that the only irreducible curves of degree one with respect to $M$ are $\Sigma$ and the ruling fibers.

        We study how $E$ transforms under the flop $\alpha$. Let 
        \[
            \xymatrix{
	           Y \ar[rd]_{\pi} \ar@{-->}[rr]^{\alpha} & & Y' \ar[ld]^{\pi'}\\
	               & \bar{Y}
            }
        \]
        be the anticanonical contraction common to $Y$ and $Y'$ \cite[Corollary 2.8]{KPZ14}. 
        It contracts exactly the disjoint flopping curves $C_i$, which correspond to the disjoint strict transforms of the lines $\ell_i \neq \ell$ that intersect $\ell$ \cite[Proposition 2.12]{KPZ14}. 
        Each $C_i$ intersects $E$ transversally in one point $p_i$. Let $\bar{E} = \pi(E)$, and recall that $\alpha_* E = F'$.
        
        Near $C_i$ the flop is analytically a pagoda of width $n=n_i \geq 1$ in the sense of Reid \cite[\S 5]{Rei83}. 
        That is, $\bar{Y}=V_n=\{xz=y(y+t^n)\}\subset\mathbb{A}^4$,
        and $Y$, $Y'$ are the small resolutions ${\rm Bl}_{(x,y)}V_n$ and ${\rm Bl}_{(x,y+t^n)}V_n$, respectively.
        In the chart $y=xv$ of the first, with coordinates $(x,v,t)$, the contraction is
        $(x,v,t)\mapsto(x,xv,v(xv+t^n),t)$ and $C_i=\{x=t=0\}$. The transition to the chart $u=v^{-1}$ is
        $z=v^2x+vt^n$, $t=t$, so $N_{C_i/Y} \cong \oo(-1)^{\oplus2}$ for $n=1$ and $\oo \oplus \oo(-2)$ for $n \geq 2$,
        where the summand $\oo$ corresponds to the $t$-direction.
        
        The surface $E$ is smooth and transverse to $C_i$, so at $p_i$ it has local equation
        $E=\{v=h(x,t)\}$ with coordinates $(x,t)$. The map $\pi|_E$ is
        \[
            (x,t) \mapsto (x,xh,h(xh+t^n),t),
        \]so, near $\pi(p_i)$, $\bar{E}$ is smooth and isomorphic to $E$. Away from the $C_i$, the contraction $\pi$ is an isomorphism, and because each $C_i$ intersects $E$ at one point, $E\to \bar{E}$ is an isomorphism. So we identify $E$ with $\bar{E}$.
        Therefore, the map $\pi'$ restricts to a map $\eta\colon F' \to E,$ which is an isomorphism away from the $p_i$. Because $M=-K_Y|_E$, and both $-K_Y|_E$ and $-K_{Y'}|_{F'}$ are pull-backs of the same divisor on $\bar{E},$ $-K_{Y'}|_{F'} = \eta^* M.$
        
        The strict transform of $\bar{E}$ in 
        ${\rm Bl}_{(x,y+t^n)}V_n$,
        is the blow-up of $\bar{E}$ along $(x,y+t^n)\cdot\oo_{\bar E}$, and
        $(x,y+t^n)\cdot \oo_{\bar E}=(x,xh+t^n)=(x,t^n)$. Thus, locally at $p_i$
        \[
            \eta\colon {\rm Bl}_{(x,t^n)} \mathbb{A}^2 \to \mathbb{A}^2.
        \]

        The blow-up ${\rm Bl}_{(x,t^{n})}\mathbb{A}^2$ has two charts given by $\{xu=t^n\}$ and $\{x=vt^n\}$. The second chart is smooth, and the first has an $A_{n-1}$ point at the origin for $n\ge2$. 
        We obtain its minimal resolution by performing $n$ blow-ups.
        We first blow up the origin, so $x=tx_1$ turns $(x,t^n)$ into $t\,(x_1,t^{n-1})$.
        Then we repeat this procedure at the point where the strict transform of the curve $\{x=0\}$ meets the new exceptional curve.
        After $n$ steps we resolve the singularity, and we obtain a chain of exceptional curves 
        \[
            E_1-\cdots-E_{n-1}-E_n,
        \] 
        where $E_1,\ldots,E_{n-1}$ are $(-2)$-curves that get contracted to the $A_{n-1}$ point and $E_n$ is a $(-1)$-curve that maps onto the exceptional curve of $\eta$.

        Therefore, $F'$ is normal, with $A_{n_i-1}$ singularities over each $p_i$, and smooth elsewhere.
        Let $\rho\colon S \to F'$ be its minimal resolution. Then $\tau=\eta\circ\rho\colon S\to E$ is the composition of $n_i$ blow-ups over each $p_i$. The centers of these blow-ups are the points infinitely near to $p_i$ determined by the subscheme $V(x,t^n)\subset E$, which lies on the curve $\{x=0\}$. For $n\geq 2$, this subscheme has tangent
        direction $\langle\partial_t\rangle\subset T_{p_i} E\cong N_{C_i/Y}|_{p_i}$.
        This direction corresponds to the value at $p_i$ of the global section of
        the summand $\oo \subset N_{C_i/Y}$. To understand how each sequence of blow-ups behaves with respect to the ruling $E \to \ell$, we compare this direction with the ruling using the differential of $\varphi$.

        Fix $i$ with $n_i \geq 2$. Let $q_i=\ell\cap\ell_i = \varphi(p_i)$, and let $f_i\subset E$ be the fibre of $E\to\ell$ over $q_i$.
        Since $\ell$ and $\ell_i$ meet transversally at $q_i$, we can choose local coordinates $(x,y,z)$ on $X$ at $q_i$ with
        $\ell=\{x=y=0\}$ and $\ell_i=\{y=z=0\}$. In the chart $y=xw$ of $Y$, with coordinates $(x,w,z)$, we have $E=\{x=0\}$,
        $C_i=\{w=z=0\}$ and $p_i=(0,0,0)$. The projection $E\to\ell$ is given by $(w,z)\mapsto z$, so
        $f_i=\{x=z=0\}$ has tangent direction $\langle\partial_w\rangle$ at $p_i$. Recall that $E$ is transverse to $C_i$, so
        $T_{p_i}E\cong N_{C_i/Y}|_{p_i}=\langle\partial_w,\partial_z\rangle$.

        The restriction $\varphi|_{C_i}\colon C_i\to\ell_i$ is an isomorphism, and the differential of $\varphi$ induces a morphism of vector bundles
        \[
	           d\varphi\colon N_{C_i/Y}\longrightarrow N_{\ell_i/X}.
        \]
        Since $\varphi(x,w,z)=(x,xw,z)$, in the bases $\partial_w,\partial_z$ and $\partial_y,\partial_z$ it is given by
        \[
            \partial_w\longmapsto x\partial_y,\quad\partial_z\longmapsto\partial_z .
        \]
        It is an isomorphism away from $p_i$, so it is injective as a morphism of sheaves, and its kernel at $p_i$ is the direction of the ruling.

        Let $s$ be a nonzero global section of $N_{C_i/Y} \cong \oo \oplus \oo(-2)$.
        By injectivity, its image is a nonzero global section of $N_{\ell_i/X}$. 
        Since $\ell_i$ is ordinary, $N_{\ell_i/X} \cong \oo \oplus \oo(-1)$, and $d\varphi(s)$ is a nonzero constant section of the $\oo$ summand. This implies that $d\varphi(s)(p_i) \neq 0,$ so $s(p_i)\not \in \langle \partial_w \rangle.$ Recall that $s(p_i)$ corresponds to a tangent direction at $p_i$ of the curve defining the sequence of blow-ups.
        Therefore, this curve intersects transversally the ruling $f_i$ at $p_i$. This implies that after the first blow-up of the sequence, the following centers do not lie on $f_i$.

        Let $\psi\colon Y' \to W$ be the blow-up of $\Gamma$, with exceptional divisor $D$. Let $\psi_F = \psi|_{F'}\colon F' \to F.$ Because $\psi^*F = F' + D$ and $K_{Y'} = \psi^* K_W + D,$ adjunction implies that $K_{F'} = \psi_F^* K_F.$ The singularities of $F'$ are du Val, so if $\mu = \psi_F \circ \rho\colon S \to F$, 
        \[
            K_S = \mu^* K_F.
        \]
        Hence $\mu$ is the minimal resolution of $F$. We need to understand the exceptional curves of $\mu$ that are strict transforms of curves contracted by $\psi_F$. Let $R$ be such a curve. Then $R$ is the whole fiber $\psi^{-1}(p) \cong \pp^1$ for a point $p\in \Gamma$, and
        \[
            -K_{Y'}\cdot R = -(\psi^* K_W + D)\cdot R = -D\cdot R = 1.
        \]
        Because $-K_{Y'}|_{F'} = \eta^* M,$ $M\cdot \eta_*(R) = 1$. Then $\eta(R)$ has degree one with respect to $M$, so it is either $\Sigma$ or $l$, and $R$ is its strict transform on $F'$.  Let $R'$ be the strict transform of $R$ on $S$. By adjunction, $R'^2 = -2$. The composition $\eta \circ \rho$ is a composition of blow-ups at points, so the self-intersection of $\eta(R)$ drops by one for each center lying on its successive strict transform. 
        
        If $\eta(R) = \Sigma$, then the self-intersection drops from $-1$ on $E$ to $-2$ on $S$. Then only one center lies on the strict transforms of $\eta(R)$, which means it must be an initial center.
It meets an endpoint of the contracted chain (the $(-1)$-curve is not part of the exceptional graph of $\mu$). Therefore, $R'$ intersects transversally $E_1$ of that chain of exceptional curves.

        If $\eta(R)$ is a ruling fiber, the discussion above implies that its strict transform does not contain infinitely near centers. Then, because the self-intersection drops from $0$ to $-2$, $R$ contains two distinct initial centers, so the resulting curve configuration is a chain of $(-2)$-curves.
Distinct additional curves are disjoint fibers of $\psi_F$ and cannot meet the same resolution chain, since their images on $F^{\prime}$ would then meet. Thus no branching occurs.

        This implies that every connected configuration of exceptional curves of $\mu \colon S\to F$ is a chains of $(-2)$-curves, which implies that $F$ only has $A_n$-type singularities.
        
    \end{proof}
	
	\begin{proof}[Proof of Proposition \ref{prop:mult2point}]
		If $X$ contains a special line $\ell$, then the proposition follows from Lemma \ref{lem:mult2-special-line}. 
		If every line of $X$ is ordinary, we may use Proposition \ref{prop:sing-line} 
		to find a line $\ell$ such that $F_\ell$ is singular. 
		By Lemma \ref{lem:mult2-ordinary-line}, $F_\ell$ has $A_n$ singularities, 
		so the local equation of $F_\ell$ at any of these points is
		\[
		xy+z^{n+1} = 0.
		\]
		If $n=1$, the quadratic part has rank three, and if $n\geq 2$, the rank of the quadratic part is two. 
	\end{proof}
	
	\begin{proposition}\label{prop:1-8}
		Let $X$ be a prime smooth Fano threefold of genus $9$.
		Then, the variety $X$ is of cluster type. 
	\end{proposition}
	
	\begin{proof}
		Let $\ell \subseteq X$ and $p\in F_\ell \subseteq \pp^3$ as in Proposition~\ref{prop:mult2point}, where $F_\ell$ is a cubic surface. 
        Choose homogeneous coordinates $x_0,x_1,x_2,x_3$ of $\pp^3$
        such that $p = [1:0:0:0]$ and $F_\ell = \{f = 0\},$ with
		\[
			f = x_0f_2(x_1,x_2,x_3) + f_3(x_1,x_2,x_3).
		\]
		Let $H$ be a general plane through $p$ such that locally $H|_{F_\ell}$ intersects $\{f_2 = 0\}$ transversally at two points. 
		Consider $Z:= {\rm Bl}_p \pp^3 \xrightarrow{\pi} \pp^3$ the blow-up of $p$, with exceptional divisor $G$. Then
		\[
			\pi^*(K_{\pp^3} + H + F_{\ell}) = K_Z + H_Z + F_Z + G.
		\]
        Furthermore, the transversality condition implies that the divisor 
		$H_Z + F_Z + G$ is simple normal crossing in a neighborhood of $ H_Z \cap F_Z \cap G$, and $ H_Z \cap F_Z \cap G$ consists of two points. 
        
        We need to show that the pair $(\pp^3,H+F_\ell)$ is log canonical. By adjunction, 
        \[
            (K_Z + H_Z + F_Z + G)|_G = K_G + L_G + C_G,
        \]
        where $L_G$ is a line and $C_G = \{f_2 = 0\}$ is a reduced conic of rank $\geq 2$. 
        The line $L_G$ intersects $C_G$ transversally at two points, and if the rank of $C_G$ is two,
        the vertex is not on $L_G$. Then $(G,L_G + C_G)$ is log canonical, so by inversion of adjunction \cite{Kaw07},
        $(Z,H_Z + F_Z + G)$ is log canonical near $G$. Hence, $(\pp^3,H+F_\ell)$ is log canonical near $p.$
 		Away from $p$, we need $(\pp^3,F_\ell)$ to be log canonical, and $F_\ell + H$ to be simple normal crossing along $H \setminus \{p\}$. If $F_\ell$ is normal, then it has du Val singularities, so it is klt, and by inversion of adjunction, $(\pp^3,F_\ell)$ is log canonical. If $F_\ell$ is nonnormal, then consider its normalization $\nu\colon F^\nu \cong \mathbb{F}_1 \to F_\ell$, with singular line $\Lambda$. The preimage of $\Lambda$ is reduced, with at most one node, and linearly equivalent to $\Sigma + l$ \cite[Proposition 2.2(2)]{KPZ14}. Hence $F_\ell$ is semi-log canonical, so $(\pp^3,F_\ell)$ is log canonical. A general $H$ avoids the other isolated singularities and the exceptional points of the double locus, and meets the latter transversely. Thus $H\cap F_\ell$ is nodal at $p$ by the rank condition and nodal or smooth elsewhere by Bertini. Adjunction to $H$ proves the lc condition near $H$. Therefore, $(\pp^3,H+F_\ell)$ is a log canonical pair. We have $H+F_\ell \sim -K_{\pp^3}$, so the pair is Calabi--Yau.
        
		Furthermore, the pair $(Z,H_Z + F_Z + G)$ has a 0-dimensional simple normal crossing stratum. Then, any dlt modification over it will have a 0-dimensional stratum. Thus, $(\pp^3,F_\ell + H)$ has coregularity zero.
        Because $\mult_p(F_\ell + H) = 3$, 
		  Proposition \ref{prop:criterion-sing-p3} implies that $(\pp^3,F_\ell + H)$ is of cluster type.

        The curve $\Gamma$ is smooth, irreducible, of genus 3, and degree $7$. Therefore, it is not a planar curve, so $\Gamma\not\subseteq H$. The multiplicity of $F_\ell$ along $\Gamma$ is one. Indeed, if $F_\ell$ is normal, its singular locus is finitely many points, and if $F_\ell$ is nonnormal, the singular locus is the line $\Lambda.$ Let $H'$ be the strict transform of $H$ in $Y'$. Then
        \[
            \psi^*(K_{\pp^3} + H + F_\ell) = K_{Y'} + H' + F'.
        \]
        The divisor $D$ is a canonical place of $(\pp^3,F_\ell + H)$, so Lemma \ref{lem:blow-up-of-cluster-type} implies that $(Y', H' + F')$ is of cluster type. Let $H_Y = \alpha^{-1}_* H'$. By definition of cluster type, $(Y,E + H_Y)$ is of cluster type. Finally, let $H_X = \varphi_* H_Y$. As $K_Y = \varphi^*(K_X) + E$ and $K_Y + E + H_Y \sim 0$, $H_Y \sim \varphi^*(-K_X) -2E$. On the other hand, we have $\varphi^*(H_X) = H_Y + mE$, which implies $m=2$ and $H_X\in |-K_X|$. Hence
        \[
            \varphi^*(K_X + H_X) = K_Y + H_Y + E.
        \]
        Then $E$ is a log canonical place of $(X,H_X)$, so Lemma \ref{lem:cluster-type-descends} implies that $(X,H_X)$ is of cluster type.
	\end{proof}
	
	\begin{proposition}\label{prop:1-9}
		Let $X$ be a prime smooth Fano threefold of genus $10$.
		Then, the variety $X$ is of cluster type. 
	\end{proposition}
	
	\begin{proof}
	Let $\ell \subseteq X$ and $p\in F_\ell \subseteq Q \subseteq \pp^4$ as in Proposition 
	\ref{prop:mult2point}. Here, $F_\ell \in |\oo_Q(2)|$.
	Choose homogeneous coordinates $u,x_0,x_1,x_2,x_3$ of $\pp^4$ such that $p = [1:0:0:0:0]$ and
	\[
		Q = \{ux_0 + f(x_1,x_2,x_3) = 0\},
	\]
	where $f(x_1,x_2,x_3)$ is a nondegenerate quadric.
	The defining equation of $F_\ell$ in $Q$ does not have $u^2$ as a term, and we may assume, 
	using that $F_\ell$ is singular at $p$ and after subtracting a multiple of the equation of $Q$, that
	\[
		F_\ell = Q \cap \{g(x_0,x_1,x_2,x_3) = 0\}.
	\]
	Consider the projection from $p$ to $\pp^3$
	\[
		[u:x_0:x_1:x_2:x_3] \mapsto [x_0:x_1:x_2:x_3],
	\]
	and let $H_0 = \{x_0 = 0\}$, $C = \{x_0 = f(x_1,x_2,x_3) = 0\}$ and $S = \{g(x_0,x_1,x_2,x_3) = 0\}.$ 
	The conic $C$ is not contained in $S$. Indeed, if it were,
	\[
		g(x_0,x_1,x_2,x_3) = x_0l(x_0,x_1,x_2,x_3) + \lambda f(x_1,x_2,x_3),
	\]
	with $l$ a linear form. On $Q$ this would imply $g(x_0,x_1,x_2,x_3) = x_0(l(x_0,x_1,x_2,x_3) -\lambda u)$,
	which contradicts that $F_\ell$ is reduced and irreducible.
	At $u=1$, the equation of $Q$ implies that $x_0$ is a quadratic function of $x_1$, $x_2$ and $x_3$, 
	so the tangent cone of $F_\ell$ at $p$ is $\{g(0,x_1,x_2,x_3) = 0\}$. 
	By Proposition \ref{prop:mult2point}, the intersection $S\cap H_0$ is a reduced conic of rank at least two. 
	Choose a plane $H_1 \neq H_0$ such that the line $H_0\cap H_1$ intersects 
	$S$ transversely at two smooth points. 
	Let $H\subseteq Q$ be the hyperplane section defined by the same linear equation
	in the variables $x_i$ as $H_1$. This implies that $p \in H$.
	Furthermore, we may assume that $\Gamma \not\subseteq H$. 
    Notice that $H$ is smooth at $p$ because $H_1 \neq H_0$.
    
    We prove that the pair $(\pp^3, H_0 + H_1 + S)$ is log canonical. Let $C_0 := H_0 \cap S$, $C_1 := H_1\cap S$ and $L=H_0\cap H_1$. As mentioned above, $C_0$ is a reduced conic of rank at least two. Similarly, $C_1$ is a reduced conic of rank at least two. The quadric $S$ is reduced. Away from $H_0$, the pair $(\pp^3,S)$ is lc. Near $H_0,$ by adjunction we get 
    \[
        (K_{\pp^3} + H_0 + H_1 + S)|_{H_0} = K_{H_0} + L + C_0.
    \]
    The divisor $L+C_0$ is reduced and only has nodes, so it is simple normal crossing. Then, by inversion of adjunction \cite{Kaw07}, $(\pp^3, H_0 + H_1 + S)$ is log canonical near $H_0$. We follow the same argument to show that $(\pp^3, H_0 + H_1 + S)$ is log canonical near $H_1$. Therefore, $(\pp^3, H_0 + H_1 + S)$ is log canonical.
	Let $q \in H_0\cap H_1 \cap S$ be one of the two points in the intersection. 
	Then $\mult_q(H_0 + H_1 + S) = 3$, and because the intersection is transversal, 
	$(\pp^3, H_0 + H_1 + S)$ is a log Calabi--Yau pair of coregularity 0.
	Proposition \ref{prop:criterion-sing-p3} implies that $(\pp^3, H_0 + H_1 + S)$ is of cluster type.
	
	The projection from the point $p$ is resolved by 
	\[
	\xymatrix{
		& Z := {\rm Bl}_p Q \ar[dr]^{\gamma} \ar[ld]_{\beta} &\\
		Q \ar@{-->}[rr] & & \pp^3,
	}
	\]
	where $\gamma$ is the blow-up of $C \subseteq \pp^3$. See, for example, the proof of \cite[Proposition 3.3]{Vir26}.
	Let $E_\beta$ and $E_\gamma$ be the exceptional divisors of $\beta$ and $\gamma$, respectively.
	Notice that $\gamma|_{E_\beta}\colon E_\beta \to H_0$ is an isomorphism.
	Let $S_Z$ and $H_Z$ be the strict transforms of $S$ and $H_1$. These strict transforms agree with the
	strict transforms of $F_\ell$ and $H$, respectively.
	This follows from the fact that $\beta$ and $\gamma$ restrict to
	isomorphisms from $Z \setminus (E_\beta \cup E_\gamma)$ onto
	$Q \setminus T_pQ$ and $\pp^3 \setminus H_0$, and the projection maps
	$F_\ell \setminus T_pQ$ onto $S \setminus H_0$ and $H \setminus T_pQ$ onto $H_1 \setminus H_0$.
	Hence, as $S$ and $H_1$ do not contain $C$, we have
	\[
		\gamma^* S = S_Z,\ \gamma^* H_1 = H_Z, \ \gamma^* H_0 = E_\beta + E_\gamma,
	\]
	and with respect to $\beta$, we obtain
	\[
		\beta^* F_\ell = S_Z + 2E_\beta,\ \beta^* H = H_Z + E_\beta.
	\]
	Fix a canonical divisor $K_Z$, and let $K_Q = \beta_* K_Z$ and $K_{\pp^3} = \gamma_* K_Z$.
	Then
	\[
		\beta^* (K_Q + F_\ell + H) = K_Z + S_Z + H_Z + E_\beta = \gamma^*(K_{\pp^3} + S + H_1 + H_0).
	\]
	The divisor $E_\gamma$ is a canonical place of $(\pp^3, H_0 + H_1 + S)$ and $E_\beta$ is a 
	log canonical place of $(Q, F_\ell + H)$. 
	Lemma \ref{lem:blow-up-of-cluster-type} and Lemma \ref{lem:cluster-type-descends} 
	imply that $(Q, F_\ell + H)$ is of cluster type. 
	Finally, using the fact that $\Gamma \not\subseteq H$, we can follow the argument in the proof of Proposition \ref{prop:1-8}, to show that the log pullback $(X,B_X)$ of $(Q, F_\ell + H)$ 
	via $\psi \circ \alpha \circ \varphi^{-1}$ is of cluster type.
	\end{proof}

\subsection{Mori-Mukai family \textnumero\,1.10}

In this subsection, we prove that any prime smooth Fano threefold of genus $12$ is of cluster type. We will proceed in two cases. First, we will assume that the threefold $X$ contains an ordinary line, i.e., a line $\ell$ for which 
$N_{\ell/X} \simeq \mathcal{O}_{\ell}\oplus \mathcal{O}_{\ell}(-1)$. 
Every prime smooth Fano threefold of genus $12$ contains an ordinary line
except for the Mukai-Umemura Fano threefold (see~\cite[Corollary 2.1.6 and Remark 2.2.4]{KPS18}).
Indeed, for the Mukai-Umemura Fano threefold the scheme of lines is everywhere non-reduced. In that case, we will seek an explicit torus chart of the Fano threefold which shows that it is cluster type directly. 
The main proposition of this subsection is the following.

\begin{proposition}\label{prop:1-10}
Let $X$ be a prime smooth Fano threefold of genus $12$.
Then, the variety $X$ is of cluster type. 
\end{proposition}

In order to prove the previous proposition, we will first prove some statements about del Pezzo surfaces with Du Val singularities containing smooth rational curves of anti-canonical degree $5$. 
The statements about surfaces will allow us to construct cluster type structures on some Sarkisov link model of the smooth Fano threefolds of interest in this subsection. 

\begin{lemma}\label{lem:du-val-degree-5}
Let $S$ be a Du Val del Pezzo surface of degree $5$.
Assume that $S$ is embedded in $\pp^N$ by the anti-canonical linear system.
Assume that $S$ contains a smooth rational curve $\Gamma$ of anti-canonical degree $5$.
Then $S$ contains a line $\ell$ satisfying the following.
For the minimal resolution $S'\rightarrow S$ the strict transform $\ell'$ of $\ell$ is a $(-1)$-curve which is either disjoint from the exceptional locus or intersects a single exceptional curve transversally at the endpoint of an $A_n$ chain\footnote{An $A_n$ chain is the chain of $n$ rational curves of the minimal resolution of an $A_n$ singularity.}.
\end{lemma} 

\begin{proof}
We will use the material from~\cite[Proposition 8.4]{CF88} and~\cite[Table 5]{Der14} for the classification of Du Val del Pezzo surfaces of degree $5$. 
From the classification of Du Val del Pezzo surfaces of degree $5$, we know that its possible singularities are
\[
\emptyset, \quad A_1, \quad 2A_1, \quad A_2, \quad A_2+A_1, \quad A_3, \text{ or } A_4.
\]
Let $\psi\colon S'\rightarrow S$ be the minimal resolution. 
Its minimal resolution $S'$ is a weak del Pezzo surface with Picard lattice 
\[
{\rm Pic}(S')=  \zz h \oplus \bigoplus_{i=1}^4 \zz e_i.
\]
The intersection products are given by
\[
h^2=1, \quad e_i^2=-1, \quad h\cdot e_j = e_i \cdot e_j =0 (i\neq j), \quad -K_{S'} \sim 3h -\sum_{i=1}^4 e_i.
\]
Choose the following roots orthogonal to $K_{S'}$:
\[
r_1 = e_1-e_2, \quad 
r_2 = e_2-e_3, \quad 
r_3 = e_3-e_4, \quad 
\text{ and } \quad
r_4 = h-e_1-e_2-e_3.
\]
These are roots for the graph of an $A_4$ singularity. 
We argue that the existence of the curve $\Gamma$ excludes the possibility of the $A_4$ singularity in $S$.
In this case, we have ${\rm Cl}(S)=\zz[e]$ where $e$ is the image of any of the $e_i$'s. We have $\psi_*h\sim 3e$ and $-K_S\sim 5e$. Thus, we conclude $\Gamma\sim -K_S$. By adjunction, we conclude that $\Gamma$ cannot be isomorphic to $\pp^1$, leading to a contradiction.
We conclude that under the existence of the curve $\Gamma$, only the singularity baskets $\{\emptyset,A_1,2A_1,A_2,A_1+A_2,A_3\}$ are possible. 
In these cases the curve class $e_4$ is either disjoint from the exceptional locus or intersects $r_3$ transversally once. 
Riemann--Roch gives an effective divisor $D\sim e_4$.
Since $-K_{S'}$ is nef and $(-K_{S'})D=1$, write
$D=\ell'+Z$, where $\ell'$ is irreducible of anticanonical
degree one and $Z$ is effective and exceptional.
Adjunction and the Hodge index theorem give $(\ell')^2=-1$.
Since $e_4$ intersects every exceptional curve non-negatively,
\[
0\leq D\cdot Z
=\frac{D^2-(\ell')^2+Z^2}{2}
=\frac{Z^2}{2}.
\]
Negative definiteness forces $Z=0$. Thus $D=\ell'$ is an
irreducible $(-1)$-curve. Note that $-K_{S'}\cdot e_4=1$ from where we conclude that the image of a general element of $|e_4|$ in $S$ is a line. 
\end{proof} 

Now, we turn to prove a statement about hyperplane systems on smooth quintic del Pezzo threefolds. 

\begin{lemma}\label{lem:smooth-quintic-del-Pezzo-threefold-line}
Let $V\subset \pp^6$ be a smooth quintic del Pezzo threefold. 
Let $H_0$ be its hyperplane class and $S\in |H_0|$ be a Du Val del Pezzo surface of degree $5$. Assume that $S$ contains:
\begin{enumerate}
\item[(i)] a smooth rational curve $\Gamma$ of anti-canonical degree $5$; and 
\item[(ii)] a line $\ell$ whose strict transform in the minimal resolution is either disjoint from the exceptional locus or intersects a single exceptional curve transversally at the end point of an $A_n$ chain. 
\end{enumerate}
Then, for a general element $H\in |H_0|$ containing $\ell$, the pair $(V,H+S)$ is log canonical.
Furthermore, we can choose $H$ in such a way that it does not contain $\Gamma$.
\end{lemma} 

\begin{proof}
First, we argue that a general element $H\in |H_0|$ containing $\ell$ is indeed smooth. The restrictions of the ambient linear forms vanishing along $\ell$ generate $\mathcal{I}_{\ell/V}(H)$. For a fixed point $p\in \ell$, the set of elements $H\in |H_0|$ containing $\ell$ and singular at $p\in \ell$ has codimension at least two. Since $\ell$ is a line, the set of elements $H\in |H_0|$ containing $\ell$ and being singular at some closed point of $\ell$ form a subset of codimension one. We conclude that a general $H\in |H_0|$ containing $\ell$ is smooth along $\ell$. Then Bertini's theorem implies that such general element is smooth. 

The ideal of $\ell$ in $S\setminus S^{\rm sing}$ is invertible.
Bertini's theorem implies that a general element of $|H_0|$ containing $\ell$ restricts to $S$ to the sum of $\ell$ and a smooth curve with transversal intersection to $\ell$ along $S\setminus S^{\rm sing}$. It suffices to understand the behavior of this linear system along $S^{\rm sing}$. 

Let $p\in S^{\rm sing}$. By (ii), we may assume that locally analytically we have 
\[
(S;p) \simeq \{ (x,y,z)\mid xy=z^{n+1}\} \subset (\mathbb{A}^3;0) 
\text{ and }
\mathcal{I}_{\ell/V,p}=\langle x,z\rangle.
\]
The equation of $H$ at $p$ is given by $ax+bz=0$. 
The ambient linear forms generate the ideal $\langle x,z\rangle$.
So, we may assume that $b(p)\neq 0$ for the equation defining some general $H$.
Solving for $z$ gives $z=xg(x,y)$ so $S|_H$ has equation
\[
x(y-x^ng(x,y)^{n+1})=0
\]
which is an slc curve. 
Inversion of adjunction implies that $(V,H+S)$ is log canonical near $H$.
We conclude that $(V,H+S)$ is log canonical.
Finally, for the last statement, note that $\Gamma$ contains a closed point outside $\ell$, so it suffices to choose $H$ not containing such point.
\end{proof} 

We use the previous lemma regarding hyperplane systems on smooth quintic del Pezzos to give a special cluster type structure on such varieties. 

\begin{proposition}\label{prop:special-ct-structure-V_5}
Let $V\subset \pp^6$ be a smooth quintic del Pezzo threefold.
Let $H_0$ be the class of a hyperplane section in $V$.
Let $S\in |H_0|$ be a Du Val del Pezzo surface. 
Assume that $S$ contains a smooth rational curve $\Gamma$ which is of anticanonical degree $5$. Then, there is a smooth element $H\in |H_0|$,
not containing $\Gamma$, for which $(V,H+S)$ is of cluster type. 
\end{proposition}

\begin{proof}
By Lemma~\ref{lem:du-val-degree-5}, we can find a line $\ell\subset S$ satisfying the following condition. The strict transform of $\ell$ to the minimal resolution of $S$ is either disjoint from the exceptional locus or intersects a single exceptional curve transversally at the endpoint of an $A_n$ chain. 
By Lemma~\ref{lem:smooth-quintic-del-Pezzo-threefold-line} a general element $H\in |H_0|$ containing $\ell$ satisfies that $(V,H+S)$ is log canonical. We argue that the pair $(V,H+S)$ is of cluster type. 
Let $\alpha \colon W \rightarrow V$ be the blow-up of $\ell$. 
Note that $\ell$ is contained in both $H$ and $S$ and so 
$\alpha^*(K_V+H+S)=K_W+H_W+S_W+E_W$ where $E_W$ is the exceptional divisor
and $H_W$ and $S_W$ are the strict transforms of $H$ and $S$ to $W$, respectively. 
For the projection link, see~\cite[\S1.1.1]{DK19} (also compare~\cite[Proof of Theorem 3.6]{CPPZ21}). There is a projective birational morphism
$\beta \colon W \rightarrow \mathbb{Q}^3$ where $\mathbb{Q}^3$ is a smooth quadric in $\pp^4$. The morphism $\beta \colon W\rightarrow \mathbb{Q}^3$ is the blow-up of a smooth twisted cubic. 
Let $A_1,A_2$ and $A_3$ be the images of $H_W$, $S_W$ and $E_W$ to $\mathbb{Q}^3$. The images $A_1$ and $A_2$ are hyperplane sections in $\mathbb{Q}^3$ and so they do not contain the twisted cubic.
The divisor $E_W$ is the strict transform of the unique hyperplane section containing the twisted cubic. Thus, we conclude that 
\[
\alpha^*(K_V+H+S)=
K_W+H_W+S_W+E_W=
\beta^*(K_{\mathbb{Q}^3}+A_1+A_2+A_3). 
\]
Note that the pair $(\mathbb{Q}^3,A_1+A_2+A_3)$ is of cluster type due to
Theorem~\ref{thm:complexity-one-is-cluster-type}. Therefore, we conclude that
$(V,H+S)$ is of cluster type due to Lemma~\ref{lem:cluster-type-descends}
and Proposition~\ref{prop:cluster-type-bir}.
\end{proof}

Now, we turn to prove that any prime smooth Fano threefold of genus $12$ that contains an ordinary line is of cluster type.

\begin{proposition}\label{prop:1-10-containing-ord-line-ct}
Let $X$ be a prime smooth Fano threefold of genus $12$.
Assume that $X$ contains an ordinary line $\ell$.
Then $X$ is a cluster type variety. 
\end{proposition} 

\begin{proof}
We have a commutative diagram as follows:
\[
\xymatrix{
X' \ar@{-->}[r]^{\chi}\ar[d]_-{\sigma} & X^+ \ar[d]^-{\tau} \\ 
X & V 
}
\]
satisfying the following conditions:
\begin{itemize}
\item the morphism $\sigma$ is the blow-up of the ordinary line $\ell$; 
\item the small birational map $\chi$ is a flop; 
\item the variety $V$ is a smooth quintic del Pezzo threefold; and
\item the morphism $\tau$ is the blow-up of a smooth rational normal quintic curve $\Gamma$.
\end{itemize} 
All the previous facts come from~\cite[Theorem 18]{DFK26}. 
Let $E'$ be the exceptional divisor of $\sigma$.
Let $E^+$ be the strict transform of $E'$ in $X^+$
and let $S$ be the strict transform of $E'$ in $V$. 
By~\cite[Proof of Theorem 18]{DFK26}, we know that $S$ has canonical singularities\footnote{This is explicitly stated in line $5$ of the last paragraph of the proof.}. 
Let $H_0$ be the class of a hyperplane section in $V$. 
By construction, we know that $S\in |H_0|$. 
Therefore, we conclude that $S$ is a Du Val del Pezzo surface.
By Proposition~\ref{prop:special-ct-structure-V_5} there is $H\in |H_0|$
containing a line $\ell_S\subset S$ and not containing $\Gamma$ such that $(V,H+S)$ is cluster type. Its crepant pullback to $X^+$ has boundary $E^++H^+$ and the flop gives $(X^{\prime},E^{\prime}+H^{\prime})$. Since $E^{\prime}$ has coefficient one, Lemma~\ref{lem:cluster-type-descends} applies to $\sigma$ and proves that $X$ is cluster type. 
\end{proof} 

Now, we turn to prove that the Mukai-Umemura threefold is of cluster type.

\begin{proposition}\label{prop:MU-ct}
The Mukai-Umemura smooth Fano threefold $X_{\rm MU}$ is of cluster type. 
\end{proposition}

\begin{proof}
We will use the description of the Mukai-Umemura smooth Fano threefold
$X_{\rm MU}$ as a Sarkisov link of a smooth quadric due to Kuznetsov and Prokhorov~\cite[Theorem 1.3 and Theorem 2.2]{KP18}. The coordinates and exceptional loci of the following description are explained in~\cite[Remark 2.6 and Remark 2.11]{CS18}.

Consider the smooth quadric threefold 
\[
Q:=\{ [x:y:z:t:w]\mid xw+4yt-5z^2=0\} \subset \pp^4. 
\]
Let $S=\{xw=z^2=yt\}\subset Q$.
The surface $S$ is an irreducible surface, indeed, in the locus
where all coordinates are non-zero $S$ is just a $2$-dimensional algebraic torus. The complement of this open locus in $S$ is a union of curves.
The surface $H=Q\cap \{z=0\}$ is a smooth quadric. 
We consider the curve $\Gamma\subset \pp^4$ parametrized by 
\[
[t_0^6:t_0^5t_1:t_0^3t_1^3:t_0t_1^5:t_1^6].
\]
The curve $\Gamma$ is a smooth rational sextic. 
Then, we have a commutative diagram as follows
\[
\xymatrix{
\widetilde{Q} \ar[d]_-{\pi} \ar@{-->}[r]^-{\chi} & \widetilde{X}_{\rm MU}\ar[d]^-{\sigma} \\ 
Q \ar@{-->}[r]^-{\phi} & X_{\rm MU}
}
\]
satisfying the following conditions: 
\begin{itemize}
\item $\pi$ is the blow-up along $\Gamma$;  
\item the strict transform of $S$ in $\widetilde{X}_{\rm MU}$ is ${\rm Ex}(\sigma)$; and 
\item the flopping curves in $\widetilde{Q}$ are the strict transforms of the lines
\[
L_1=\{z=t=w=0\} \text{ and } L_2=\{x=y=z=0\}.
\]
\end{itemize} 
Consider the open embedding $\iota \colon \mathbb{G}_m^3 \rightarrow Q$ with
\begin{align*}
 (a,b,c) & \mapsto \left[ 
1:a:b:\frac{b^2-c}{a}:b^2+4c 
\right] 
\end{align*}
The image of this open embedding is
\[
T:=\{ [x:y:z:t:w] \mid xyz(z^2-yt)\neq 0\}. 
\]
One can check that $T\cap (S\cup H)=\emptyset$.
The inverse chart is $(y/x,z/x,(z^2-yt)/x^2)$. 
Note that the volume form on $T$ equals 
\[
\omega_T= - \frac{dy\wedge dz\wedge dt}{z(z^2-yt)}
\]
from where it follows that 
${\rm Poles}_{Q}(\omega_T)=S+H$.
We conclude that the pair $(Q,S+H)$ is of cluster type.
Therefore, due to the previously described Sarkisov link
and Lemma~\ref{lem:blow-up-of-cluster-type}, we conclude that
the pair $(X_{\rm MU},\phi_*H)$ is of cluster type. 
\end{proof} 

\begin{proof}[Proof of Proposition~\ref{prop:1-10}]
If $X$ is the Mukai-Umemura smooth Fano threefold, then $X$ is of cluster type due to Proposition~\ref{prop:MU-ct}.
If $X$ is not the Mukai-Umemura smooth Fano threefold, then $X$ admits an ordinary line (see~\cite[Corollary 2.1.6 and Remark 2.2.4]{KPS18}). 
Then, by Proposition~\ref{prop:1-10-containing-ord-line-ct} we conclude that $X$ is of cluster type. 
\end{proof} 

\section{Other smooth Fano threefolds}
\label{sec:other}
In this section, we study smooth Fano threefolds which are not blow-ups of $\pp^3$ or $\qq^3$ and 
not in the main series.
This section is organized as follows.
In Subsection~\ref{subsec:torus-act}, we study smooth Fano threefolds with a torus action.
In Subsection~\ref{subsec:2cap2} and Subsection~\ref{subsec:grass}, we study complete intersections of two quadrics and sections of Pl\"ucker embeddings of ${\rm Gr}(2,5)$, respectively, and their corresponding blow-ups.
Subsection~\ref{subsec:double-cover} treats the double cover of $\pp^1\times\pp^2$.
In Subsection~\ref{subsec:toric-hyperf}, we move to consider the cluster type condition for toric hypersurfaces and their blow-ups. 

\subsection{Varieties with a torus action}
\label{subsec:torus-act}
In this subsection, we argue that most smooth Fano threefolds admitting two-dimensional torus actions are cluster type. First, we recall which families of smooth Fano threefolds are toric.

\begin{proposition}\label{prop:toric}
Let $X$ be a smooth Fano threefold in the families \textnumero\,1.17, \textnumero\,2.33, \textnumero\,2.34, \textnumero\,2.35, \textnumero\,2.36, \textnumero\,3.25, \textnumero\,3.26, \textnumero\,3.27, \textnumero\,3.28, \textnumero\,3.29, \textnumero\,3.30, \textnumero\,3.31, \textnumero\,4.9, \textnumero\,4.10, \textnumero\,4.11, \textnumero\,4.12, 
\textnumero\,5.2 or \textnumero\,5.3 then $X$ is a toric variety and hence of cluster type.
\end{proposition}

Now, we turn to study smooth Fano threefolds with the action of a two-dimensional algebraic torus.
The toric classification in Proposition~\ref{prop:toric} is due to~\cite{Bat82,WW82}.

\begin{proposition}\label{prop:T-comp-one}
Let $X$ be a smooth Fano threefold which admits an effective action
of a two-dimensional algebraic torus.
Assume that $X$ is neither isomorphic to
a hypersurface of bi-degree $(1,2)$ in $\pp^2\times \pp^2$ nor the blow-up of 
$\mathbb{Q}^3$ along two conics.
Then, the variety $X$ is of cluster type. 
\end{proposition} 

\begin{proof}
If $X$ is toric, then the statement follows from Proposition~\ref{prop:toric}, therefore, we may assume that $X$ admits a torus action of complexity one, i.e., the action of a $2$-dimensional algebraic torus.

Let $\mathbb{T}$ be the two-dimensional algebraic torus. 
By~\cite[Theorem 5.6]{AHS08}, we know that $X$ is $\mathbb{T}$-equivariantly isomorphic to $X(\mathcal{S})$
where $\mathcal{S}$ is a divisorial fan on the normalized Chow quotient of $X$ by $\mathbb{T}$.
Note that such normalized Chow quotient must be a rational curve as $X$ is rationally connected (see, e.g,~\cite{Zha06}).
Therefore, there is a $\mathbb{T}$-equivariant projective birational morphism
$r\colon \widetilde{X}(\mathcal{S})\rightarrow X$
and a $\mathbb{T}$-equivariant good quotient
for the torus action $\pi\colon \widetilde{X}(\mathcal{S})\rightarrow \pp^1$ (see, e.g,~\cite[Section 12.2]{AIPSV12}). 
Note that every p-divisor $\mathcal{D}$ of $\mathcal{S}$ has SNC support as every divisor on $\pp^1$ must be snc. 
Henceforth, we know that $\pi$ must be toroidal (see~\cite[Proposition 2.6]{LS13}).
Therefore, every degenerate divisor of $\pi\colon \widetilde{X}(\mathcal{S})\rightarrow \pp^1$ can be contracted over $\pp^1$ by using a toric birational map. The morphism $\pi$ is trivial with fiber a projective surface $F$ outside three closed points $\{0\},\{1\}$, and $\{\infty\}$ and at most two of these fibers are multiple.
This statement is a consequence of the classification and description 
of $\mathbb{T}$-complexity one smooth Fano threefolds due to S\"u\ss{}  (see~\cite[Section 5]{Sus13}). 
Here, we used the fact that $X$ is not isomorphic to a hypersurface of multi-degree $(1,2)$ in $\pp^2\times \pp^2$ nor isomorphic to the blow-up of $\mathbb{Q}^3$ along two conics.
Indeed, those are precisely the cases in which $\pi$ has three multiple fibers (see~\cite[Page 14 and Page 16]{Sus13}). Let $B_F$ be the closure in $\widetilde{X}$ of the product of the torus invariant boundary of $F$ with $\pp^1\setminus \{0,1,\infty\}$. 
Let $H_0$ and $H_\infty$ be two reduced fibers of $\pi$ for which
$H_0+H_\infty$ contain all the multiple fibers of $\pi$.
Therefore, by construction, the pair $(\widetilde{X}(\mathcal{S}),H_0+H_\infty+B_F)$ is a $\mathbb{T}$-equivariant log Calabi--Yau pair of index one and coregularity zero.
Furthermore, as the base is $\pp^1$ and the general fiber a toric surface $F$, the 
complexity of the pair $(\widetilde{X}(\mathcal{S}),H_0+H_\infty+B_F)$ equals the number
of components of $\pi^{-1}(1)$ minus one.

We can contract all the prime components of
$\pi^{-1}(1)$ except one component using a toric birational map. These contractions globalize as follows. Choose their defining semiample
divisors with the same restriction to the general fiber as a fixed
$\pi$-ample divisor. By the $p$-divisor description, these semiample divisors extend over an
affine neighborhood of $1$ and glue with the fixed divisor away from $1$, 
up to a divisor pulled back from the base. The resulting divisors are
$\pi$-semiample and $\pi$-ample away from $1$. Twisting by
$\pi^*\mathcal O_{\mathbb P^1}(N)$ for $N\gg0$ gives the desired projective contractions.

By doing so, the complexity of the Calabi--Yau pair drops to zero
and so we can apply~\cite[Theorem 2]{BMSZ18}
to deduce that this is a toric model of the pair.
Let $q:\widetilde X\to Z$ be the resulting contraction and put
$\Delta:=H_0+H_\infty+B_F$ and $\Delta_Z:=q_*\Delta$.
Then
\[
K_{\widetilde X}+\Delta=q^*(K_Z+\Delta_Z).
\]
Every $q$-exceptional divisor has coefficient zero in $\Delta$, hence
is a canonical place of the toric pair $(Z,\Delta_Z)$.
Thus $(\widetilde X,\Delta)$ is of cluster type by Lemma~\ref{lem:cluster-type-descends}
The birational contraction $r\colon \widetilde{X}\rightarrow X$ only contracts components 
of $B_F$. Indeed, by~\cite[Theorem 10.1]{AH06}, every exceptional divisor of $r$
must be horizontal over $\pp^1$. Thus, by Lemma~\ref{lem:cluster-type-descends}, we conclude that $X$ is of cluster type. 
\end{proof}  

\subsection{Complete intersection of two quadrics}
\label{subsec:2cap2}
In this subsection, we study complete intersections of two quadrics and their blow-ups which are smooth Fano threefolds.

\begin{lemma}\label{lem:q1-cap-q2}
Let $X$ be a smooth complete intersection of two quadrics in $\pp^5$.
Let $H$ be a generator of its Picard group. 
Let $C\subset X$ be one of the following types of smooth curves.
\begin{enumerate}
\item[(i)] an elliptic curve, which is a scheme theoretic intersection of two hyperplanes with $X$;
\item[(ii)] a conic which is not contained in a plane in $X$; or
\item[(iii)] a line.
\end{enumerate} 
Then, there is a smooth element $B\in |H|$ that contains $C$. 
\end{lemma}

\begin{proof}
In the first case, let $H_1$ and $H_2$ be the two hyperplane sections in $X$ cutting $C$ and let $\Lambda$ be the pencil generated by them. Since $C$ is smooth of codimension $2$, the differentials of $H_1$ and $H_2$ are linearly independent at every point of $C$. Therefore every element of $\Lambda$ is smooth at every point of $C$. By Bertini's theorem a general element of $\Lambda$ is smooth outside of $C$. Hence for $H \in \Lambda$ general, the intersection $H\cap X$ is smooth and contains $C$.

Assume that $C$ is a smooth conic.
Let $Q_1$ and $Q_2$ be two quadrics that cut-out  $X$. Let $\Pi$ be the plane containing $C$.
Note that $\Pi$ is not contained in $X$. We may assume that $\Pi\cap X=C$ and that $C=\Pi\cap Q_1$ scheme-theoretically.
The hyperplanes containing $\Pi$ are precisely the hyperplanes containing $C$. Let $\Lambda$ be the linear system of such hyperplanes. By Bertini's theorem, a general element of $\Lambda$ has smooth intersection with $X$ away from $C$.
Let $p\in C$ be a fixed point. Given $H\in \Lambda$, the intersection $H\cap X$ is singular at $p$ precisely when $H$ contains $T_pX=T_pQ_1\cap T_pQ_2$. The linear equations of $\Pi$ generate $\mathcal I_{C/X}(H)$ because $X\cap\Pi=C$ scheme-theoretically. Their conormal evaluations span a two-dimensional space, so singularity at $p$ imposes two independent linear conditions. 
By moving the point $p\in C$, we conclude that the set of hyperplanes in $\Lambda$ whose intersection with $X$ is singular at some point of $C$ gives a one-dimensional subvariety.
Hence, a general element of $\Lambda$ has smooth intersection with $X$ and contains $C$. 

Finally, assume that $C$ is a line. The linear system of hyperplanes in $\pp^5$ that contain $C$ is a three-dimensional space $\Lambda$.
For $p\in C$ fixed and $H\in \Lambda$, we have that $H\cap X$ is singular at $p$ precisely when $H$ contains $T_pX$.
Therefore, hyperplanes in $\Lambda$ that are singular at $p\in C$ form a one-dimensional subvariety. 
By moving the point $p\in C$, we conclude that the set of hyperplanes in $\Lambda$ whose intersection with $X$ is singular at some point of $C$ forms a $2$-dimensional subvariety of $\Lambda$.
By Bertini's theorem, a general element $H\in \Lambda$ has smooth intersection with $X$ outside $C$. 
Hence, we conclude that for $H\in \Lambda$ general, the intersection $H\cap X$ is smooth and contains $C$.
\end{proof}

\begin{proposition}\label{prop:cluster-type-q1capq2}
Let $X$ be a smooth Fano threefold which is either a complete intersection of two quadrics or a blow-up of a complete intersection of two quadrics. Then, the variety $X$ is of cluster type.
\end{proposition}

\begin{proof}
By the classification of smooth Fano threefolds, we know that $X$ is either a complete intersection of two quadrics
or the blow-up of such variety along a curve $C$ which is either 
elliptic, conic, or a line. 
For the remainder of the proof, $X$ denotes the Picard rank one complete intersection. We pass to its blow-up only at the end. If no center is prescribed, choose a line $C$. By Lemma~\ref{lem:q1-cap-q2}, $C$ is contained in a smooth element $B_1\in |H|$, where $H$ is the generator of the Picard group. By construction, the surface $B_1$ is a smooth del Pezzo surface of degree $4$. Let $D\in |-K_{B_1}|$ be a nodal anti-canonical curve with two components $D_1$ and $D_2$. We may assume that $D_1^2=-1$. By Kawamata-Viehweg vanishing, we know that there exists $B_2\in |-(K_X+B_1)|$ such that $B_1\cap B_2=D$. 
We consider the log Calabi--Yau pair $(X,B_1+B_2)$.
The exact sequence $0\to\mathcal O_X\to\mathcal O_X(H)\to\mathcal O_{B_1}(-K_{B_1})\to0$ and $H^1(\mathcal O_X)=0$ lift $D$. General lifts are smooth off its nodes. Varying the lift by $B_1$ makes the normal derivative nonzero at every node. Hence $B_2$ is smooth and adjunction proves the lc property.
We argue that $(X,B_1+B_2)$ is of cluster type, and then by Lemma~\ref{lem:blow-up-of-cluster-type}, we conclude that the blow-up of $X$ along an elliptic, conic, or line is again of cluster type.

The divisors $B_1,B_2$ generate a pencil with smooth general member, a del Pezzo surface of degree four, and base curve $D$.
We can find a dlt modification $(X',B')\rightarrow (X,B_1+B_2)$ that separates this pencil.
Then, we have a fibration $X'\rightarrow \pp^1$ such that the fiber over $\{0\}$ is the strict transform of $B_1$ and the fiber over $\{\infty\}$ is the strict transform of $B_2$.
Therefore, if we apply the canonical bundle formula for $(X',B')$ over $\pp^1$, we obtain the pair $(\pp^1,\{0\}+\{\infty\})$. 
The pair $(X',B')$ has two horizontal components over $\pp^1$, let's say $E_1+E_2$, such that the restriction to a general fiber is $D_1+D_2$ where $(D_1)^2=-1$.

We argue that we can contract $E_1$ over $\pp^1$.
Since $X^{\prime}\to X$ extracts only lc places, pulling back $(X,(1-\epsilon)(B_1+B_2))$ for small $\epsilon>0$ gives an effective klt boundary with nef and big anticanonical divisor. Hence $X^{\prime}$ is of Fano type, and is a relative Mori dream space~\cite{BCHM10}. On the generic fiber the chosen ray is $\mathbb R_{\geq0}[D_1]$, supported by $-K_F+D_1$. We extend this supporting class and run its relative MMP, then we take its ample model. The preliminary steps are vertical, and its only generic contraction is $D_1$. All steps are crepant for the log Calabi--Yau pair. The inverse extracts only coefficient zero or coefficient one places, so Lemma~\ref{lem:blow-up-of-cluster-type} lifts
the cluster type property. This realizes the desired contraction of $E_1$ over $\pp^1$ (compare with~\cite[Theorems 3.3 and 3.7]{KM98}). 

Let $X'\dashrightarrow X''$ be the induced birational contraction
and $(X'',B'')$ be the induced log Calabi--Yau pair on $X''$.
The boundary divisor $B''$ has a unique prime component which is horizontal over $\pp^1$ and the general fiber of $X''\rightarrow \pp^1$ is a smooth del Pezzo surface of degree $5$.
Thus, by Proposition~\ref{prop:fibration-to-smooth-curve}, we conclude that $(X'',B'')$ is of cluster type.
Then, by Lemma~\ref{lem:blow-up-of-cluster-type}, we conclude that our pair $(X,B_1+B_2)$ is of cluster type.
\end{proof}

\subsection{Section of Pl\"ucker embedding of ${\rm Gr}(2,5)$ by a codimension three linear subspace}
\label{subsec:grass}
In this subsection, we study smooth Fano threefolds which are blow-ups of sections of Pl\"ucker embeddings of the Grassmannian ${\rm Gr}(2,5)$ by codimension 3 subspaces.

\begin{lemma}\label{lem:plucker-section}
Let $X$ be a smooth Fano threefold which is a section of 
the Pl\"ucker embedding of the Grassmannian ${\rm Gr}(2,5)$ by codimension 3 subspaces. Let $H$ be the generator of its Picard group.
Let $C\subset X$ be one of the following classes of curves:
\begin{enumerate}
\item[(i)] an elliptic curve, which is a scheme theoretic intersection of two hyperplanes with $X$;
\item[(ii)] a twisted cubic whose blow-up is Fano;
\item[(iii)] a conic; or 
\item[(iv)] a line. 
\end{enumerate} 
Then, there is a smooth element $B\in |H|$ that contains $C$. 
\end{lemma} 

\begin{proof}
In the first case, let $H_1$ and $H_2$ be the two hyperplane sections in $X$ cutting $C$ and let $\Lambda$ be the pencil generated by them. Since $C$ is smooth of codimension $2$, the differentials of $H_1$ and $H_2$ are linearly independent at every point of $C$. Therefore every element of $\Lambda$ is smooth at every point of $C$. By Bertini's theorem a general element of $\Lambda$ is smooth outside of $C$. Hence for $H \in \Lambda$ general, the intersection $H\cap X$ is smooth and contains $C$.

Assume that $C$ is a twisted cubic. 
Let $\Lambda$ be the linear system of hyperplanes in $\pp^6$ that contain $C$. Every element of $\Lambda$ must contain $\Pi$ the three-dimensional projective space spanned by $C$. 
The linear system $|\mathcal{I}_C(H)|$ consists of the elements of $\Lambda$ restricted to $X$.
The restrictions of the quadratic equations of $X$ span $H^0(\mathcal I_{C/\Pi}(2))$. Indeed, rank two would leave a residual bisecant line of nonpositive anticanonical degree on the Fano blow-up. Lower rank would put a surface of degree at most two in $X$, contradicting $\operatorname{Pic}(X)=\zz H$ and $H^3=5$. Thus $X\cap\Pi=C$ scheme-theoretically and the linear equations of $\Pi$ generate $\mathcal I_{C/X}(H)$. Therefore, $|\mathcal{I}_C(H)|$ is two-dimensional and its base locus is indeed $C$. 
Let $p\in C$. 
A hyperplane $\Gamma \in \Lambda$ satisfies that $\Gamma \cap X$ is singular at $p$ if and only if $\Gamma \supseteq T_pX$. 
Note that $T_pX\cap \Pi$ has dimension one, so there is a unique hyperplane $\Gamma \in \Lambda$ that contains both $T_pX$ and $\Pi$. 
Thus, for each $p\in C$ there is a unique $\Gamma \in \Lambda$ for which $\Gamma \cap X \in |\mathcal{I}_C(H)|$ is singular at $p\in C$.
Therefore, the set of all $\Gamma\in \Lambda$ for which $\Gamma \cap X$ is singular at some point of $C$ forms a subvariety of $\Lambda$ of dimension at most one. Thus, for $\Gamma \in \Lambda$ general, we have that $\Gamma \cap X$ contains $C$ and is smooth.

Assume that $C$ is a conic. Let $\Pi$ be the span of the conic.
Hence, the system $|\mathcal{I}_C(H)|$ consists of the restriction
to $X$ of hyperplanes of $\mathbb{P}^6$ that contains $\Pi$.
Let $\Lambda$ be the system of hyperplanes in $\pp^6$ that contains $\Pi$. The system $\Lambda$ is three-dimensional.
The plane $\Pi$ is not contained in $X$ and so $X\cap \Pi=C$.
The first assertion follows from $\operatorname{Pic}(X)=\zz H$ and $H^3=5$. The quadratic equations then give the asserted scheme-theoretic intersection. Thus the conormal evaluations have rank two, as required by the tangent space count. 
Therefore, the base locus of $|\mathcal{I}_C(H)|$ is precisely $C$.
Fix a point $p\in C$. A hyperplane $\Gamma \in \Lambda$ satisfies that $\Gamma\cap X \in |\mathcal{I}_C(H)|$ is singular at a $p\in X$ if and only if $\Gamma$ contains both $\Pi$ and $T_pX$. 
Such set of hyperplanes form a $\pp^1$ in the linear system $\Lambda$ and so the set of hyperplanes in $\Lambda$ whose intersection with $X$ are singular at some point of $C$ form a 2-dimensional subvariety of $\Lambda$. We conclude that a general element of $|\mathcal{I}_C(H)|$ is smooth. 

Assume that $C$ is a line. The system $|\mathcal{I}_C(H)|$ is $4$-dimensional in this case. Indeed, the system $|\mathcal{I}_C(H)|$ consists of the restrictions to $X$ of the elements of the linear system $\Lambda$ of hyperplanes in $\pp^6$ containing $C$. It is clear that the linear system $|\mathcal{I}_C(H)|$ has base locus $C$. 
Given a point $p\in C$, a hyperplane
$\Gamma\in \Lambda$ satisfies that $\Gamma\cap X \in |\mathcal{I}_C(H)|$ is singular at $p\in C$ if and only if $\Gamma$ contains $T_p X$.
Thus, such set of hyperplanes forms a $2$-dimensional space for a fixed $p\in C$. Moving $p\in C$, we observe that the set of hyperplanes $\Gamma \in \Lambda$ for which $\Gamma \cap X \in |\mathcal{I}_C(H)|$ is singular at some point $p\in C$ form a $3$-dimensional subvariety of the $4$-dimensional system $\Lambda$.
Thus, for $\Gamma \in |\mathcal{I}_C(H)|$ general, we have that $\Gamma \cap X$ is smooth and contains $C$. 
\end{proof}

\begin{proposition}\label{prop:cluster-type-section-grass}
Let $X$ be a smooth Fano threefold which is either 
a section of the Pl\"ucker embedding of ${\rm Gr}(2,5)$ by a codimension $3$ subspace or a blow-up of such variety. 
Then, the variety $X$ is of cluster type. 
\end{proposition} 

\begin{proof}
By the classification of smooth Fano threefolds, we know that $X$ is either a section of the Pl\"ucker embedding of ${\rm Gr}(2,5)$ by a  codimension $3$ subspace or a blow-up of such variety along a curve $C$ which is an elliptic curve, a twisted cubic, a conic, or a line. 
For the remainder of the proof, $X$ denotes the Picard rank one variety. If no blow-up center is prescribed, choose a line $C$. By Lemma~\ref{lem:plucker-section} and Kawamata--Viehweg, we can find a Calabi--Yau pair $(X,B_1+B_2)$ such that $B_1$ contains $C$.
Without loss of generality, we may assume that $B_1\cap B_2$ is
an irreducible nodal curve.
Furthermore, we have $B_1\sim B_2$ and the induced pencil has general fiber a smooth del Pezzo surface of degree $5$.
As above, $H^1(\mathcal O_X)=0$ lifts the chosen nodal divisor on $B_1$, and varying the lift by $B_1$ gives a smooth $B_2$ and an lc pair. The pencil is smooth over a nonempty open set by Bertini. Its crepant resolution blows up Cartier curves on the general surface fiber, so leaves that fiber unchanged.
Let $(X',B') \rightarrow (X,B_1+B_2)$ be a dlt modification that separates the pencil. Therefore, we have a fibration $X'\rightarrow \pp^1$ such that the fiber over $\{0\}$ is the strict transform of $B_1$ and the fiber over $\{\infty\}$ is the strict transform of $B_2$. If we apply the canonical bundle formula for $(X',B')$ over $\pp^1$, we obtain the log pair $(\pp^1,\{0\}+\{\infty\})$. 
The boundary $B'$ has a unique prime component which is horizontal over $\pp^1$ and its restriction to a general fiber is an irreducible nodal curve. 
The general fiber of $X'\rightarrow \pp^1$ is a smooth del Pezzo surface of degree $5$. Therefore, by Proposition~\ref{prop:fibration-to-smooth-curve}, we conclude that $(X',B')$ is of cluster type.
Then, by Lemma~\ref{lem:cluster-type-descends}, we conclude that our pair $(X,B_1+B_2)$ is of cluster type.
Finally, by Lemma~\ref{lem:blow-up-of-cluster-type}, we deduce that the blow-up of $X$ at $C$ is of cluster type as well.
\end{proof} 

\subsection{Double cover of $\pp^1\times \pp^2$}
\label{subsec:double-cover}
In this subsection, we study the double cover
of $\pp^1 \times \pp^2$ along a general divisor of multi-degree $(2,2)$.
This smooth Fano threefold comes naturally equipped with a boundary divisor which is the pull-back of the sum of some torus invariant divisors on $\pp^1\times \pp^2$. In this case, we show that such boundary does not give a cluster type. Regardless of this, we produce a special boundary on this variety which gives a cluster type pair. We present both constructions as the latter is motivated by the former.

\begin{proposition}\label{prop:double-cover-p1xp2}
Let $X$ be a smooth Fano threefold which is a double cover of $\pp^1\times \pp^2$ along a smooth divisor of multi-degree $(2,2)$.
Let $H_1$ be a general divisor of multi-degree $(1,0)$
and $F_i$ be two general divisors of multi-degree $(0,1)$ with $i\in \{1,2\}$. 
Let $H_{X,1}$ and $F_{X,i}$, with $i\in \{1,2\}$, be the reduced pull-backs of $H_1$ and $F_i$, respectively.
Then, the pair $(X,H_{X,1}+F_{X,1}+F_{X,2})$ is not a  cluster type pair.
\end{proposition} 

\begin{proof}
Let $f\colon X\rightarrow \pp^1\times \pp^2$ be the induced cover and observe that we have 
\[
K_X+H_{X,1}+F_{X,1}+F_{X,2} = f^*\left(
K_{\pp^1\times \pp^2}+H_1+F_1+F_2+\frac{1}{2}(2,2)
\right) \sim 0.
\]
Here, $H_{X,1}$ is the reduced pre-image of $H_1$ in $X$
and $F_{X,i}$ is the reduced pre-image of $F_i$ in $X$ for $i\in \{1,2\}$.
By construction, the pair $(X,H_{X,1}+F_{X,1}+F_{X,2})$ is a log Calabi--Yau pair of index one, coregularity zero, and complexity $2$.
We have a commutative diagram of crepant maps
\[
\xymatrix{
(X,H_{X,1}+F_{X,1}+F_{X,2}) \ar[r]^-{f}\ar[rd]_-{g} & \left(\pp^1\times \pp^2, H_1+F_1+F_2+\frac{1}{2}(2,2) \right)\ar[d]^-{\pi} \\
 & (\pp^2, B)
}
\]
where $f$ is a finite crepant morphism between Calabi--Yau pairs
and $\pi$ is a crepant fibration, i.e., the pair $(\pp^2,B)$ is induced by the canonical bundle formula. The morphism $g$ has connected fibers. The boundary $B$ is the sum of two lines $L_1$ and $L_2$ (the images of $F_1$ and $F_2$) and a conic with coefficient $\frac{1}{2}$. 
Let $\phi_S\colon S\rightarrow \pp^2$ be the blow-up at the intersection of $L_1$ and $L_2$, so we have a fibration 
$\pi_S\colon S\rightarrow \pp^1$. 
Let $\phi_Y\colon Y\rightarrow X$ be the blow-up of $F_{X,1}\cap F_{X,2}$.
Let $\phi_S^*(K_{\pp^2}+B)=K_S+E_S+B_S$ where $B_S$ is the strict transform of $B$ in $S$. 
Write
\[
\phi_Y^*(K_X+H_{X,1}+F_{X,1}+F_{X,2})=
K_Y+H_{Y,1}+F_{Y,1}+F_{Y,2}+E_Y, 
\]
where the subscript $Y$ denotes strict transform on $Y$ 
and $E_Y$ is the exceptional divisor of $\phi_Y$. 
Note that we have a fibration $\pi_Y\colon Y\rightarrow S$.
Then, we have a commutative diagram of crepant morphisms:
\[
\xymatrix{
(Y,H_{Y,1}+F_{Y,1}+F_{Y,2}+E_Y) \ar[r]^-{\phi_Y}\ar[d]_-{\pi_Y} &
(X,H_{X,1}+F_{X,1}+F_{X,2}) \ar[r]^-{f}\ar[rd]_-{g} & \left(\pp^1\times \pp^2, H_1+F_1+F_2+\frac{1}{2}(2,2) \right)\ar[d]^-{\pi} \\
(S,B_S+E_S) \ar[rr]_-{\phi_S}\ar[d]_-{\pi_S} & & (\pp^2, B) \\ 
(\pp^1,\{0\}+\{\infty\}) & & 
}
\]
where $\{0\}$ and $\{\infty\}$ are the images in $\pp^1$ of the strict transforms of $L_1$ and $L_2$.
The fibration $\pi_S\circ \pi_Y$ has relative Picard rank two. 
The general fiber $G$ of $Y\rightarrow \pp^1$
is a del Pezzo surface of degree $4$. 
Note that both $F_{Y,1}$ and $F_{Y,2}$ are vertical over $\pp^1$
while the divisors $H_{Y,1}$ and $E_Y$ are horizontal over $\pp^1$.
Let $G_{Y,1}$ and $G_Y$ be the restrictions of $H_{Y,1}$ and $E_Y$ to $G$, respectively. Note that $G_{Y,1}^2=G_Y^2=0$ and so~\cite[Theorem 1.7]{MY24} implies that the pair $(Y,H_{Y,1}+F_{Y,1}+F_{Y,2}+E_Y)$ is not of cluster type.
Therefore, Lemma~\ref{lem:blow-up-of-cluster-type} implies that
$(X,H_{X,1}+F_{X,1}+F_{X,2})$ is not cluster type.
\end{proof} 

By producing some tangency on the divisors $F_{X,1}$ and $F_{X,2}$ we obtain a cluster type boundary.

\begin{proposition}
\label{prop:double-cover-p1xp2-cluster-type}
Let $X$ be a smooth Fano threefold which is a double cover of 
$\pp^1\times \pp^2$ along a smooth divisor of multi-degree $(2,2)$. 
Then $X$ is of cluster type.
\end{proposition}

\begin{proof}
Let $X \xrightarrow{g} \pp^1\times \pp^2$ be 
the double cover branched over a divisor of multi-degree $(2,2)$. 
We follow the construction of \cite[Lemma 5.3]{LMV24}, 
based on \cite[Lemma 11.5]{ALP23}. Consider the diagram
\[
	\xymatrix{
		& X \ar[d]_g \ar@/_1pc/[ldd]_{\pi_1} \ar@/^1pc/[rdd]^{\pi_2} &  \\
		& \pp^1 \times \pp^2 \ar[ld] \ar[rd] & \\
		\pp^1	& & \pp^2		
	}
\]
where $\pi_2$ is a conic bundle 
with quartic discriminant curve $\Delta \subseteq \pp^2.$
Let $H = \pi_1^* \oo_{\pp^1}(1)$ and $F = \pi^*_2 \oo_{\pp^2}(1).$ 
Then $- K_X \sim H + 2F.$
	
Let $q\in \Delta$ be a general point. 
Then the fiber over $q$ corresponds to two lines 
$\pi_2^{-1}(q) = L_1 \cup L_2,$ 
meeting at one node. 
Each line $L_i$ is mapped isomorphically to $\pp^1$ via $\pi_1$. 
Let $\ell_1,\ell_2 \subseteq \pp^2$ be two distinct lines passing through $q$. 
Define the divisors $D_1,D_2\in |F|$ as $D_1 = \pi_2^{-1}(\ell_1)$ 
and $D_2 = \pi_2^{-1}(\ell_2)$. 
Additionally, define a general divisor $D_3\in |H|$ 
that avoids the node $L_1\cap L_2$. 
Then $B:= D_1 + D_2 + D_3 \in |-K_X|$, 
and so $(X,B)$ is a Calabi--Yau pair.
	
Let $Y :={\rm Bl}_{L_1} X \xrightarrow{\phi} X$,
and set $E$ to be the exceptional divisor of $\phi$, 
and $D_i'$ to be the strict transform of $D_i$, $i=1,2,3.$ 
Then the divisor $B_Y = D_1' + D_2' + D_3' + E$ 
is simple normal crossing, and
\[
	K_Y + B_Y = \phi^*(K_X + B).
\]
	
Let $A = \phi^*(H+F) - E$. 
The linear system defines a birational morphism 
$Y \xrightarrow{\psi} Q$ to a singular quadric 
$Q\subseteq \pp^4$. We have that $\dim Q = 3$ and $\rank \Cl(Q) = 2$. 
Set $B_Q = \psi_* B_Y$. Then $K_Q + B_Q = \psi_*(K_Y + B_Y) \sim 0$,
and since $K_Y + B_Y$ is trivial with respect to $\psi$,
$\psi^*(K_Q + B_Q) = K_Y + B_Y.$
Using the blow-up description, we can see that $A^2\cdot E=3$ and $A^2\cdot D_i'=1$ for each $i\in  \{1,2,3\}$, we conclude that no component of $B_Y$ gets contracted by $\psi$.
Hence every exceptional divisor of $\psi$ has coefficient zero in $B_Y,$
thus, every exceptional divisor of $\psi$ is a canonical place of $(Q,B_Q).$
The pair on $Q$ is log canonical, and has complexity $3+2-4 = 1.$
Apply Theorem \ref{thm:complexity-one-is-cluster-type} to $(Q,B_Q)$,
then Lemma \ref{lem:blow-up-of-cluster-type} to $\psi,$
and Lemma \ref{lem:cluster-type-descends} to $\phi$ to conclude that $X$ is of cluster type.
\end{proof}

\begin{corollary}
\label{coro:3.4}
Let $X$ be a smooth Fano threefold which is a double cover of
$\pp^1\times \pp^2$ along a smooth divisor of multi-degree $(2,2)$.
Let $Y$ be the blow-up along a smooth fiber of the conic fibration 
$X \to \pp^2$. 
Then $Y$ is of cluster type.
\end{corollary}

\begin{proof}
Let $X$ be a double cover of $\pp^1 \times \pp^2$ 
branched over a divisor of multi-degree $(2,2)$,
and let $\pi_2\colon X\to \pp^2$ be the induced conic fibration. 
Let $C = \pi_2^{-1}(r)$ be a smooth fiber of $\pi_2$.
Let $D := D_1 + D_2 + D_3$ be the divisor on $X$ constructed
in Proposition \ref{prop:double-cover-p1xp2-cluster-type}. 
Recall that $D_i=\pi_2^{-1}(\ell_i)$ for $i=1,2$, 
where $\ell_1,\ell_2$ are two distinct lines in $\pp^2$
intersecting at a general point $q\in \Delta$. 
Furthermore, we may assume that $\ell_1$ is 
the line connecting the points $r$ and $q$ in $\pp^2$.
	
Let $Y = {\rm Bl}_C X \xrightarrow{\phi} X$.
We argue that the exceptional divisor $E$ of $\phi$
is a canonical place of $(X,D)$. 
Indeed, $C\subseteq D_1$ with multiplicity one, 
and $C\not\subseteq D_2$. 
Even more, $C$ is a double cover of $\pp^1$, 
so a general choice of $D_3$ does not contain $C$. 
If $D_1',D_2',D_3'$ are the strict transforms 
of $D_1,D_2,D_3$ in $Y$, respectively, then
\[
	K_Y + D_1' + D_2' + D_3' = \phi^*(K_X + D_1 + D_2 + D_3).
\]
Therefore, Lemma \ref{lem:blow-up-of-cluster-type} implies that 
$(Y,D_1' + D_2' + D_3')$ is of cluster type.
\end{proof}

\subsection{Toric hypersurfaces and their blow-ups}
\label{subsec:toric-hyperf}
In this subsection, we study the cluster type property for smooth Fano threefolds which are hypersurfaces in toric varieties, and their blow-ups. 

\begin{proposition}\label{prop:(1,2)-in-p2xp2}
Let $X$ be a smooth Fano threefold which is a hypersurface of multi-degree $(1,2)$ in $\pp^2\times \pp^2$. Then, the variety $X$ is of cluster type. 
\end{proposition} 

\begin{proof}
Let $H_1$ be a general hypersurface in $\pp^2\times \pp^2$ of multi-degree $(0,1)$
and let $F_1$ and $F_2$ be two general hypersurfaces in $\pp^2\times \pp^2$ of multi-degree $(1,0)$. 
We write $H_{X,1}$ and $F_{X,1},F_{X,2}$ 
for the restriction of $H_1$ and $F_1,F_2$ to $X$. 
Then, by construction, the pair $(X,H_{X,1}+F_{X,1}+F_{X,2})$ is a log Calabi--Yau pair. 
Let $\pi_1 \colon X \rightarrow \pp^2$ be the restriction to $X$ of the projection onto the first component of $\pp^2\times \pp^2$.
Let $L_1$ and $L_2$ be the images of $F_{X,1}$ and $F_{X,2}$ via $\pi_1$. 
Therefore, the pair induced on $\pp^2$ by applying the canonical bundle formula for 
$(X,H_{X,1}+F_{X,1}+F_{X,2})$ equals $(\pp^2,L_1+L_2+\frac{1}{2}Q)$ where $Q$ is a quadric in $\pp^2$.
Let $\pi_S\colon S \rightarrow \pp^2$ be the blow-up of $L_1\cap L_2$
and let $\pi_Y\colon Y  \rightarrow X$ be the blow-up
of $F_{X,1}\cap F_{X,2}$. 
Therefore, we have
\[
\pi_S^*(K_{\pp^2}+L_1+L_2+\frac{1}{2}Q)=
K_S+L_{S,1}+L_{S,2}+\frac{1}{2}Q_S+E_S
\]
where the subscript $S$ indicates strict transform on $S$ and $E_S$ is the exceptional divisor.
Analogously, we can write
\[
\pi_Y^*(K_X+H_{X,1}+F_{X,1}+F_{X,2})=
K_Y+H_{Y,1}+F_{Y,1}+F_{Y,2}+E_Y,
\]
where the subscript $Y$ indicates strict
transform on $Y$ and $E_Y$ is the exceptional divisor. 
Then, we have a commutative diagram of crepant morphisms between Calabi--Yau pairs 
\[
\xymatrix{
(Y,H_{Y,1}+F_{Y,1}+F_{Y,2}+E_Y) \ar[r]^-{\pi_Y}
\ar[d]^-{\phi_Y}
& 
(X,H_{X,1}+F_{X,1}+F_{X,2})\ar[d]^-{\pi_1} \\ 
\left(S,L_{S,1}+L_{S,2}+\frac{1}{2}Q_S+E_S\right) \ar[r]^-{\pi_S} \ar[d]^-{\phi}
& \left(\pp^2,L_1+L_2+\frac{1}{2}Q\right)
 \\ 
(\pp^1,\{0\}+\{\infty\}) & \\
}
\]
where $\{0\}$ and $\{\infty\}$ are the images
on $\pp^1$ of $L_{S,1}$ and $L_{S,2}$.
The general fiber of the fibration 
$\phi\circ \phi_Y$ is a del Pezzo surface of degree $5$. 
Therefore, by~\cite[Theorem 1.7]{MY24}, we conclude that the pair 
$(Y,H_{Y,1}+F_{Y,1}+F_{Y,2}+E_Y)$ is of cluster type.
Therefore, the pair
$(X,H_{X,1}+F_{X,1}+F_{X,2})$ is of cluster type by Lemma~\ref{lem:cluster-type-descends}.
\end{proof} 

\begin{proposition}\label{prop:(1,1)-in-p2xp2}
Let $X$ be a smooth Fano threefold which is a hypersurface of multi-degree $(1,1)$ in $\pp^2\times \pp^2$.
Then, the variety $X$ is of cluster type. 
\end{proposition} 

\begin{proof}
Let $H_1$ and $H_2$ be two general divisors in $\pp^2\times\pp^2$ of multi-degree $(1,0)$.
Let $F_1$ and $F_2$ be two general divisors
in $\pp^2\times \pp^2$ of multi-degree $(0,1)$. 
Let $H_{X,1}$ and $H_{X,2}$ be the restrictions of $H_1$ and $H_2$ to $X$.
Let $F_{X,1}$ and $F_{X,2}$ be the restrictions of $F_1$ and $F_2$ to $X$.
By construction, the pair $(X,H_{X,1}+H_{X,2}+F_{X,1}+F_{X,2})$ is a log Calabi--Yau pair of index one and complexity one.
Then, by Theorem~\ref{thm:complexity-one-is-cluster-type}, we conclude that $X$ is of cluster type.
\end{proof} 

\begin{proposition}\label{prop:(1,1,2)-multi}
Let $X$ be a smooth hypersurface
of multi-degree $(1,1,2)$ in $\pp^1\times \pp^1\times \pp^2$.
Then, the variety $X$ is of cluster type. 
\end{proposition} 

\begin{proof}
Consider $H_1,H_2$, and $H_3$ to be general hypersurfaces of multi-degree
$(1,0,0)$, $(0,1,0)$, and $(0,0,1)$, respectively. 
Let $H_{X,i}$ be the restriction of $H_i$ to $X$.
Therefore, the pair $(X,H_{X,1}+H_{X,2}+H_{X,3})$ is a log Calabi--Yau pair
of index one, coregularity zero, and complexity three.
Let $\pi\colon X\rightarrow \pp^2$ be the projection
induced into the third component of $\pp^1\times\pp^1\times\pp^2$.
Note that $H_{X,1}$ and $H_{X,2}$ are horizontal over $\pp^2$.
Thus, by Proposition~\ref{prop:fibration-to-smooth-surface}, we conclude that the pair $(X,H_{X,1}+H_{X,2}+H_{X,3})$ is of cluster type. 
\end{proof} 

\begin{proposition}\label{prop:blow-up-p1xp2}
Let $X$ be a smooth Fano threefold which 
is a blow-up of $\pp^1\times \pp^2$ along a curve $C$ 
of bi-degree $(5,2)$ such that the composition
$C\hookrightarrow \pp^1\times \pp^2\rightarrow \pp^2$ is an embedding. Then, the variety $X$ is of cluster type. 
\end{proposition} 

\begin{proof}
Let $j\colon C\hookrightarrow \pp^1\times \pp^2 \rightarrow \pp^2$ be the composition, which by assumption, is an embedding. 
Therefore, $j(C) \subset \pp^2$ is a smooth conic. 
Hence, if $L$ is a general line in $\pp^2$, then the pair $(\pp^2,j(C)+L)$ is of cluster type. 
Let $\pi_1\colon \pp^1\times \pp^2 \rightarrow \pp^1$
be the projection onto the first component 
and $\pi_2 \colon \pp^1\times \pp^2\rightarrow \pp^2$ 
be the projection onto the second component.
By Lemma~\ref{lem:product-cluster-type}, we know that the pair
\[
(\pp^1\times\pp^2, \pi_1^*(\{0\}+\{\infty\}) +
\pi_2^*(j(C)+L)) 
\]
is of cluster type. 
Hence, by Lemma~\ref{lem:blow-up-of-cluster-type}, we conclude that $X$ is of cluster type.
\end{proof}

\begin{proposition}\label{prop:blow-up-of-(1,1)-in-P2xP2}
Let $X$ be a hypersurface of multi-degree $(1,1)$ in $\pp^2\times \pp^2$. Let $H_1$ and $H_2$ be two elements with smooth intersection of 
$\left|-\frac{1}{2}K_X\right|$ and let $C$ be their intersection. Let $Y$ be the smooth Fano threefold obtained by blowing up $X$ along $C$. 
Then, the variety $Y$ is of cluster type. 
\end{proposition} 

\begin{proof}
Let $F_1$, $F_2$, and $F_3$ be three hypersurfaces
in $\pp^2\times \pp^2$ of multi-degree $(1,1)$, $(1,0)$, and $(0,1)$, respectively.
Note that the log Calabi--Yau pair $(\pp^2\times \pp^2, X +F_1+F_2+F_3)$
has coregularity zero.
Indeed, $F_2\cap F_3\simeq \pp^1\times \pp^1$ 
and the restriction of both $X$ and $F_1$ to 
this surface must intersect. 
Hence, the pair has a simple normal crossing
zero dimensional log canonical center.
If we denote by $F_{X,i}$ the restriction 
of $F_i$ to $X$, we obtain 
a log Calabi--Yau pair
$(X,F_{X,1}+F_{X,2}+F_{X,3})$ of index one,
coregularity zero, and complexity two.
We choose a smooth $H_1$ in the pencil defining $C$. Such a choice exists because $\mathcal I_C(H)$ is generated by the two pencil sections. The differentials are independent along $C$, and Bertini applies outside $C$. Lift this $H_1$ to $F_1$, then we choose $F_2$ and $F_3$ generally. Bertini on the smooth $H_1$ gives an lc boundary and a transverse zero-dimensional stratum, so all the preceding choices are available for this $C$.
Let $\pi\colon X\rightarrow \pp^2$ be the induced projection onto either factor of $\pp^2\times \pp^2$.
Note that $F_{X,1}$ is horizontal over $\pp^2$.
On the other hand, either $F_{X,2}$ or $F_{X,3}$ 
must be horizontal over $\pp^2$ as well. 
Therefore, the boundary of log Calabi--Yau pair
$(X,F_{X,1}+F_{X,2}+F_{X,3})$ of coregularity zero 
has two prime components which are horizontal over $\pp^2$.
By Proposition~\ref{prop:fibration-to-smooth-surface}, we conclude that the pair $(X,F_{X,1}+F_{X,2}+F_{X,3})$
is of cluster type.
Since $F_{X,1}$ contains $C$, by Lemma~\ref{lem:blow-up-of-cluster-type}, we conclude that $Y$ is of cluster type. 
\end{proof} 

\begin{proposition}\label{prop:3-8}
Let $\alpha\colon {\rm Bl}_p \pp^2\rightarrow \pp^2$ be the blow-up at $p$. 
Consider the product ${\rm Bl}_p \pp^2\times \pp^2$
with its projections $\pi_1$ and $\pi_2$.
Let $X$ be a smooth Fano threefold which is a divisor in the linear system 
\[
|(\alpha \circ \pi_1)^*\mathcal{O}_{\pp^2}(1) \otimes \pi_2^*\mathcal{O}_{\pp^2}(2)|.
\]
Then, the variety $X$ is of cluster type. 
\end{proposition}

\begin{proof}
Set $S:={\rm Bl}_p \pp^2$ and 
let $T:= S \times \pp^2$.
Let $\ell$ be the strict transform on ${\rm Bl}_p \pp^2$
of a line in $\pp^2$ passing through $p$
and let $e$ be the exceptional curve of $\alpha$.
Therefore, we have 
\[
X\in |\pi_1^*\mathcal{O}_S(\ell+e) \otimes \pi_2^*\mathcal{O}_{\pp^2}(2)|. 
\]
Note that $\omega_S^\vee\simeq\mathcal O_S(2e+3\ell)$, hence, we have 
\[
\omega_X^\vee \simeq \left(\pi_1^*\mathcal{O}_S(2\ell+e)\otimes \pi_2^*\mathcal{O}_{\pp^2}(1)\right)|_X.
\]
Let $H_1$ be product with $\pp^2$ of a general element in $|\mathcal{O}_S(\ell+e)|$
and let $H_2$ be the product with $\pp^2$ of a general element in $|\mathcal{O}_S(\ell)|$. 
Let $F_1$ be the product with $S$ of a general line in $\pp^2$.
Therefore, by construction, we have a log Calabi--Yau pair
\[
(T,X+H_1+H_2+F_1) 
\]
of index one and coregularity zero.
This pair has coregularity zero as 
the three divisors $H_1,H_2$ and $F_1$ intersect transversally inside $X$. 
Thus, the pair obtained by adjunction
\[
(X,H_{X,1}+H_{X,2}+F_{X,1}) 
\]
is a log Calabi--Yau pair of index one and coregularity zero.
Let $\pi_2\colon X \rightarrow \pp^2$ be the induced fibration onto $\pp^2$.
Note that both $H_{X,1}$ and $H_{X,2}$ are horizontal over $\pp^2$.
Thus, we can use Proposition~\ref{prop:fibration-to-smooth-surface} to
conclude that $X$ is of cluster type. 
\end{proof} 

\begin{proposition}\label{prop:3-9}
Let $X$ be a smooth Fano threefold which is the blow-up of the cone
over the Veronese embedding of $\pp^2$ in $\pp^5$ with center
the disjoint vertex and a quartic curve in $\pp^2$. 
Then, the variety $X$ is of cluster type. 
\end{proposition} 

\begin{proof}
The blow-up of the vertex of the cone over the Veronese embedding
is a toric variety $T$. The quartic curve in $\pp^2$ is contained
in the torus invariant boundary of $T$ and so Lemma~\ref{lem:blow-up-of-cluster-type}
implies that $X$ is of cluster type. 
\end{proof}

\begin{proposition}\label{prop:3-13}
Let $X$ be a divisor in $\pp^2\times \pp^2$ of bi-degree $(1,1)$.
Let $C$ be a curve in $X$ of bi-degree $(2,2)$ such that the projection onto
either $\pp^2$ factor gives an embedding.
Let $Y$ be the smooth Fano threefold ${\rm Bl}_C X$.
Then, the variety $Y$ is of cluster type. 
\end{proposition}

\begin{proof}
Let $j\colon C\hookrightarrow \pp^2$ be the induced embedding into the first component of $\pp^2\times \pp^2$. 
Let $H_1$ and $H_2$ be two general lines in $\pp^2$. Therefore, we have a log
Calabi--Yau pair 
\[
(\pp^2\times \pp^2, X + \pi^*_1j(C) + \pi_2^*H_1+\pi_2^*H_2)
\]
of index one and coregularity zero. 
Let $H_{X,1}$ and $H_{X,2}$ be the restriction of $\pi_2^*H_1$ and $\pi_2^*H_2$ to $X$. Let $F$ be the restriction of $\pi^*_1j(C)$ to $X$.
By construction, the log Calabi--Yau pair
$(X,H_{X,1}+H_{X,2}+F)$ has index one and coregularity zero.
On the other hand, the surface $F$ contains the curve $C$.
Note that we have an induced projection $\pi\colon X\rightarrow \pp^2$
such that $H_{X,1}$ and $H_{X,2}$ are horizontal over $\pp^2$. 
Thus, by Proposition~\ref{prop:fibration-to-smooth-surface}, we conclude
that the pair $(X,H_{X,1}+H_{X,2}+F)$ is of cluster type.
As $F$ contains the curve $C$, Lemma~\ref{lem:blow-up-of-cluster-type} then 
implies that $Y$ is of cluster type. 
\end{proof} 

\begin{proposition}\label{prop:3-17}
Let $X$ be a smooth Fano threefold which is a divisor
of multi-degree $(1,1,1)$ in $\pp^1\times \pp^1 \times \pp^2$.
Then, the variety $X$ is of cluster type. 
\end{proposition} 

\begin{proof}
Let $F,G$ and $H_i$ be general divisors in $\pp^1\times \pp^1\times \pp^2$
of multi-degree $(1,0,0)$, $(0,1,0)$, and $(0,0,1)$, respectively, for $i\in \{1,2\}$. 
Therefore, the pair $(\pp^1\times\pp^1\times\pp^2,X+F+G+H_1+H_2)$ is log Calabi--Yau of index one and coregularity zero.
Let $(X,F_X+G_X+H_{X,1}+H_{X,2})$ be the log Calabi--Yau pair induced by adjunction
and let $X\rightarrow \pp^1\times \pp^1$ be the fibration induced by projecting into the first two components of $\pp^1\times\pp^1 \times \pp^2$.
Note that $H_{X,1}$ and $H_{X,2}$ are horizontal over $\pp^1\times \pp^1$.
Therefore, Proposition~\ref{prop:fibration-to-smooth-surface} implies that $X$ is of cluster type. 
\end{proof} 

\begin{proposition}\label{prop:3-21}
Let $X$ be a smooth Fano threefold which is the blow-up of $\pp^1\times \pp^2$
along a curve $C$ of multi-degree $(2,1)$. Then, the variety $X$ is of cluster type.
\end{proposition}

\begin{proof}
Note that the curve $C$ is contained in a torus invariant divisor of $\pp^2\times\pp^1$ so the statement follows from Lemma~\ref{lem:blow-up-of-cluster-type}. 
\end{proof} 

\begin{proposition}\label{prop:3-22}
Let $X$ be a smooth Fano threefold which is the blow-up of $\pp^2\times \pp^1$
along a conic $C$ in $\pp^2\times \{x\}$ for $x\in \pp^1$. Then, the variety $X$ is of cluster type. 
\end{proposition}

\begin{proof}
Note that the curve $C$ is contained in a torus invariant divisor of
$\pp^2\times \pp^1$ so the statement follows from Lemma~\ref{lem:blow-up-of-cluster-type}.
\end{proof}

\begin{proposition}\label{prop:4.1}
Let $X$ be a smooth Fano threefold in $(\pp^1)^4$ of multi-degree $(1,1,1,1)$.
Then, the variety $X$ is of cluster type. 
\end{proposition} 

\begin{proof}
Let $H_i$ be a hypersurface in $(\pp^1)^4$ of multi-degree $e_i$ with $i\in\{1,\dots,4\}$. 
Let $H_{X,i}$ be the restriction of $H_i$ to $X$.
Therefore, the pair
$(X,H_{X,1}+H_{X,2}+H_{X,3}+H_{X,4})$ 
is a log Calabi--Yau pair
of index one and coregularity zero.
Let $\pi\colon X\rightarrow \pp^1\times \pp^1$ be the restriction
of the projection on the last two components of the product $(\pp^1)^4$.
Note that both divisors $H_{X,1}$ and $H_{X,2}$ are horizontal over $\pp^1\times \pp^1$. Therefore, by Proposition~\ref{prop:fibration-to-smooth-surface},
we conclude that $X$ is a cluster type variety. 
\end{proof} 

\begin{proposition}\label{prop:4.2}
Let $X$ be a smooth Fano threefold which
is the blow-up of the cone over a smooth quadric in $\pp^3$
along the disjoint union of the vertex and an elliptic curve on the quadric. 
Then $X$ is of cluster type.
\end{proposition} 

\begin{proof}
Let $T$ be the blow-up of the cone over a smooth quadric in $\pp^3$ at the vertex. 
The variety $T$ is toric. Furthermore, the toric boundary contains the elliptic curve on the quadric. Therefore, by Lemma~\ref{lem:blow-up-of-cluster-type}, we conclude that $X$ is of cluster type. 
\end{proof}

\begin{proposition}\label{prop:4.3-13}
Let $X$ be a smooth Fano threefold which is the blow-up
of $(\pp^1)^3$ along a curve $C$ of degree $(1,1,n)$ where $n\in \{2,3\}$. 
Then, the variety $X$ is of cluster type.
\end{proposition} 

\begin{proof}
Let $\pi_1 \colon (\pp^1)^3 \rightarrow \pp^1\times \pp^1$ be the projection onto the first two components and $\pi_2 \colon (\pp^1)^3 \rightarrow \pp^1$ be the projection onto the last component. Then, the image of $C$ on $\pp^1\times \pp^1$ is a curve of degree $(1,1)$. Denote this curve by $\ell$.
Note that the pair $(\pp^1\times \pp^1, \ell+F_1+F_2)$ is cluster type, where
$F_1$ and $F_2$ are general fibers of the projections onto the first and second
components of $\pp^1\times\pp^1$, respectively.
On the other hand, the pair $(\pp^1,\{0\}+\{\infty\})$ is toric. 
Thus, we conclude that the pair 
\[
((\pp^1)^3, \pi_1^*(\ell+F_1+F_2)+\pi_2^*(\{0\}+\{\infty\})) 
\]
is of cluster type by Lemma~\ref{lem:product-cluster-type}.
Since the boundary divisor $\pi_1^*\ell$ contains the curve $C$, we can then
use Lemma~\ref{lem:blow-up-of-cluster-type} to conclude that $X$ is of cluster type as well.
\end{proof} 

In the following example, we show that Proposition~\ref{prop:fibration-to-smooth-curve} does not apply to deduce that a smooth Fano threefold of Picard rank $3$ and volume $14$ is of cluster type. 

\begin{example} 
\label{example:family3.2}
{\em 
Consider the toric bundle $T=\mathbb{P}(\mathcal{O}\oplus \mathcal{O}(-1,-1)^{\oplus 2})$ over $\pp^1\times \pp^1$.
Let $\mathcal{L}$ be the tautological line bundle of $T\rightarrow \pp^1\times \pp^1$.
Let $X$ be a smooth Fano threefold which is a general element of $|\mathcal{L}^2\otimes \mathcal{O}(2,3)|$.

Let $\pi\colon T \rightarrow \pp^1\times \pp^1$ be the projection.
Note that $\mathcal{O}_T(-K_T-X) \simeq \mathcal{L}\otimes \pi^* \mathcal{O}(2,1)$. 
We consider $H$ to be a general element of $\mathcal{L}$.
Let $F_1$ and $F_2$ be the pull-back to $T$ of two general fibers
of the projection into the first factor $\pi_1\colon \pp^1 \times \pp^1 \rightarrow \pp^1$.
Let $G$ be the pull-back to $T$ of a general fiber of 
the projection into the second factor 
$\pi_2\colon \pp^1 \times \pp^1 \rightarrow \pp^1$. 
Therefore, the pair
$(T,X+H+F_1+F_2+G)$ is a log Calabi--Yau pair of coregularity zero. We let $H_X$ (resp. $F_{X,1}$, $F_{X,2}$, and $G_X$) 
be the restriction of $H$ (resp. $F_1$, $F_2$, and $G$) 
to $X$. Therefore, the pair
$(X,H_X+F_{X,1}+F_{X,2}+G_X)$ is a log Calabi--Yau pair
of coregularity zero and index one. 
Note that this pair has complexity two. 
We have a fibration $\psi \colon X\rightarrow \pp^1$ induced by 
the projection of $\pp^1\times\pp^1$ onto the first component.
The log Calabi--Yau pair induced on $\pp^1$ via $\psi$ 
is $(\pp^1,\{0\}+\{\infty\})$. On the other hand, a general fiber
of $\psi$ is a del Pezzo surface of degree $3$. Therefore, we cannot apply Proposition~\ref{prop:fibration-to-smooth-curve} to conclude that $(X,H_X+F_{X,1}+F_{X,2}+G_X)$ is cluster type.
}
\end{example}

Despite Example \ref{example:family3.2}, we can find a cluster type boundary on a smooth Fano threefold of Picard rank 3 and volume 14.

\begin{proposition}\label{prop:3-2_cluster_type}
Let $X$ be a smooth Fano threefold given by a smooth element of $|\mathcal{L}^2\otimes \mathcal{O}(3,2)|$ in the $\pp^2$-bundle $T=\mathbb{P}(\mathcal{O}\oplus \mathcal{O}(-1,-1)^{\oplus 2})$ over $\pp^1\times \pp^1,$ where $\mathcal{L}$ is the tautological quotient line bundle. Then, the variety $X$ is of cluster type.
\end{proposition}

\begin{proof}
    We follow the construction in the proof of \cite[Lemma 5.4]{LMV24}, which we recall for the reader's convenience.

    We use the quotient convention throughout.
    Thus we may see $X$ as a smooth divisor $X\sim 2M + F_1$ in the $\pp^2$-bundle $T := \pp_{\pp^1\times \pp^1}(\oo^{\oplus 2} \oplus \oo(1,1)),$ where $M$ is the tautological divisor and $F_1$ is the pullback of a divisor of bidegree $(1,0)$ in $\pp^1 \times \pp^1$. Denote by $F_2$ the pullback of a divisor of bidegree $(0,1)$ in $\pp^1 \times \pp^1$. By adjunction we get $-K_X\sim M+F_2$. A general fiber $F_1$ is a smooth del Pezzo surface of degree $6$, while $F_2$ is a smooth del Pezzo surface of degree $3$.
	
	There exists a divisor $L \sim M - F_1 - F_2$ on $X$ such that $L \cong \pp^1 \times \pp^1$. The linear system $|M|$ induces a morphism $\phi\colon T \to \pp^5$ such that the restriction to $X$ contracts $L$ to $L_0\cong \pp^1$ along one of its rulings. We then have $ g := \phi|_X \colon X\to Y \subset \pp^5$, where $Y$ is a smooth subvariety of degree $5$ that admits a quadric bundle $f\colon Y \to \pp^1$. Let $F$ be the class of a fiber of $f$, and $H$ the restriction of a hyperplane section of $\pp^5$.
	
	Let $C$ be the image of a general member of $|F_2|$, that is, $C \in |H-F|$. Such cubic surface contains $L_0$ as a line. Choose another line $L_1$ in $C$ disjoint from $L_0$. Then the normal bundle $N_{L_1/Y} \cong \oo \oplus\oo(-1)$. Let $\psi\colon W:= {\rm Bl}_{L_1} Y \to Y$ be the blow-up of $L_1$, with exceptional divisor $E$. Then the linear system $|\psi^* H - E|$ induces a birational morphism $\pi\colon W \to \pp^3$. This construction may be summarized with the following diagram
	\[
	\xymatrix{
		& W \ar[d]_{\psi} \ar[rd]^{\pi} \\
		X \ar[r]^{g}	& Y \ar@{-->}[r] & \pp^3.
	}
	\]

    We now construct a boundary $B$ on $\pp^3$ such that its log pullback $(X,B_X)$ is a cluster type pair.

    Let $A:= \psi^* H - E$ and $C_W$ be the strict transform of $C$ in $W$. Then $C_W \sim \psi^*(H-F) - E$. On $Y$ we have the intersection numbers
	\[
		H^3 = 5,\ H^2\cdot F = 2,\ H\cdot F^2 = F^3 = 0, \quad \text{ and } \quad\ H\cdot L_1=F\cdot L_1=1. 
	\]
	Additionally, on $W$ we obtain
	\[
		(\psi^*H)^2\cdot E = \psi^*H \cdot \psi^* F \cdot E = (\psi^*F)^2\cdot E = 0,\ \psi^*H\cdot E^2 = -1,\ \psi^* F\cdot E^2 = -1, \quad \text{and} \quad E^3 = 1.
	\]
	Hence,
	\begin{align*}
		A^2\cdot E &= -2\psi^*H\cdot E^2 + E^3 = 3\\
		\intertext{and}
		A^2\cdot C_W &= (\psi^* H)^3 - (\psi^* H^2)\cdot \psi^* F + 3\psi^* H\cdot E^2
						- \psi^* F\cdot E^2 - E^3 = 0.
	\end{align*}
	Therefore, the morphism $\pi$ contracts the cubic surface $C_W$ and maps $E$ to an irreducible cubic surface $S$. In fact, $C_W$ is contracted to a line $\ell \subseteq \pp^3$, which is consistent with the fact that the map $\pi\circ \psi^{-1}$ corresponds to the projection from the line $L_1$ in $\pp^5$ \cite[Lemma 5.4]{LMV24}.

    Let $\Pi \subseteq \pp^3$ be a general plane that does not contain $\ell$ and the centers of exceptional divisors of $\pi$. Let $P:= \pi^* \Pi$. Hence, $P$ is a general member of $|A|$, and Bertini implies that it is smooth. Recall that $A = \psi^* H - E$, so the elements of $|A|$ are the strict transforms of restriction of hyperplanes in $\pp^5$ containing $L_1$. Therefore, there exists $H_0 \in |H|$ such that its strict transform is $P$. We may assume that $H_0 \cap C = L_1 \cup Q$, where $Q$ is a smooth conic that intersects $L_1$ transversally at two points. \\ 
    
    \textit{Claim:} The divisor $E + C_W + P$ is a simple normal crossing divisor.\\ 
	
	\begin{proof}[Proof of the Claim]
		The divisor is simple normal crossing over the points away from $L_1 \cap Q$ by Bertini. For a general $H_0$, it intersects $C$ transversally at every point of $(L_1\cup Q)\setminus(L_1\cap Q)$.
        Outside $L_1$ the blow-up is an isomorphism and $C_W + P = C + H_0$, which is simple normal crossing along $Q\setminus L_1$ by the transversality above.
        Let $y_0 \in L_1 \setminus Q$. Since $C$ and $H_0$ are transverse at $y_0$ and $C\cap H_0 = L_1$ near $y_0$, we may pick local coordinates $(x,y,z)$ such that $C=\{x=0\}$, $H_0=\{y=0\}$ and $L_1=\{x=y=0\}$. On the chart $x=uy$ of the blow-up of $(x,y)$ we get $E=\{y=0\}$, $C_W=\{u=0\}$, and $P$ does not appear, since the pullback of $H_0$ is $E$. On the chart $y=vx$ we get $E=\{x=0\}$, $P=\{v=0\}$, and $C_W$ does not appear. In both charts the components are coordinate hyperplanes, so $E+C_W+P$ is simple normal crossing over a neighborhood of $y_0$.
        Over a point of $L_1 \cap Q$, because $H_0|_C = Q + L_1$, we may pick local coordinates $(x,y,z)$ such that $C = \{x = 0\}$, $L_1 = \{x = y = 0\}$ and $H_0 = \{x + yz = 0\}$. Then, on the chart $x = uy$ of the blow-up of $(x,y)$ we get $E = \{y=0\}$, $C_W = \{u = 0\}$ and $P = \{u + z=0\}$, which intersect transversally at the point $(0,0,0)$. On the chart $y=vx$, we get $E =\{x=0\}$ and $P = \{1 + vz = 0\}$, with $C_W$ not appearing. As $(v,z) \neq (0,0)$ on $P,$ the divisor $E+P$ is simple normal crossing. Therefore, $E + C_W + P$ is a simple normal crossing divisor.
	\end{proof}
	
	The claim implies that $(\pp^3,S+\Pi)$ is an index-one log Calabi--Yau pair of coregularity zero. At the point $p:=\ell\cap\Pi$, 
    the boundary $S+\Pi$
    has multiplicity $\geq 3$, and log canonicity gives multiplicity $\leq 3$, so Proposition \ref{prop:criterion-sing-p3} implies that the pair  $(\pp^3,S + \Pi)$ is of cluster type. Because $\pi^*(K_{\pp^3} + S+\Pi) = K_W + E + P + C_W$, the morphism $\pi$ only extracts log canonical and canonical places, so $(W,  E + P + C_W)$ is of cluster type by Lemma \ref{lem:blow-up-of-cluster-type}. We have that $\psi^* C = C_W + E$ and $\psi^* H_0 = P + E$, so $\psi^*(K_Y + C + H_0) = K_W + C_W + P + E$. Lemma \ref{lem:cluster-type-descends} implies that $(Y , C + H_0)$ is of cluster type. Finally, by generality of $\Pi$, we may assume that $L_0\not \subseteq H_0$. Then, $g^* C = C_X + L$ and $g^* H_0 = H_X$, where $C_X$ and $H_X$ are the strict transforms of $C$ and $H_0$, respectively. We compute 
	\[
		g^*(K_Y + C + H_0) = (K_X -L) + (C_X + L) + H_X = K_X + H_X + C_X.
	\]
	Therefore, $g$ only extracts canonical places of $(Y,C+H_0)$ and Lemma \ref{lem:blow-up-of-cluster-type} implies that the pair $(X,C_X + H_X)$ is of cluster type.
\end{proof}

Finally, we study the cluster type condition for a smooth Fano threefold which is not a toric hypersurface, but instead a complete intersection in a toric variety.

\begin{proposition}\label{prop:3-24}
Let $X$ be a smooth Fano threefold which is an intersection
of divisors of multi-degree $(1,1,0)$ and $(0,1,1)$ in $\pp^1\times\pp^2\times\pp^2$.
Then, the variety $X$ is of cluster type. 
\end{proposition} 

\begin{proof}
Let $F,G$ and $H_i$ be general divisors of multi-degree
$(1,0,0)$, $(0,1,0)$, and $(0,0,1)$ in $\pp^1\times \pp^2\times\pp^2$, for $i\in \{1,2\}$.
Let $F_X$, $G_X$, and $H_{X,i}$ be the restrictions of $F,G$ and $H_i$, respectively.
Then, we have a log Calabi--Yau pair
\begin{equation}\label{eq:cy-3-24}
(X,F_X+G_X+H_{X,1}+H_{X,2}). 
\end{equation}
We argue that this Calabi--Yau pair has coregularity zero. 
Indeed, $F\cap G\cap H_1\simeq \pp^1\times \pp^1 =:S$ 
and the intersection of $S$ with $X$ is the intersection
of a curve of multi-degree $(1,0)$ and a curve of multi-degree $(1,1)$.
Thus, the surfaces $F_X,G_X$, and $H_{X,1}$ intersect at a point in $X$. 
Hence, the pair~\eqref{eq:cy-3-24} has coregularity zero. 
Let $\pi_2\colon X \rightarrow \pp^2$ be the fibration induced by projecting
$X$ into the second component of $\pp^1\times \pp^2\times \pp^2$.
Note that $H_{X,1}$ and $H_{X,2}$ are both horizontal over $\pp^2$. 
Hence, by Proposition~\ref{prop:fibration-to-smooth-surface}, we conclude that the variety $X$ is of cluster type.
\end{proof}

\section{Special varieties in the Mori-Mukai family \texorpdfstring{\textnumero}{No.}\texorpdfstring{\,}{ }2.4}
\label{sec:special-2.4}

In this section we will show that some special blow-ups of $\pp^3$ along the intersection of two cubic surfaces are not of cluster type. 

\begin{theorem}\label{thm:2.4-not-cluster-type}
Let $C$ be the intersection of $\{x^3+y^3+z^3+w^3=0\}$ and $\{a_0x^3+a_1y^3+a_2z^3+a_3w^3=0\}$ in $\pp^3$, 
where $a_i$ are pairwise distinct nonzero coefficients. The blow-up of $\pp^3$ along $C$ is a Fano variety in  \textnumero\,2.4 which is not of cluster type.
\end{theorem}

Before we prove the main theorem of this section, we will start with some statements that hold for every smooth complete intersection of two cubic surfaces.
The strategy of the proof of the next proposition will be further developed and used in Section~\ref{sec:MM-1-6} (see, e.g., Lemma~\ref{lem:movable-system}).

\begin{proposition}\label{prop:2.4-smooth-cubic}
Let $C$ in $\pp^3$ be a smooth complete intersection of two cubic surfaces. Let $X$ be the blow-up of $\pp^3$ along $C$. Let $B$ be the strict transform in $X$ of the sum of a plane and a smooth cubic, containing $C$, in $\pp^3$. Then, the pair $(X,B)$ is not of cluster type.
\end{proposition}

\begin{proof}
Let $S$ be a smooth member in the pencil generated by the cubic surfaces defining $C$, and $P$ be any plane in $\pp^3$. Set $B = S_X + P_X$ as in the statement. For the sake of contradiction, assume that $(X,B)$ is of cluster type. So, in particular, the log Calabi--Yau pair $(\pp^3,S+P)$ is log rational. Let $\Delta \in |-K_S|$ be a reduced nodal plane section which is not smooth. Consider $R$ an effective $\qq$-divisor such that $R \sim_{\qq} b\Delta$, for some $b\geq 0$. One can view $\Delta$ as the tangent plane section at its nodal points, that is, $\Delta = T_p \coloneqq T_pS \cap S$ for all $p \in \text{Sing}(\Delta)$. For $p$ a nodal point of $\Delta$, one has that $(S,T_p)$ is log canonical at $p$. If $b> 0$, by the contrapositive of \cite[Corollary 1.13]{CPW16}, the pair $(S, \frac{1}{b}R)$ is log canonical at $p$. For $p$ a smooth point of $\Delta$, we have that the pair $(S,\Delta)$ is plt near $p$, and so any place  over $p$ is not strictly log canonical. Thus, we obtain the inequality 
\begin{equation*}
v(R) \leq b v(\Delta)
\end{equation*} for every log canonical place $v$ of $(S,\Delta)$. Consequently, for $0\leq \epsilon \ll 1 $, the pair 
\begin{equation}\label{lc pair from ineq lc places}
(S,(1-\epsilon b)\Delta+\epsilon R)
\end{equation}
is log canonical. 

Let $D \subset X$ be a prime divisor different from $S_X,P_X$ and $E$, where the latter is the exceptional divisor of the blow-up $\varphi \colon X  \rightarrow \pp^3$. Its image in $\pp^3$ is a surface of degree $d$ and multiplicity $m \geq 0$ along $C$. Set $H$ to be the pullback of a plane section in $\pp
^3$. Thus, we have the following divisor class relations $$D \sim d H - mE = dH - 3mH + 3mH -mE = m(3H-E) +(d-3m)H.$$

Restricting to $S_X$, we have that $R \coloneqq D|_{S_X} \sim -(d-3m)K_S$. Since the class $R$ is effective and $-K_S$ is ample, we have $d-3m \geq 0$.

Note that $P \cap S$ is a reduced nodal plane cubic. If it is smooth, then the coregularity of the log Calabi-Yau pair $(\pp^3,S+P)$ is one, which contradicts log rationality. Thus, plugging $\Delta = P|_S$ and $R=D|_{S_X}$ under $S_X\simeq S$ in \eqref{lc pair from ineq lc places}, we get that the pair $(S,(1-\epsilon b)P|_S+\epsilon R)$ is log canonical for $0 \leq \epsilon \ll 1$, where $b \coloneqq d- 3m$. Since $\varphi$ is crepant, its restriction to log canonical places is a crepant map between the induced log Calabi-Yau pair structures on them \cite[Theorem 4.16]{K13}. So, the pair $(S_X, (1-\epsilon b )P_X|_{S_X} + \epsilon D|_{S_X})$ is log canonical. Inversion of adjunction \cite[Theorem]{Kaw07} implies that
\begin{equation}\label{pair inv of adj}
(X,S_X + (1 - \epsilon b)P_X + \epsilon D )
\end{equation}
is log canonical near $S_X$ for $0 \leq \epsilon \ll 1$.

If $v$ is a log canonical place of $(X,B)$ whose center is contained in $S_X$, its discrepancy with respect to the pair~\eqref{pair inv of adj} is precisely $ \epsilon (bv(P_X)-v(D))$. Therefore, one obtains 
\begin{equation}\label{ineq lc places}
v(D) \leq b v(P_X).
\end{equation}

The remaining lc place is $P_X$ itself. Indeed, at a point of $C\cap P$ the plane transform has local equation $xy+z^r=0$, where $r$ is the intersection multiplicity, and thus has only Du Val singularities. Hence $(X,P_X)$ is plt, and $(X,S_X+P_X)$ is plt away from $S_X$.

Set $U = X \setminus B$. Notice that we can write $ \text{Pic}(X)=\text{Cl}(X) = \zz  H \oplus \zz  E = \zz  H \oplus \zz  (3H - E)$. We have the following exact sequence: $$\zz  S_X \oplus \zz  P_X \longrightarrow\Cl(X) \longrightarrow\Cl(U) \longrightarrow 0.$$

As the first arrow is an isomorphism, we get that $\Cl(U) = 0$. The cluster type condition on $X$ implies the existence of an embedding $\mathbb{T} \coloneqq \mathbb{G}_m^3 \hookrightarrow U$ by \cite[Theorem 16.2.3]{Cor25} and \cite[Theorem 1.3(4)]{EFM24}. Let $D_1,\dots,D_r$ be the prime divisorial
components of $U\setminus \mathbb{T}$. We have the following exact sequence:
\begin{equation}\label{eq:LES}
0\longrightarrow\oo(U)^\times/\cc^\times\longrightarrow\oo(\mathbb{T})^\times/\cc^\times
\xrightarrow{\ \ddivv_U\ }\bigoplus_{i=1}^{r}\zz[D_i]\longrightarrow\Cl(U)
\longrightarrow\Cl(\mathbb{T})\longrightarrow0.
\end{equation}

We claim that the middle arrow is an isomorphism. The surjectivity comes from the fact that $\Cl(U)=\Cl(\mathbb{T})=0$, and the injectivity comes from the fact $\oo(U)^{\times} = \cc^\times$. Indeed, for the latter, an invertible regular function $f$ on $U$ defines a rational function on $X$ whose divisor is supported on $S_X +P_X $. Thus, abusing notation, $\ddivv(f) = c S_X + dP_X \sim 0$ for some $c,d \in \zz$. Since the classes $S_X$ and $P_X$ are linearly independent in $\Cl(X)$, then $c=d=0$. Hence, $f \in  \oo(X) \simeq \cc$ and we get the asserted fact.

Therefore $r=3$, and so there exist characters $h_1,h_2,h_3$ forming a basis of the character lattice of $\mathbb{T}$ such that $\ddivv_U(h_i)=D_i$. As $U$ is smooth, hence normal, each $h_i$ can be extended to a regular function on $U$. Without loss of generality, we may assume that the closure of $D_1$ in $X$ is not $E$. By abuse of notation, still denote it by $D_1$. As $D_1$ is also different from $S_X$ and $P_X$, and writing its class in $\Cl(X)$ in the basis formed by the latter as before, one has $$ \ddivv_X(h_1) = D_1 - mS_X - bP_X,$$ where $b,m \geq 0$. 

For every log canonical place centered in $S_X$, the inequality~\eqref{ineq lc places} implies $v(h_1) \leq -mv(S_X) \leq 0$. For the remaining log canonical place $P_X$, we have $v_{P_X}(h_1)=-b\leq 0$. Thus, we obtain the following claim.\\

\textit{Claim:} For every log canonical place $v$ of $(X,B)$, we have $v(h_1)\leq 0$.\\

Let us show this leads to a contradiction. Choose a top-degree rational form $\tau$ with $\ddivv_X(\tau) = -B$.  On $\mathbb{T}$, it is therefore a nowhere zero regular form, so we may write
$$
\tau=ch_1^{k_1}h_2^{k_2}h_3^{k_3}\,
\frac{dh_1}{h_1}\wedge\frac{dh_2}{h_2}\wedge\frac{dh_3}{h_3},$$
where $c\in\cc^\times $ and $\ k_i\in\zz$.

For any divisorial valuation $w$ of the common function field of $\mathbb{T}$ and $X$, the log discrepancy with respect to $(X,B)$ is given by $1+\ord_w(\tau)$. Indeed, taking $K_X = -B = \ddivv(\tau)$ and the corresponding canonical divisor on a model where $w$ is a divisor, then $K_X+B = 0$ as a divisor.

Now, looking at the toric compactification $(\pp^1)^3$ in the coordinates $h_i$, the valuations at $h_i=0$ and $h_i=\infty$ have discrepancies $k_i$ and $-k_i$,
respectively. Log canonicity then forces $k_i=0$ for all $i$. By the Claim, we obtain that the valuation satisfying $w(h_1)=1$ and $w(h_2)=w(h_3)=0$ is a log canonical place, contradicting the inequality~\eqref{ineq lc places}. So, the cluster type property is not satisfied by the log Calabi-Yau pair $(X,B)$.

\end{proof}

\begin{lemma}\label{lem:2.4-normal-quartic}
Let $C$ in $\pp^3$ be a smooth complete intersection of two cubic surfaces. Every irreducible slc quartic surface containing $C$ is normal.
\end{lemma}

\begin{proof}
Write $C=\{F=G=0\}$ and $S=\{LF+MG=0\}$, where $L,M$ are independent linear forms, and put $\ell=\{L=M=0\}$. Let $\nu\colon\bar S\to S$ be the normalization, $H=\nu^*\mathcal O_S(1)$ and $\Delta$ its conductor. Then $K_{\bar S}+\Delta\sim0$. The non-normal locus does not contain $C$ by~\cite[Proposition 3.3]{Duc24}. Moreover $S$ is smooth generically along $\ell$, since there $dS=F\,dL+G\,dM$. The strict transforms $\bar C,\bar\ell$ map finitely and birationally to the smooth curves $C,\ell$, hence isomorphically. Comparing fixed divisors in $[F:G]=[-M:L]$ gives $\bar C\sim2H+\bar\ell$.

For $\mathcal F=\mathcal O_{\bar S}(-\bar\ell)$, the two curve ideal sequences give $\chi(\mathcal F)-\chi(\mathcal F(-2H))=\chi(\mathcal O_{\bar C})-\chi(\mathcal O_{\bar\ell})=-10$. 
Since $H$ is Cartier, $\mathcal{F}$ and $\mathcal{F}(-2H)$ have the same local correction terms in the Riemann--Roch formula. Therefore  $\chi(\mathcal F)-\chi(\mathcal F(-2H))=-2H^2-2H \cdot \bar\ell-H \cdot K_{\bar S}=-10+H\cdot \Delta$. Thus $H \cdot \Delta=0$, and ampleness gives $\Delta=0$. Since $S$ is a hypersurface, it is $S_2$, hence normal.
\end{proof}

The following lemma will be required to perform some local computations.

\begin{lemma}\label{lem:2.4-local-computations}
Let $(S;0)$ be the germ of a surface singularity in $\mathbb{A}^3_{u,v,w}$ with local equation $$q(u,v)+w^3l(u,v),$$
where $q$  is a convergent power series with no constant or linear terms and $l$ a nonzero linear form. Assume the quadratic part $q_2$ of $q$ has rank one. After a linear change of the variables $u,v$, write $q_2=u^2$ and $l=\rho u+\sigma v$,
where $\rho,\sigma\in \mathbb{C}$ are not both zero. Then, there are local analytic
coordinates $(U,v,w)$, with $\partial U/\partial u(0)=1$, in which the equation of $S$
becomes
\[
U^2+h(v)+\sigma vw^3-\frac{\rho^2}{4}w^6+R(v,w)=0.
\]
Here, $h$ vanishes to order at least three and $R \in (v^2w^3,vw^6,w^9) \subset \mathbb {C}\{v,w\}.$

Furthermore, if $(S,0)$ is an isolated log canonical singularity, then it is Du Val or simple elliptic. 
\end{lemma}

\begin{proof}
Set $t=w^3$. By the implicit function theorem, the equation $q_u(u,v)+\rho t=0$ has a convergent solution $u=a(v,t)$, with $a(v,0)=O(v^2)$ and $a_t(0,0)=-\rho/2$. Completing the square in $u$ gives an analytic coordinate $U$, with $\partial U/\partial u(0)=1$, such that in the new coordinates $U, v, w$, with $t=w^3$, the equation of $S$ is 
$$
U^2+q(a(v,t),v)+t(\rho a(v,t)+\sigma v)=0.
$$
The constant term in $t$ is a series $h(v)$ of order at least three. The coefficient of $t$ is $\sigma v+O(v^2)$, and that of $t^2$ is $-\rho^2/4+O(v)$. Substituting $t=w^3$ gives the stated form and the required ideal containing $R$.

We now prove the last assertion. Suppose $(S,0)$ is an isolated log canonical singularity.
Write $h(v)=h_3v^3+h_4v^4+O(v^5)$. If $\sigma \neq 0$ and $h_3 \neq 0$, assign weights $(9,6,4)$ to $(U,v,w)$, the terms of least weight are $U^2+h_3v^3+ \sigma vw^3$, so the singularity is Du Val of type $E_7$. If $\sigma\neq0$, $h_3=0$ and $h_4\neq0$, assign weights $(2,1,1)$ to $(U,v,w)$, the terms of least weight are  $U^2+h_4v^4+\sigma vw^3$. The binary quartic $v(h_4v^3+\sigma w^3)$ has four distinct roots, so the exceptional curve of the weighted blow-up of weights $(2,1,1)$ is smooth elliptic and the singularity is simple elliptic. If $\sigma \neq 0$ and $h_3=h_4=0$, the weighted blow-up with weights $(15,6,8)$ gives weighted order $30$, and log discrepancy $29-30=-1$, thus a non log canonical place, which is a contradiction.

Finally, if $\sigma=0$, then $\rho\neq0$. If $h_3 \neq 0$, the weights $(3,2,1)$ give the leading equation $U^2+h_3v^3-\rho^2w^6/4$, the exceptional curve of the weighted blow-up of weights $(3,2,1)$ is a smooth elliptic curve, thus the singularity is simple elliptic. If $h_3=0$, the weights $(6,3,2)$ give weighted order $12$ and yield a non log canonical place, a contradiction. 
\end{proof}

The following lemma is the main reason for our examples to not be of cluster type, although the general element in \textnumero\,2.4 is of cluster type.

\begin{lemma}\label{lem:2.4-singular-surfaces}
Let $C$ be the intersection of $\{x^3+y^3+z^3+w^3=0\}$ and $\{a_0x^3 +a_1y^3+a_2z^3+a_3w^3=0\}$ in $\pp^3$, where $a_i$ are pairwise distinct nonzero coefficients. Let $S$ be a singular irreducible surface containing $C$, then the following hold: 

\begin{enumerate}
\item If $S$ is a cubic surface, then it is a cone over a smooth plane cubic.
\item If $S$ is a quartic surface and $(\mathbb{P}^3,S)$ is log canonical, then $S$ has only canonical and simple elliptic singularities.
\end{enumerate}

\end{lemma}

\begin{proof}

If $S$ is a cubic surface, then it must be in the pencil formed by $F=\{x^3+y^3+z^3+w^3=0\}$ and $G=\{a_0x^3+a_1y^3+a_2z^3+a_3w^3=0\}$. The only singular members in that pencil are cones over smooth plane cubics.

If $S$ is an irreducible quartic surface, it must be normal by Lemma~\ref{lem:2.4-normal-quartic}. We only need to exclude isolated (degenerate) cusps, see \cite[Table 1]{Duc24}.  
Let us write $S$ as $LF+MG$, where $L$ and $M$ are linear. For a singular point $p$, let $q_2$ and $q_3$ be the quadratic and cubic terms of the local equation in affine coordinates centered at $p$. We only need to check when $q_2$ has rank at most one, since otherwise the point is an $A_n$ singularity. Furthermore, if the rank is one, to be a cusp, the restriction of $q_3$ to $q_2=0$ needs to be zero or the cube of a linear form. If the rank of $q_2$ is zero, then the projective tangent cubic $\{q_3=0\} \subset \pp^2$ cannot be smooth, or the singularity would be simple elliptic.

We use local coordinates $(u,v,w)$ and we will assume that $p$ is a cusp.\\

\textit{Case 1:} In this case, we assume that $p\in C$.\\ 

As $F(p)=G(p)=0$, the singularity condition implies $dS(p)=0$, whence $L(p)=M(p)=0$. Quadratic rank at most one forces both $L$ and $M$ to vanish on the tangent line to $C$. 
We first deal with the case where all coordinates of $p$ are nonzero. We rescale the coordinates to have $p=(1,1,1,1)$, and work in the affine chart where the fourth coordinate is $1$. After changing the cubic generators, we can write
$f=x^3-1-\alpha (t^3-1)$ and $g=y^3-1-\beta (t^3-1)$, with $\alpha,\beta$ different from each other, zero and one. Set $z=t-1$, $u=x-1-\alpha z$ and $v=y-1-\beta z$. Thus the homogeneous terms of degrees $1$ and $2$ are
$f_1=3u$,  $g_1=3v$,
$f_2=3u^2+6\alpha uz +3\alpha (\alpha-1)z^2$ and $g_2=3v^2+6\beta vz +3\beta(\beta-1)z^2$.

After possibly interchanging $u$ and $v$ we can rewrite $S=(u+kv)(f+kg)+c(vf-ug)$, where $k,c$ are constants with $c\neq 0$.

Thus $q_2=3(u+kv)^2$ and its cubic part on $u+kv=0$ is$$3cv(  k(k+1)v^2  +2k(\beta-\alpha)vz  +   (  \alpha(\alpha-1) +k\beta(\beta-1)   )z^2    ).$$ Put $a=\alpha(\alpha-1)$, $b=\beta(\beta-1)$ and $\delta=\alpha-\beta\ne0$. Vanishing or a cube would force $k\delta=a+kb=0$, hence $a=0$, a contradiction. Thus a rank-one quadratic part cannot give a cusp. If $q_2$ is of rank zero, then $S=c(vf-ug)$. After multiplying the defining equation by $\frac{1}{3c}$, its cubic part is $q_3=u^2v-uv^2+2(\alpha-\beta )uvz+\alpha(\alpha-1)vz^2-\beta(\beta-1)uz^2$. The three points on $z=0$ are smooth. In the chart $z=1$, a singular point would satisfy $v-u=\delta$, $u^2=a$ and $v^2=b$, hence $u=\frac12-\alpha$, contradicting $u^2-a=\frac14$. Thus the tangent cubic is smooth, and the singularity is simple elliptic.

If one coordinate is zero, then the others are nonzero. In this case we can perform a change of coordinates to have $p=(1,1,0)$ and write:
$f=x^3-1-\alpha w^3$ and $g=y^3-1-\beta w^3$. 
With $u=x-1$, $v=y-1$ and $S=Lf+Mg$.
Thus $S=q(u,v)+w^3l(u,v)$, where $l=-\alpha L- \beta M$. By Lemma~\ref{lem:2.4-local-computations}, if the quadratic term $q_2$ is of rank one, then the singularity is Du Val or simple elliptic. If it has rank zero, then $L=cv$ and $M=-cu$, so the weighted blow-up with weights $(3,3,2)$ on $(u,v,w)$ gives an exceptional divisor with log discrepancy $-1$, a contradiction.\\

\textit{Case 2:} In this case, we assume that $p\notin C$.\\ 

We can change the generators of the pencil and rescale to have $F=x^3+y^3+z^3+w^3$ and $G=b_0x^3+b_1y^3+b_2z^3+b_3w^3$, with $F(p)=A \neq 0$ and $G$ vanishing at $p$. Therefore $L(p)=0$, but $M(p) \neq 0$. We can separate the terms by degrees as:
$F=A+f_1+f_2+f_3$, $G=g_1+g_2+g_3$ and $M=1+m_1$. The singularity condition gives $L=-\frac{g_1}{A}$. Thus
$q_2=g_2+(m_1-\frac{f_1}{A})g_1$ and $q_3=g_3+m_1g_2-\frac{f_2}{A}g_1$.

We first consider the case where $p$ has only nonzero coordinates. If every $b_i$ is nonzero, the Hessian of $G$ is diagonal and invertible.
Since the restriction of $g_2$ and $q_2$ to $\{g_1=0\}$ coincide, and $g_2$ restricted to $\{g_1=0\}$ has rank two. The restriction of  $q_2$ has rank at least two, a contradiction.
Now, if some $b_i$ is zero, we can assume $b_3=0$. Thus, the vanishing set of $G$ is a cone over a plane cubic. In an affine chart fixing a different coordinate, write $z=t_0+w$, then $G$ and $L$ do not involve $w$. Writing the base plane cubic locally as $y^3=r+st^3$ where $r$ and $s$ are nonzero constants, its Wronskian with respect to $t$ is $Wr(1,t,y)=\frac{2rst}{y^5}.$ Which is nonzero at $p$. Hence the image of $p$ on the plane cubic is not a flex point. Thus $q_2$ cannot have rank zero. If $q_2$ has rank one, the $w$-coefficient of $m_1$ is $3t_0^2/A$, and the kernel of $q_2$ contains the $w$-direction. Here $g_1$ is a multiple of $v$. Restricting $q_3$ to $q_2$ we obtain $v(c_0v^2+c_1vw+c_2w^2)$, whose $vw^2$ coefficient is $-3t_0/A\ne0$, so it is neither zero nor a cube. Therefore $p$ is not a cusp, by the criterion at the beginning of the proof.

We now consider the case where exactly one coordinate of $p$ is zero. Without loss of generality, $z=0$. Set the local coordinate $w=z$.
Since $f_1, g_1, g_2$ do not involve $w$, rank at most one forces $L$ and $M$ to be independent of $w$. Thus $S=q(u,v)+w^3l(u,v)$, where $l=L+b_3M$.
If $b_3=0$, then the previous Wronskian computation also shows that the image of $p$ on the plane cubic is not a flex point. Therefore $q_2$ has rank one, and by Lemma~\ref{lem:2.4-local-computations} it is not a cusp.
If $b_3\neq 0$, then $l(0) \neq 0$, so $l$ has an analytic cube root. The analytic change of coordinates $W=w l(u,v)^{\frac{1}{3}}$ transforms the equation into $q(u,v)+W^3$. In the case where $q_2$ has rank one, completing the squares gives $U^2+h(v)+W^3$. Log canonicity forces $\ord h \leq 6$, since the weighted blow-up with weights $(21,6,14)$ would yield a non log canonical center otherwise. Therefore the singularity would be canonical or simple elliptic. 
In the case where $q_2$ has rank zero, the tangent cubic is $q_3(u,v)+W^3$. Since the plane cubic is log canonical, the cubic $q_3(u,v)$ must have three distinct roots. 
Thus the tangent cubic of $S$ is smooth and $p$ is simple elliptic.

Then, if exactly two coordinates of $p$ are zero, then we can consider $u$ along the coordinate line determined by two nonzero coordinates, and set $v$ and $w$ to be the zero coordinates.
Since $g_1$ is a multiple of $u$, rank at most one forces $L$ and $M$ to depend only on $u$. Thus $S=q(u)+v^3l_2(u)+w^3l_3(u)$.
If $b_2$ and $b_3$ are nonzero, we can choose analytic cubic roots of $l_2, l_3$ and replace $v, w$ by $v l_2(u)^{\frac{1}{3}}$ and $w l_3(u)^{\frac{1}{3}}$ respectively. The equation becomes $q(u)+v^3+w^3=0$.
Rank one gives $U^2+v^3+w^3$, a canonical singularity. In rank zero, log canonicity forces $\operatorname{ord}q=3$, otherwise the weighted blow-up with weights $(3,4,4)$ would yield a non log canonical center. Thus, the singularity is simple elliptic.
If $b_3=0$, then $b_2\neq 0$ and $l_3$ is a nonzero multiple of $u$. After normalizing the unit $l_2$, the rank one case is non cuspidal by Lemma~\ref{lem:2.4-local-computations}. In rank zero, the weighted blow-up with weights $(3,3,2)$ yields a non log canonical center.

If three coordinates vanish, $G(p)=0$ forces the coefficient of the remaining coordinate to vanish. Then $g_1=0$, so the singularity condition gives $L=0$ and $S=MG$, contradicting irreducibility.

\end{proof}

\begin{proof}[Proof of Theorem~\ref{thm:2.4-not-cluster-type}]
By the Jacobian criterion the curve $C$ is smooth, its blow-up is a smooth Fano threefold in the family \textnumero\,2.4 by \cite{MM81}.

Every index one pair $(X,B)$ with reduced boundary, as required for cluster type, is the log pull-back of a log Calabi--Yau pair  $(\pp^3,B')$. Thus, for $B$ to be effective, $B'$ must contain $C$. $C$ is not contained in planes or quadric surfaces in $\pp^3$. Therefore $B'$ must be an irreducible quartic or the sum of a plane and an irreducible cubic in the pencil of cubics containing $C$.
If $B'$ is a smooth irreducible quartic, then $(\pp^3,B')$ has coregularity two.
If $B'$ is a singular irreducible quartic, then Lemma~\ref{lem:2.4-singular-surfaces} implies that the pair $(\pp^3,B')$ has coregularity greater than or equal to one. 
If $B'$ is the sum of a smooth cubic and a plane, then the pair $(X,B)$ is not of cluster type by Proposition~\ref{prop:2.4-smooth-cubic}.
If $B'$ is the sum of a  singular cubic and a plane, by Lemma~\ref{lem:2.4-singular-surfaces} it has a non rational component, therefore the pair $(\pp^3,B')$ is not log rational. In any case, we conclude that the pair $(X,B)$ is not of cluster type.
\end{proof}

\begin{corollary}
Cluster type is not invariant under smooth Fano deformation. Moreover, a cylindrical smooth Fano threefold need not be cluster type.
\end{corollary}
\begin{proof}
Combine Theorems~\ref{thm:blow-up-P3} and~\ref{thm:2.4-not-cluster-type} in family \textnumero\,2.4. Cylindricity holds for every member by~\cite[Example 3.15]{CPPZ21}. A general one-parameter perturbation of the defining cubic pencil, restricted to the smooth complete-intersection locus, gives the claimed deformation.
\end{proof}

\section{The Mori-Mukai family \texorpdfstring{\textnumero}{No.}\texorpdfstring{\,}{ }1.6} 
\label{sec:MM-1-6}

In this section, we prove that a general element $X$ of the Mori-Mukai family \textnumero\,1.6 is not of cluster type. The main theorem of this section is the following.

\begin{theorem}\label{thm:1-6-not-cluster-type}
Let $X$ be a general element of the Mori-Mukai family \textnumero\,1.6.
Then, the variety $X$ is not of cluster type.
\end{theorem}

Recall that such $X$ is a section of the half-spinor embedding of a connected component of 
${\rm OGr}_{+}(5,10)$ by a codimension $7$ subspace.    
A smooth member of
the Mori-Mukai family 1-6 is a prime Fano threefold of genus seven and satisfies
\[
{\rm Pic}(X)\simeq \zz H, 
\quad 
H\sim -K_X,
\quad 
\text{ and }
\quad 
H^3=12. 
\]
It is also a transverse linear $\pp^8$-section
of the spinor tenfold $\Sigma \subset \pp^{15}$.
These threefolds have an $18$-dimensional moduli space.
On a non-empty open subset of this moduli every line has normal bundle $\mathcal{O}\oplus \mathcal{O}(-1)$ (see~\cite[Proposition 2.1(ii)]{IM07}). The moduli interpretation is also used in~\cite[Theorem 1.1]{IKTT24}.

Below, we will prove a lemma regarding cluster type pairs $(X,B)$ (see Lemma~\ref{lem:movable-system}).
Roughly speaking, under some mild conditions, when $(X,B)$ is of cluster type, there exists a movable linear system $\mathcal M\subseteq|-mK_X|$ and a log canonical place $v$ of $(X,B)$ such that 
$v(\mathcal M)>mv(B)$. Here, $v(\mathcal{M})$ is the multiplicity along $v$ of a general element of $\mathcal{M}$ while $v(B)$ is the multiplicity along $v$ of $B$.
This means that a general element of $\mathcal{M}$ is more singular than the boundary $B$, at least from the perspective of the valuation $v$.
This also gives a criterion for proving that a pair $(X,B)$ is not of cluster type. 

\subsection{Movable systems on cluster type varieties}
\label{subsec:movable-system} In this subsection, we prove a lemma regarding singularities of movable linear systems on cluster type varieties. This lemma can be understood as a version of the Noether-Fano inequality in the context of cluster type pairs (see, e.g.,~\cite[Section 2.2]{Isk04}). 

\begin{lemma}\label{lem:movable-system}
Let $(X,B)$ be a cluster type pair
where $X$ is a $\qq$-factorial klt Fano variety. 
Assume that $B$ is irreducible and 
${\rm Cl}(X)\simeq \zz[B]$. 
Then, there exists a movable linear system $\mathcal{M}\subseteq |-mK_X|$ and a log canonical place $v$ of $(X,B)$ with 
\[
v(\mathcal{M}) > mv(B). 
\]
Furthermore, we can choose the center of $v$ in $X$ to be a closed point. 
\end{lemma}

\begin{proof}
Let $\psi\colon (T,B_T)\dashrightarrow (X,B)$ be a crepant birational map
from a toric log Calabi--Yau pair $(T,B_T)$ such that $\psi$
does not contract divisors from the open torus $\mathbb{T}$ of $T$.
Let $U:=X\setminus B$
and $\mathbb{T}\hookrightarrow U$ be an open embedding
of an algebraic torus (see, e.g.,~\cite[Theorem 1.3.(4)]{EFM24}).
A complexity computation shows that $U\setminus \mathbb{T}$
supports precisely $n$ prime divisors $D_1,\dots,D_n$ where
$n=\dim X$. 
Note that ${\rm Cl}(\mathbb{T})={\rm Cl}(U)=0$. 
Thus, the homomorphism
\[
\mathcal{O}(\mathbb{T})^*/\mathbb{C}^* \rightarrow 
\bigoplus_{i=1}^n \zz[D_i]
\]
must be an isomorphism. 
The previous isomorphism is the same as in the exact sequence~\eqref{eq:LES}.
This implies that we can find a character basis 
$x_1,\dots,x_n$ of the algebraic torus $\mathbb{T}$ such that 
\[
{\rm div}_X(x_i) = D_i - m_i B 
\]
for each $i$. 
The $m_i$ are positive because $D_i\sim m_iB$ is nonzero effective and $B$ is ample.
Let $m$ be the lcm of the positive integers $m_1,\dots,m_n$
and consider the linear system $\mathcal{M}$ of $|mB|$ spanned by the divisors
\[
mB + \frac{m}{m_i}{\rm div}_X(x_i).  
\]
The previous divisors are linearly independent on $\mathbb{T}$
and have distinct support on the complement of the torus.
Therefore, the linear system $\mathcal{M}$ is movable.  
Finally, it suffices to choose a log canonical place
$v$ of $(T,B_T)$ such that:
\begin{itemize}
\item the center $c_X(v)$ is a closed point; and 
\item we have $v(x_i)=w_i>0$ for each $i$. 
\end{itemize} 
The positivity condition holds on a full-dimensional cone. For the closed center condition, embed $X$ by $|rB|$ and express its coordinates on $\mathbb T$ as Laurent polynomials. Outside the finitely many walls equating weights of two exponent vectors, their initial terms are monomials. Ratios of coordinates of minimal weight then have constant residues or vanish, so the valuation has closed point center.
Thus, we can always choose $v$ that satisfies both conditions.
Then, we can compute 
\[
v(\mathcal{M}) \geq mv(B) + \min_i \frac{mw_i}{m_i} > mv(B).
\]
This finishes the proof. 
\end{proof} 

\begin{proof}[Proof of Lemma~\ref{lem:movable-intro}]
Note that the conditions imply $B$ irreducible and $X$ being $\qq$-factorial. Then, the statement follows from Lemma~\ref{lem:movable-system}.
\end{proof}

Applying Lemma~\ref{lem:movable-system} to a threefold is in practice difficult. We will achieve this when $B$ is normal in Subsection~\ref{subsec:normal}. 
In the case that $B$ is non-normal, instead, we will perform some blow-ups along the non-normal locus of $B$ and check $v(\mathcal{M})\leq mv(B)$ in this crepant model. The following lemma states that the inequality can be checked after adjunction to some suitable surface.

\begin{lemma}\label{lem:criteria-under-adjunction}
Let $(X,S+E)$ be a log canonical pair
where $S$ and $E$ are normal Cartier divisors.
Assume that adjunction to $E$ gives $(E,C)$ where $C=S|_E$.
Let $G\geq 0$ be a $\qq$-Cartier divisor containing neither $E$ nor $S$.
Fix $t\geq 0$. 
If every log canonical place $w$ of $(E,C)$ satisfies 
\[
w(G|_E) \leq t w(C), 
\]
then every log canonical place $v$ of $(X,S+E)$ with center contained in $E$ satisfies $v(G)\leq t v(S)$. 
\end{lemma} 

\begin{proof}
On a fixed log resolution of $(E,C+G|_E)$, the hypothesis controls the zero-discrepancy divisors, while all positive discrepancies remain positive for sufficiently small $\epsilon>0$. Hence $(E,(1-t\epsilon)C+\epsilon G|_E)$ is lc.
Inversion of adjunction then implies that 
$(X,E+(1-\epsilon t)S+\epsilon G)$ is log canonical near $E$. 
Evaluating $v$ then gives 
$\epsilon(tv(S)-v(G))\geq 0$. 
\end{proof}

Now, we provide a valuation estimate at a node. In the following lemma, we will write $I_p(A,B)$ for the local intersection number of two effective $\qq$-divisors $A$ and $B$ at a point $p$ of a surface $S$.

\begin{lemma}\label{lem:node-estimate}
Let $(S;p)$ be a smooth surface germ. 
Let $C_1$ and $C_2$ be two transversal curves through $p\in S$ defined by $x$ and $y$, respectively.
Let $R\geq 0$ be a $\qq$-Cartier divisor containing neither branch. Let $w$ be a monomial valuation with weights $\mu$ and $\nu$ at $x$ and $y$, respectively. 
Set $q=\mu+\nu$. Then, we have 
\[
w(R) \leq \min \{ 
\nu I_p(R,C_1), \mu I_p(R,C_2)
\} 
\leq 
\frac{\mu\nu}{q}I_p(R,C_1+C_2) \leq \frac{q}{4}I_p(R,C_1+C_2).
\]
In particular, for $C\subset \pp^2$ an irreducible nodal cubic and $D\geq 0$ a $\qq$-divisor on $\pp^2$ of degree $k$, we get
\[
w(D) \leq \frac{3k}{4}w(C)
\] for any log canonical place $w$ of $(\pp^2,C)$.
\end{lemma} 

\begin{proof}
We may assume $R$ is Cartier and defined by $f$. Restricting to the coordinate axes gives $w(f)\leq \nu {\rm ord}_y f(0,y)$
and $w(f)\leq \mu {\rm ord}_x f(x,0)$.
Taking convex combination with coefficients $\mu/q$ and $\nu/q$ gives the first assertion. 

For the second assertion, write $D=aC+R$ where $R$ is effective and does not contain $C$ on its support. Every lc place of $(\pp^2,C)$, different from $C$ itself, is monomial in the branches at the node.
Then, we can use B\'ezout theorem to conclude
\[
w(D) \leq \left(a+\frac{3}{4}(k-3a) \right)w(C)\leq \frac{3k}{4}w(C).
\]
The inequality for $w={\rm ord}_C$ follows from the fact that $a\leq k/3$.
\end{proof} 

We finish this subsection by providing some estimates for singularities of scrolls. 

\begin{lemma}\label{lem:scroll-estimates}
Let $X=\mathbb{F}_1$ be the first Hirzebruch surface. 
Let $f$ be the class of a fiber in $X$ and $s$ be the class of the $-1$-curve. Let $q$ be a node of a nodal irreducible curve $C_0\in |2s+3f|$. 
Consider the monomial valuation $w$ with positive weights
$\mu$ and $\nu$ along the branches of $C_0$. Put
$u=\min(\mu,\nu)$ and 
$v=\max(\mu,\nu)$. 
Then, for every effective $D\sim As+Bf$, we have 
\[
w(D) \leq (B-A)u+Av.
\]
\end{lemma}

\begin{proof}
Note that the node is not contained in the negative section as $C_0\cdot s=1$. 
On the other hand, for a fiber $f$ we have $C_0\cdot f=2$ and thus $w(f_q)=u$. Moreover, the fiber $f_q$ must be transversal to both branches. Write $D=R+\sigma s+\sum_i\tau_i f_i$, removing the negative section and all fibers from $R$, and set $\tau=\sum_i\tau_i$. 
Intersecting $R$ with $s$ and fibers gives 
\[
\tau \leq B-A + \sigma 
\text{ and }
\mult_q R \leq R f= A-\sigma. 
\]
A monomial valuation of weights at most $v$ 
takes values at most $v$ times the ordinary multiplicity. 
Hence, we can compute
\[
w(D) \leq \tau u + v(A-\sigma)  
\leq (B-A)u+Av+\sigma(u-v).
\]
This completes the proof.
\end{proof}

\subsection{Normal boundaries}\label{subsec:normal} In this subsection, we study normal anti-canonical boundaries on general elements of the Mori-Mukai family \textnumero\,1.6. First, we aim to understand the singularities of such normal anti-canonical boundaries. Finally, we turn to prove that the inequality $v(\mathcal{M})\leq mv(B)$ for every movable linear system $\mathcal{M}\subseteq |-mK_X|$ and every log canonical place $v$ of $(X,B)$. This, together with Lemma~\ref{lem:movable-system} imply that no such pair $(X,B)$ is of cluster type. Thus, no normal boundary $B\in |-K_X|$ can lead to a cluster type pair. In the statement of the next proposition we will use the language of hypersurface cusp singularities (see~\cite[Table 1]{Duc24}).  

\begin{proposition}\label{prop:sing-normal}
Let $X$ be a general smooth prime Fano of genus seven. 
Let $B\in |-K_X|$ be a normal boundary of coregularity zero. 
Then $B$ is smooth except for a single $p\in B$ which is a cusp singularity $T_{3,3,4}$. 
\end{proposition} 

\begin{proof}
First, we note that any normal boundary $B\in |-K_X|$ of coregularity zero, must have a single cusp singularity and possibly other Du Val singularities. 
Coregularity zero supplies an exceptional rational cycle on the minimal resolution. It cannot lie in a fiber of a ruling over a positive genus curve, so that resolution is rational. Since $\chi(\mathcal O_B)=2$ and the resolution has Euler characteristic one, the total local geometric genus is one. Thus there is exactly one cusp and the other points are Du Val.
Every boundary divisor $B\in |-K_X|$ can be written as $B=\Sigma \cap \pp^7$ where $\Sigma$ is the $10$-dimensional spinor variety.
Here, the intersection is happening in the $15$-dimensional ambient space $\pp^{15}$ of the spinor variety ${\rm OGr}^+(5,10)$.
Thus, the boundaries $B$ are parametrized by a $64$-dimensional space ${\rm Gr}(8,16)$.

Set $L=\mathcal{O}_B(1)$, so we have an exact sequence 
\begin{equation}\label{eq:LES0}
0\rightarrow 
\mathcal{A}_L 
\rightarrow 
\mathcal{A}_{\Sigma}|_B 
\rightarrow N_{B/\Sigma} 
\rightarrow T^1_B \rightarrow 0. 
\end{equation}
Here, $\mathcal{A}_L$ is the Atiyah sheaf. 
Let $J$ be the image in $N_{B/\Sigma}$ of the previous sequence. 
Bott-Borel-Weil and the Atiyah extension give 
\[
H^{>0}(\mathcal{A}_\Sigma)=0 \qquad 
H^*(\mathcal{A}_{\Sigma}(-j))=0 \quad (1\leq j \leq 8).
\]
See, e.g.,~\cite[Theorem IV]{Bot57} and~\cite[\S 2]{Ati57}.
Explicitly, $T_\Sigma(-j)$ is acyclic for $1\leq j\leq7$, and at $j=8$ its only cohomology is $H^9(T_\Sigma(-8))=\cc$.
At $j=8$, the connecting map 
$H^9(T_\Sigma(-8))\rightarrow H^{10}(\mathcal{O}_{\Sigma}(-8))$ is the isomorphism induced by $c_1(\mathcal{O}_{\Sigma}(1))$. Koszul resolution together with the aforementioned vanishings give $H^{>0}(\mathcal{A}_{\Sigma}|_B)=0$. 

On a resolution $\widetilde{B}$ of $B$, extension of reflexive one-forms gives logarithmic forms along the exceptional divisors (see~\cite[Theorem 1.5]{GKKP11}). Residues are constant on the projective irreducible components. Their connecting class pairs with exceptional curves by the intersection matrix of the exceptional. Negative definiteness of this matrix implies that every residue vanishes. Rationality of $\widetilde{B}$ then gives $H^0(\Omega_B^{[1]})=0$.
A non-zero map $\mathcal{A}_L\rightarrow \mathcal{O}_B$ 
either factors through $T_B$, contrary to the preceeding vanishing, or splits the Atiyah extension, contrary to its restriction to a positive degree curve.
Serre duality then gives $H^2(\mathcal{A}_L)=0$. 
The two exact sequences obtained from~\eqref{eq:LES0} by using $J$, give  
$H^1(J)=0$ and $H^0(N_{B/\Sigma})\rightarrow H^0(T^1_B)$ surjective.
This proves the simultaneous versality for the singularity germs of $B$ induced by varying its linear section in $\Sigma$.
The Grassmannian and the local hypersurface deformation bases are smooth, so surjectivity of this tangent map gives smoothness of the local deformation map. 
Here, $H^0(N_{B/\Sigma})$ is the $64$-dimensional tangent space of ${\rm Gr}(8,16)$.

We conclude that the locus of linear sections
$B=\Sigma\cap \pp^7$ with the prescribed analytic basket of singularities, consisting of a single $T_{p,q,r}$ and Du Val singularities, has codimension
\[
\tau(B):= p+q+r-2 + \sum_{p\in {\rm Sing}(B)\setminus {\rm LCC}(B)}\tau(p), 
\]
in $\operatorname{Gr}(8,16)$.
Here, $p+q+r-2$ is the Tjurina number of the cusp singularity
and the other summand is the sum of the Tjurina number of the Du Val singularities. 

If $\tau(B)\geq 9$, then the singularity type of $B$ cannot appear as the singularities of an anti-canonical rational boundary of coregularity zero in a general smooth Fano threefold in Mori-Mukai family \textnumero\,1.6. Indeed, the locus of such singularities is invariant under spinor automorphisms.
On the general open subset on which the smooth Fano threefolds
in the Mori-Mukai family \textnumero\,1.6 have trivial automorphisms, 
such orbits of the spinor automorphism have dimension $45$. 
The incidence of $\pp^7\subset\pp^8$ adds seven parameters, so its image has dimension at most $(64-\tau)+7-45=26-\tau\leq17<18$. Only finitely many baskets occur, since the relative $T^1$ sheaf has bounded finite fiber length on this finite type family. 
We conclude that $\tau(B)\leq 8$. 
On the other hand, the coregularity zero assumption implies that 
\[
\frac{1}{p}+\frac{1}{q}+\frac{1}{r} < 1
\text{ and }
2\leq p\leq q \leq r. 
\]
The smallest Tjurina number in this range is $8$ attained only by $(3,3,4)$. 
Furthermore, the Tjurina number of a Du Val singularity is non-zero, and so we conclude that all other points of $B$ must be smooth. 
\end{proof}

Now, we will turn to set up some notation for the next two proofs in this subsection. 

\begin{notation}\label{notation-normal-1.6}
{\em 
Let $X$ be a general smooth prime Fano threefold of genus seven.
Let $H$ be an ample generator of the Picard group of $X$. 
Let $B\in |-K_X|$ be a normal boundary with a single $T_{3,3,4}$ singularity at $p\in X$.
Let $\pi\colon Y\rightarrow X$ be the blow-up of $X$ at $p$.
Write $\pi^*(K_X+B)=K_Y+S+E$,
where $S$ is the strict transform of $B$ in $Y$.
The surface $S$ is a smooth rational surface
and $S\cap E=C$ is a nodal cubic curve. 
Let $p_0 \in Y$ be the node of $C$.
We let $L=\pi^*H|_{S}$. By adjunction, we can compute
\[
K_{S} \sim - C, 
\quad 
L^2=12, 
\quad 
L\cdot C=0, \text{ and }
C^2=-3. 
\]
}
\end{notation}

\begin{lemma}\label{lem:morphism}
With Notation~\ref{notation-normal-1.6}. 
The surface $S$ admits a projective morphism 
$\phi\colon S\rightarrow \pp^3$ which is birational onto its image. Furthermore, its image is normal and has degree $n\in \{2,3\}$ in $\pp^3$.
\end{lemma} 

\begin{proof}
First, we study the scheme of lines $\mathcal{S}$ of $X$ through $p\in X$. 
We write $X=\Sigma \cap \pp^8 \subset \pp^{15}$ and consider Pfaffian coordinates in $\pp^{15}$.
Here, the boundary $B$ is given by $\Sigma\cap \pp^7 \subset \pp^{15}$. 
More precisely, we consider a skew-symmetric $5\times 5$ matrix $A$ and consider the projective coordinates
$[1:A:{\rm Pf}_1(A) : \cdots : {\rm Pf}_5(A)]$. 
Let $V=T_pX$.
The scheme of lines through $p\in X$ is
\[
\mathcal{S}=\{q_1=\dots=q_5=0\} \subset \pp(V) 
\quad 
\text{ with }
\quad q_i={\rm Pf}_i(A|_V).
\]
The scheme $\mathcal{S}$ is finite and reduced, indeed its tangent space at a line is $H^0(N_{\ell/X}(-p))=0$.
No three points of $\mathcal S$ are collinear. In this case, each quadratic Pfaffian would then vanish on their line, contradicting finiteness. Four such points impose four independent conditions on plane quadrics, so their space of quadratic equations is a pencil. A fifth would force a common component. The syzygy argument below therefore excludes four or more points. 

If $\mathcal{S}=\emptyset$, the kernel of $A$ is 
$\mathcal{O}(-2)\subset \mathcal{O}^5$.
Its quotient $\mathcal{E}$ has non-degenerate alternating form with values in $\mathcal{O}(1)$ and hence $\mathcal{E}(-1)\simeq \mathcal{E}^*$.
Since $\operatorname{mult}_p B=3$, the quadratic term of its local
hyperplane equation vanishes on $T_pX$. In spinor coordinates this
gives $\sum_{i=1}^5 c_iq_i=0$ for a nonzero constant vector
$(c_1,\ldots,c_5)$.
The constant Pfaffian relation puts the kernel in a fixed $\mathcal{O}^4$ summand, so $H^0(\mathcal{E}^*)$ is non-zero. However, the short exact sequence 
\[
0\rightarrow \mathcal{O}(-3) \rightarrow 
\mathcal{O}(-1)^5 \rightarrow \mathcal{E}(-1)\rightarrow 0 
\]
gives $H^0(\mathcal E(-1))=0$, leading to a contradiction. We conclude that $\mathcal{S}$ is non-empty.

We argue that $\mathcal{S}$ cannot contain exactly $4$ points. If there were four points in $\mathcal{S}$, the five quadratic Pfaffians $q_i$ would span at most a pencil. Let's call the pencil $q=aQ+bQ'$ for independent constant vectors $a$ and $b$ and coprime quadrics $Q$ and $Q'$. The identity $Aq=0$ and the absence of linear syzygies between $Q$ and $Q'$ would then imply $Aa=Ab=0$.
Thus $A$ has rank at most two and all Pfaffians vanish identically, leading to a contradiction. We conclude that $\mathcal{S}$ contains $1$, $2$, or $3$ points. 

The scheme-theoretic intersection $X\cap T_p X$ is the cone over the scheme $\mathcal{S}$. 
Note that $X\cap T_p X$ is contained in $B$. 
Indeed, the direction of each line lies in the tangent hyperplane defining $B$. Thus, we conclude that $X\cap T_p X \subset B$ and it contains $1,2$, or $3$ lines.
Let $\ell_1,\dots,\ell_k$ be the strict transform of such lines in $S$. Note that $\ell_i^2=-1$ for each $i$.

We consider the linear system 
\[
D:= L-2C -\sum_{i=1}^k \ell_i.
\]
The vector space $H^0(X,\mathcal{I}_p^2(H))$ has dimension $5$. Furthermore, its restriction to $S$ has dimension $4$. Here, with restriction, we mean restriction to $S$ of the strict transform to $Y$ of the sections of this space. Indeed, the kernel is generated by the hypersurface $B$ itself. Note that an element of $H^0(X,\mathcal{I}_p^2(H))$ is the restriction to $X$ of a hyperplane containing $T_pX$. 
Therefore, the restrictions of the sections of 
$H^0(X,\mathcal{I}_p^2(H))$ to $S$ give elements of
$H^0(S,\mathcal{O}_S(D))$.
The strict lines meet $C$ at smooth points. Indeed, a line transverse to the exceptional plane cannot lie in a surface tangent to that plane at the node. Reducedness of $\mathcal S$ gives a projected quadratic Pfaffian\footnote{This is the quadratic Pfaffian restricted to the tangent plane $\pp(V)$.} with nonzero derivative along $C$ at each such point. After subtracting $2C$, its residual equation on $S$ is a local equation of $\ell_i$ times a unit. At the node $p_0$ of $C$, no $\ell_i$ passes through $p_0$,
so some quadratic Pfaffian is nonzero there. After subtracting
$2C$, its residual restriction is therefore a unit at $p_0$. Outside $C$ the base scheme is the reduced union of these lines. Thus the residual base ideal is exactly $\mathcal O_S(-\sum\ell_i)$, with no infinitely near base point.
We conclude that $D$ is base-point-free
and that $H^0(S,\mathcal{O}_S(D))$ has $4$ independent generating sections.
By performing adjunction to $S$, one can compute
\[
D^2=k, \quad 
D\cdot C = 6-k, \quad \text{ and } \quad 
D\cdot \ell_i =0. 
\]
Let $\phi\colon S\rightarrow \pp^3$ be the morphism defined by these four sections.
The four independent sections give a nondegenerate surface image, so $k=D^2=\deg(\phi)\deg(\phi(S))$ with $1\leq k\leq3$ forces $k\in\{2,3\}$ and $\deg(\phi)=1$. Put $n=k$.

A general member of $|D|$ is smooth of genus $n-2$. 
It normalizes a general plane section of the degree $n$ image, whose arithmetic genus is also $n-2$. That plane section is smooth. Hence, the hypersurface image is regular in codimension one and $S_2$ and so it is normal.
Writing $Z=\phi(S)$ and $A=\mathcal O_Z(1)$, adjunction gives $\phi_*C=(4-n)A$ and $\phi_*L=(9-2n)A$. A contracted irreducible curve $R$ has $R^2<0$ and $2p_a(R)-2=R^2-C R$, hence is a $(-1)$-curve meeting $C$ once or a $(-2)$-curve disjoint from $C$. Neither passes through its node $p_0$, so $\phi$ is an isomorphism there.
\end{proof} 

\begin{proposition}\label{prop:normal-cluster-type-1.6}
Let $X$ be a general smooth prime Fano threefold of genus seven.
Let $B\in |-K_X|$ be a normal boundary.
Then, the pair $(X,B)$ is not of cluster type. 
\end{proposition}

\begin{proof}
Assume $(X,B)$ is cluster type and use Notation~\ref{notation-normal-1.6}. Therefore,  Proposition~\ref{prop:sing-normal} applies. Let $\phi\colon S\to Z\subset\pp^3$ be Lemma~\ref{lem:morphism}, $n=\deg Z$ and $A=\mathcal O_Z(1)$. Then $\phi_*C=(4-n)A$ and $\phi_*L=(9-2n)A$, and $\phi$ is an isomorphism near $p_0$.

For a movable $\mathcal M\subseteq|mH|$, let $k=\operatorname{mult}_p\mathcal M$ and let $G$ be the strict transform of a general member. Write $G|_S=\ell C+R$, where $C\not\subset\operatorname{Supp}(R)$, and set $t=k+\ell$. The effective pushforward of $R\sim mL-tC$ has class $dA$, with $d=(9-2n)m-(4-n)t\geq0$. Lemma~\ref{lem:node-estimate} gives, for every lc place $w$ of $(S,C)$,
\[
\frac{w(G|_S)}{w(C)}\leq
\begin{cases}
5m-2k-\ell\leq5m-2k,&n=2,\\
(9m-3k+\ell)/4\leq3m-k,&n=3.
\end{cases}
\]
Indeed, the total intersection of $\phi_*R$ and $\phi_*C$ is $4d$ on the quadric and $3d$ on the cubic. In the latter case $t\leq3m$. The boundary component valuation satisfies the same bounds. On $E\simeq\pp^2$, $G|_E$ has degree $k$, so $w(G|_E)\leq(3k/4)w(C)$.

Apply Lemma~\ref{lem:criteria-under-adjunction} in both directions, and put $a=v(S)$ and $e=v(E)$. For every lc place other than $E$ centered over $p$,
\[
v(G)\leq
\begin{cases}
\min\{3ka/4,(5m-2k)e\},&n=2,\\
\min\{3ka/4,(3m-k)e\},&n=3.
\end{cases}
\]
For $n=3$ this is at most $ma+(3m-k)e$. For $n=2$, $0\leq k\leq5m/2$, and for interior values
\[
\frac{m}{3k/4}+\frac{3m-k}{5m-2k}-1
=\frac{3k^2-14mk+20m^2}{3k(5m-2k)}>0.
\]
Taking the corresponding convex combination bounds the minimum by $ma+(3m-k)e$. For $k=0$ or $k=5m/2$, the minimum is zero, so the same
bound holds directly. Since $v(\mathcal M)=ke+v(G)$ and $v(B)=a+3e$, we obtain $v(\mathcal M)\leq mv(B)$. For $E$ itself use $k\leq3m$. Indeed, away from $p$ the pair is plt. Lemma~\ref{lem:movable-system} now implies that $(X,B)$ is not of cluster type.
\end{proof} 

\subsection{Non-normal boundaries}\label{subsec:non-normal}
In this subsection, we prove that for a general prime smooth Fano threefold $X$ of genus $7$ and any non-normal boundary $B\in |-K_X|$ the pair $(X,B)$ is not of cluster type.
The strategy is similar to that of Subsection~\ref{subsec:normal}. First, we aim to classify the singularities of non-normal anti-canonical boundaries in $X$ of coregularity zero. 
We will prove that the singular locus of such boundary divisors is a line and finitely many of Du Val singularities (see Lemma~\ref{lem:sing-line}). 
Then, we will classify the possible singularities of the strict transform of $B$ when blowing up such line (see Lemma~\ref{lem:singularities-blow-up-line}). 
Finally, for each singularity basket of $B$, we will show that $v(\mathcal{M})\leq m v(B)$ holds for every movable linear system $\mathcal{M}\subseteq |-mK_X|$ and every log canonical place $v$ of $(X,B)$. This will be the content of Proposition~\ref{prop:non-normal-case-1},
Proposition~\ref{prop:non-normal-case-2}, and 
Proposition~\ref{prop:non-normal-case-3}, respectively, depending on the different singularity baskets of $B$.

\begin{lemma}\label{lem:sing-line}
Let $X$ be a prime smooth Fano threefold of genus $7$. 
Let $B\in |-K_X|$ be a non-normal boundary for which $(X,B)$ is log canonical and the coregularity of $(X,B)$ is zero.
Then, the singular locus of $B$ is a single non-normal line and Du Val points.
\end{lemma} 

\begin{proof}
Let $\Sigma \subset \pp^{15}$ be the spinor tenfold and 
$\Sigma^{-}$ be its dual spinor variety. 
We consider the flag $W\subset A$ where $\dim A=9$ and $\dim W=8$ such that 
\begin{align*}
&X=\Sigma \cap \pp(A), \quad B=\Sigma \cap \pp(W),\\
&\Gamma =\Sigma^{-}\cap \pp(A^{\perp}), \text{ and }
B^{-}=\Sigma^{-}\cap \pp(W^{\perp}).
\end{align*}
We argue that if $X$ is smooth and $B$ is arbitrary, then 
$B^{-}$ is normal.
Note that Kuznetsov proved in~\cite[Lemma 3.20]{Kuz18b} that $\Gamma$ is smooth whenever $X$ is smooth. 
Every component of $B^{-}$ has dimension at least two, as $B^{-}$ contains $\Gamma$, then every component must have dimension two. 
As an expected dimensional linear section of the smooth spinor variety, the divisor $B^-$ is a local complete intersection, hence $S_2$. Smoothness of its Cartier section $\Gamma$ implies smoothness along $\Gamma$. Ampleness excludes a positive-dimensional singular locus. Thus $B^-$ is regular in codimension one, hence normal. 

For $q\in\Sigma^-$ we write $H_q \subset \pp^{15}$ for the corresponding hyperplane. We define the contact locus to be 
\[
\Pi_q:=\{x\in \Sigma \mid \text{ $H_q$ is tangent to $\Sigma$ at $x$}\}. 
\]
In~\cite[Proposition 1.1.(iv)]{IM07}, the authors show that the contact locus at $q\in\Sigma^-$
is a $4$-dimensional linear subspace of the opposite spinor variety.
A point $x\in B$ lies in the singular locus if and only if 
$W^{\perp}\cap N_x\neq 0$ where $N_x$ is the orthogonal space to the tangent $T_x\Sigma$. A non-zero vector in $W^{\perp}\cap N_x$ determines a point $q\in \pp(W^{\perp}) \cap \pp(N_x)$. The corresponding hyperplane is tangent to $\Sigma$ so $q\in B^{-}$ and $x \in \Pi_q$.
By spinor duality, we also have $q\in {\rm Sing}(B^{-})$ if and only if 
$\pp(W)\cap \Pi_q \neq \emptyset$.
This implies the set-theoretic equality
\[
{\rm Sing}(B) = \bigcup_{q\in {\rm Sing}(B^{-})}(\pp(W)\cap \Pi_q).
\]
If two points $q_1,q_2\in {\rm Sing}(B^{-})$ satisfy $\Pi_{q_1}\cap \Pi_{q_2}\neq \emptyset$, then both points are elements of $\pp(W^{\perp}\cap N_x)$ which is one-dimensional. This leads to a contradiction. 
Then, all the linear spaces $\pp(W)\cap \Pi_q$ are disjoint. 
Thus ${\rm Sing}(B)$ is a disjoint union of points and lines. 
On a smooth crepant model of the normalized pair, coregularity zero gives a rational boundary cycle. Such a cycle cannot lie in a fiber of a ruling over a positive genus curve, so the model is rational. Its reduced anticanonical boundary is connected, by $H^1(\omega)=0$ and $0\to\omega\to\mathcal O\to\mathcal O_D\to0$. Its image is the non-klt locus of $B$, hence connected. The disjoint contact loci therefore give exactly one non-normal line and no additional isolated non-klt point. All other singularities are Du Val. 
\end{proof} 

We will use the following notation throughout the rest of this subsection.

\begin{notation}\label{not:blow-up-line}
{\em
We recall the construction of the double-projection theorem. 
All the statements used in this notation follow from the work of Prokhorov, see, e.g.,~\cite[Theorem 8.3, Lemma 8.13, Corollary 8.16]{Pro25}.
For any line on a prime smooth Fano threefold of genus seven, 
its blow-up has a small anti-canonical morphism, and it also admits a small modification into a degree five del Pezzo fibration. 
In particular, the linear system $|H-2\ell|$ is a pencil. 

Let $X$ be a general prime smooth Fano threefold of genus $7$.
In particular, all the lines contained in $X$ are ordinary.  
Let $B\in |-K_X|$ be a non-normal boundary for which
$(X,B)$ has coregularity zero. 
Let $\sigma\colon Y={\rm Bl}_\ell X \rightarrow X$ be the blow-up of $\ell$.
Let $p\colon Y\rightarrow V\subset \pp^6$ be its anti-canonical model
which is a small morphism.
There is a small modification $\chi\colon Y\dashrightarrow Y^{+}$
such that there is a fibration $f\colon Y^+\rightarrow \pp^1$ whose general fiber is a smooth del Pezzo surface of degree $5$.
We have an induced small morphism $p^+\colon Y^+\rightarrow V$. 
This leads to the following commutative diagram 
\[
\xymatrix{
Y\ar[d]^-{\sigma}\ar[rd]^-{p}\ar@{-->}[r]^-{\chi} & Y^+ \ar[d]^-{p^{+}} 
\ar[r]^-{f}& \pp^1 \\
X & V. & \\ 
}
\]
We turn to explain the geometry of the exceptional divisor of $\sigma$ and its strict transforms in different models.
Let $E_Y$ be the exceptional divisor of $\sigma$.
We define the pencil $\Lambda_\ell$ on $E_Y$ to be the induced pencil by the double projection from $\ell$.
More precisely, the pencil $\Lambda_\ell$ is defined 
by the projectivized image of $H^0(Y,\mathcal{O}_Y(H-2E_Y))$ on $H^0(E_Y,\mathcal{O}_{E_Y}(-K_{E_Y}))$.

Note that $E_Y \simeq \mathbb{F}_1$.
The morphism $p$ sends $E_Y$ isomorphically to a cubic scroll $F_0\subset V \subset \pp^6$.
We denote by $E$ the strict transform of $E_Y$ in $Y^+$.
The induced fibration $E\rightarrow \pp^1$ is precisely the graph of the pencil $\Lambda_\ell$.
Let $F\subset Y^+$ be the class of a fiber of $f$. 
As above, we let $H$ be the class of the pull-back of $-K_X$ to $Y$.
We write $S$ for the strict transform of $B$ in $Y$ and 
$S^+$ for the strict transform of $B$ in $Y^+$.
Below, we may call $S^+$ the {\em flopped surface}. 

Now, let $\mathcal{M} \subset |-mK_X|$ be a movable linear system in $X$. 
Let $k={\rm mult}_{\ell}(\mathcal{M})$ and $b=2m-k$. 
Let $G$ be a general element in $\mathcal{M}$.
We denote by $G_Y$ the strict transform of $G$ in $Y$ 
and by $G^+$ the strict transform of $G$ in $Y^+$.
We write $C=S^+\cap E$ and $G_S=G^+|_{S^+}$ for the restrictions used below. 
Note that $G^+ \sim mF + bE$ where $0\leq k \leq 2m$ and $0\leq b \leq 2m$.
Finally, we have 
\begin{equation}\label{eq:flop}
mv(B)-v(\mathcal{M})=
bv(E)+mv(S^+)-v(G^+). 
\end{equation} 
}
\end{notation} 

For the rest of this subsection, we may use Notation~\ref{not:blow-up-line} in the proof of the statements.
On the common surface model $\eta\colon\hat E\to E$ and $\rho\colon\hat E\to F_0\simeq\mathbb F_1$, put $\hat C=\eta^*C$, $h=\rho^*(s+2f)$, $D=G_Y|_{E_Y}$ and $c=k-m=m-b$. Then $D\sim ks+(m+k)f$ and $G_{\hat E}=\rho^*D-cK_{\hat E/F_0}\sim c\hat C+bh$. 

Here $G_{\hat E}:=\eta^*(G^+|_E)$. The restrictions agree outside
the $\rho$-exceptional locus, and both sides are linearly
equivalent to $c\hat C+bh$. Their difference is exceptional
and numerically trivial, hence zero by the negativity lemma.
In the cases below, $S^+$ is a degree-five del Pezzo surface, and both $E$ and $S^+$ are normal Cartier divisors in the terminal
Gorenstein threefold $Y^+$. Thus, adjunction has no additional
different. We use this same crepant model for both restrictions below. 

\begin{lemma}\label{lem:singularities-blow-up-line}
Let $X$ be a general prime smooth Fano threefold of genus $7$. 
Let $B\in |-K_X|$ be a non-normal slc boundary of coregularity zero. 
Assume $B$ is non-normal along an ordinary line $\ell$.
Let $\sigma \colon (Y,S+E_Y)\rightarrow (X,B)$ be the crepant blow-up of $\ell$. Set $C_0:=S\cap E_Y$. 
Let $\Lambda_\ell\subset |C_0|$ be the pencil induced by the double projection from $\ell$.
Let $S^+$ be the flopped surface in $Y^+$.
Then, one of the following three cases occurs:
\begin{enumerate}
\item $S^+$ has Du Val singularities, $C_0$ is irreducible nodal, and the node is not a base point of $\Lambda_\ell$; 
\item $S^+$ is smooth, $C_0$ is irreducible nodal, the node is a base point of $\Lambda_\ell$ and the other members of the pencil are smooth at the node and transversal there to both branches of the node; or 
\item $S^+$ is smooth, $C_0$ has two rational components, and neither node is a base point of $\Lambda_\ell$.
\end{enumerate} 
\end{lemma} 

\begin{proof}
First, we note that the displayed cases for the singularities of $C_0$  and base points of $\Lambda_\ell$ have codimension one or two 
in the space of pencils with a marked member. The excluded nodal configurations, including three or more components, a reducible member with a nodal base point, or non-transversality at a nodal base point, have codimension at least three in this space. We will argue that this codimension counting implies that the other cases do not happen for $X$ a general prime smooth Fano threefold of genus $7$.

We consider $V\subset \pp^6$ which is the intersection of three
quadrics and contains the cubic scroll $F_0$.
For a general choice, the threefold $V$ has $8$ ordinary double points $p_1,\dots,p_8$. The scroll $F_0$ is a surface of degree three in $\pp^4$ (see, e.g.,~\cite[Proof of Lemma 6.2]{ALP23}).
We may assume that this fixed $\pp^4$ is determined by the last two variables $x_5$ and $x_6$ of $\pp^6$, being zero. 
After choosing projective coordinates, we can assume that $F_0$ is defined by
\[
Q_1=x_0x_3-x_1x_2, \qquad 
Q_2=x_0x_4-x_2^2, \quad
\text{ and } \quad 
Q_3=x_1x_4-x_2x_3. 
\]
Therefore, the variety $V$ is determined by the equations
\[
Q_i + x_5L_{i1}+x_6L_{i2}=0 \text{ for $i\in \{1,2,3\}$.}
\]
Here $x_5L_{i1}+x_6L_{i2}$ represents a quadratic form vanishing on 
$\{x_5=x_6=0\}$. There are $13$ quadratic monomials involving $x_5$ or $x_6$. We conclude that there is a $39$-dimensional space parametrizing the possible varieties $V$. Choosing a hyperplane $\Pi\subset\pp^6$ containing the fixed
$\pp^4$ adds one parameter, giving a $40$-dimensional space
of pairs $(V,\Pi)$. The Hilbert--Burch resolution $0\to\mathcal O(-3)^2\to\mathcal O(-2)^3\to\mathcal I_{F_0/\pp^4}\to0$ identifies the induced sections with the contraction of each triple $(L_{1j},L_{2j},L_{3j})$ with the scroll relation. 
Explicitly, under
$(x_0,x_1,x_2,x_3,x_4)=(va^2,ua,vab,ub,vb^2)$,
the two syzygy vectors are $ar$ and $br$, where
$r=(vb,-u,va)$ is nowhere zero.
Its entries span $H^0(s+f)$, and multiplication $H^0(s+f)\otimes H^0(s+2f)\to H^0(2s+3f)$ is surjective on scroll monomials. Thus these coefficients dominate the ordered pencil generators. The two-component strata have dimension six in $|-K_{F_0}|=\pp^8$, with splittings $s+(s+3f)$, $(s+f)+(s+2f)$ or $f+(2s+2f)$. Three component strata have dimension at most five. The irreducible nodal stratum has codimension one. Requiring its
node to be a pencil base point adds one condition. Requiring a base
point at a reducible node, or tangency to either branch at a nodal
base point, adds one further condition.

Now, let $(X,B)$ be a pair for which $C_0$ and $\Lambda_\ell$ do not satisfy (1),(2), or (3). Then the corresponding triples $(X,\ell,B)$ belong to a family of dimension at most $37$. The group $G$ of projective transformations of $\pp^6$ preserving the scroll $F_0$ has dimension $20$. Indeed, the automorphism group of the embedded scroll has dimension
six; extending its action from $\pp^4$ to $\pp^6$ adds
$5\cdot2+2^2=14$ parameters. On the locus of $X$ with finite automorphism group, the stabilizer of the marked construction is finite, so its $G$-orbits have dimension twenty. Thus the image in the moduli has dimension at most $40-3-20=17<18$. 

In case (1), $S^+$ is normal. Its codimension-one singularities are removed by the blow-up of the non-normal line, and it is Cartier, hence $S_2$. Outside $E$ its singularities are Du Val. At a node of $C$, its local equation has quadratic rank at least two, hence any singularity there is $A_n$. It is smooth at the smooth points of $C$ by adjunction.

We argue that $S^+$ is smooth in the cases (2) and (3).
In cases (2) and (3), the pairs $(\Lambda_\ell,C_0)$ already have codimension two. Additional singularities impose an independent condition. Fix the pencil and write the marked hyperplane as $x_6=0$. Near a point of the scroll, we eliminate two equations whose normal differentials are independent. In the graph chart $x_6=\lambda x_5$, the remaining equation, after removing the scroll factor, is $c_0(u,v)+\lambda c_1(u,v)+x_5\kappa(u,v,x_5,\lambda)=0$, where $c_0,c_1$ define the pencil. The marked fiber $S^+$ is singular at a node of $c_0$ precisely when $\kappa(p,0,0)=0$. The variation of the scalar $\kappa(p,0,0)$ in the three $x_5^2$ coefficients equals the dot product of these coefficients with the nonzero scroll relation at $p$. 
More precisely, we have a relation
\[
\delta \kappa(p,0,0)=\sum_{i=1}^3 r_i(p)\delta a_i,
\]
where $a_i$ is the coefficient of $x_5^2$ in the $i$-th defining
quadric and $\vec r=(r_1,r_2,r_3)=(vb,-u,va)$ is the
nowhere-zero relation vector above.
Hence, the map from variation of the three coefficients to the variation of the scalar $\kappa(p,0,0)$ is surjective.
Therefore, even with the pencil fixed, the scalar obstruction $\kappa(p,0,0)$ can vary freely. 
Outside the scroll on $x_6=0$, we have $x_5\neq 0$, so these
coefficients vary the three defining equation values independently.
After passing to a flat projective family of
marked flopped surfaces $S^+$, the singular fiber locus is closed and is
proper in each codimension two stratum.
Hence generic smoothness applies in this situation to 
the projection from the universal surface outside the scroll to the three coefficient parameters. Thus a singular $S^+$ in any codimension two case imposes at least one more condition. Smoothness elsewhere along $C$ follows from adjunction, and the same dimension count as above excludes these fibers for general $X$.
\end{proof} 

\begin{proposition}\label{prop:non-normal-case-1}
Let $X$ be a general prime smooth Fano threefold of genus $7$. 
Let $B\in |-K_X|$ be a non-normal slc boundary of coregularity zero. 
Assume $B$ is non-normal along an ordinary line $\ell$.
Let $\sigma\colon (Y,S+E_Y)\rightarrow (X,B)$ be the crepant blow-up of $\ell$. Set $C_0:=S\cap E_Y$. 
Let $\Lambda_\ell\subset |C_0|$ be the pencil induced by the double projection from $\ell$.
Assume that $S^+$ has Du Val singularities, 
$C_0$ is irreducible nodal, and the node is not a base point of $\Lambda_\ell$. 
Then, for every movable linear system $\mathcal{M}\subseteq |-mK_X|$ 
and every log canonical place $v$ of $(X,B)$ centered at a point of $X$, we have $v(\mathcal{M})\leq mv(B)$. 
\end{proposition}

\begin{proof}
Let $G_{E_Y}$ be the restriction of $G_Y$ to $E_Y$. Let $\eta\colon\hat E\to E$ be a crepant resolution and $\rho\colon\hat E\to F_0\simeq\mathbb F_1$ resolve $\Lambda_\ell$. Set $\hat C=\eta^*C$, $G_{\hat E}=\eta^*(G^+|_E)$, and $h=\rho^*(s+2f)$. Restriction on a common resolution of $\chi$ gives
\[
G_{\hat E}=\rho^*G_{E_Y}-(k-m)K_{\hat E/F_0}
\sim(m-b)\hat C+bh.
\]
Write
\[
G_{\hat E}=\alpha\hat C+R 
\]
where $R$ is effective and does not contain $\hat C$. 
Let $w$ be a monomial valuation centered at the node $q$ of $\hat C$ with weights $\mu$ and $\nu$. 
Set $q_0=\mu+\nu$ and $t=\min(\mu,\nu)/q_0 \in (0,1/2]$.
By Lemma~\ref{lem:node-estimate}, we have
\[
w(R) \leq \frac{\mu\nu}{\mu+\nu}I_q(R,\hat C),
\]
where $I_q$ denotes the local intersection product. 
This implies 
\begin{equation}\label{eq:first-approx}
\frac{w(G_{\hat E})}{q_0} \leq \alpha + 5bt(1-t). 
\end{equation}
On the other hand, Lemma~\ref{lem:scroll-estimates} implies 
\begin{equation}\label{eq:second-approx}
\frac{w(G_{\hat E})}{q_0} \leq k+(m-k)t-\alpha(1-t). 
\end{equation} 
We can multiply inequality~\eqref{eq:first-approx} by $1-t$ and
add inequality~\eqref{eq:second-approx} and then divide by $2-t$ to get
\[
\frac{w(G_{\hat E})}{q_0} \leq m+b\frac{(1-t)(5t(1-t)-1)}{2-t}. 
\]
The rational function on $t$ displayed in the right-hand side of the previous inequality is bounded above by $\frac{1}{10}$ in the interval $(0,1/2]$.
For $t\leq1/3$ use $t(1-t)\leq2/9$ and $(1-t)/(2-t)\leq1/2$. For $t\geq1/3$ use $1/4$ and $2/5$, respectively. On the other hand, we have $b\leq 2m$. 
We conclude that 
\begin{equation}\label{one-approx}
w(G_{\hat E}) \leq \frac{6}{5}mw(\hat C).
\end{equation} 
At the same time, the global log canonical threshold of a degree five Du Val del Pezzo surface is at least $1/6$ (see~\cite[Theorem 1.4 and Theorem 1.5]{PW10}). 
Thus, we deduce that 
\begin{equation}\label{second-approx}
w(G_S) \leq 6b w(C)
\end{equation} 
for every log canonical place $w$ of $(S^+,C)$.
By Lemma~\ref{lem:criteria-under-adjunction}, we conclude that 
\[
v(G^+) \leq \min \left\{ 
\frac{6}{5}mv(S^+),6bv(E) 
\right\} \leq 
\frac{5}{6} \frac{6}{5}mv(S^+) + \frac{1}{6}6bv(E) = 
mv(S^+)+bv(E). 
\]
The first inequality follows from Lemma~\ref{lem:criteria-under-adjunction} by plugging inequality~\eqref{one-approx}
and inequality~\eqref{second-approx}. 
The second inequality follows as the value is a convex combination 
of the two values in the minimum. 
By equality~\eqref{eq:flop}, we conclude that $mv(B)\geq v(\mathcal{M})$. 
\end{proof} 

\begin{proposition}\label{prop:non-normal-case-2}
Let $X$ be a general prime smooth Fano threefold of genus $7$. 
Let $B\in |-K_X|$ be a non-normal slc boundary of coregularity zero. 
Assume $B$ is non-normal along an ordinary line $\ell$.
Let $\sigma\colon (Y,S+E_Y)\rightarrow (X,B)$ be the crepant blow-up of $\ell$. Set $C_0:=S\cap E_Y$. 
Let $\Lambda_\ell\subset |C_0|$ be the pencil induced by the double projection from $\ell$.
Assume that the flopped surface $S^+$ is smooth, the curve
$C_0$ is irreducible nodal, and the node is a base point of $\Lambda_\ell$. 
Assume that other members of $\Lambda_\ell$ are smooth at the node and transversal there to both branches of the node.
Then, for every movable linear system $\mathcal{M}\subseteq |-mK_X|$ 
and every log canonical place $v$ of $(X,B)$ centered at a point of $X$, we have $v(\mathcal{M})\leq mv(B)$. 
\end{proposition}

\begin{proof}
Let $G$ be a general element of the movable linear system $\mathcal{M}$.
Let $G^+$ be its strict transform in $Y^+$
and let $G_E$ be its restriction to $E$. 

Let $\eta\colon \hat{E}\rightarrow E$ be a crepant resolution of singularities and $\rho\colon \hat{E}\rightarrow F_0\simeq \mathbb{F}_1$ be a resolution of the pencil $\Lambda_\ell$.
We let $\hat{C}=\eta^*(E\cap S^+)$
and $G_{\hat{E}}=\eta^*(G_E)$.
We write $h=\rho^*(s+2f)$. 
Therefore, we have $G_{\hat{E}}\sim (m-b)\hat{C}+bh$
and $G_S\sim -bK_{S^+}$.
Let $w$ be a monomial valuation on $F_0$ centered at the node of $C_0$
with weights $\mu$ and $\nu$. 
We set $u=\min(\mu,\nu)$ and $v=\max(\mu,\nu)$. 
Note that 
\[
w(K_{\hat{E}/F_0})=u \text{ and }
w(\hat{C})=\mu+\nu-u=v.
\]
By the previous equalities, we can use Lemma~\ref{lem:scroll-estimates} to conclude that 
$w(G_{\hat E})\leq mu+kv-(k-m)u=bu+kv\leq2m\,w(\hat C)$.
On the other hand, using the fact that $S^+$ is smooth, we deduce that its global log canonical threshold is at least $\frac{1}{2}$ (see~\cite[Theorem 1.3]{PW10}).
Therefore, we conclude that $w(G_S)\leq 2bw(C)$ for every lc place
$w$ of $(S^+,C)$. By Lemma~\ref{lem:criteria-under-adjunction}, we conclude that 
\[
v(G^+) \leq \min \{ 2mv(S^+),2bv(E)\} \leq mv(S^+)+bv(E). 
\]
Therefore, by equality~\eqref{eq:flop}, we conclude that $v(\mathcal{M})\leq mv(B)$.
\end{proof} 

\begin{proposition}\label{prop:non-normal-case-3}
Let $X$ be a general prime smooth Fano threefold of genus $7$. 
Let $B\in |-K_X|$ be a non-normal slc boundary of coregularity zero. 
Assume $B$ is non-normal along an ordinary line $\ell$.
Let $\sigma\colon (Y,S+E_Y)\rightarrow (X,B)$ be the crepant blow-up of $\ell$. Set $C_0:=S\cap E_Y$. 
Assume that the flopped surface $S^+$ is smooth, the curve 
$C_0$ has two rational components, and neither node of $C_0$ is a base point of $\Lambda_\ell$.
Then, for every movable linear system $\mathcal{M}\subseteq |-mK_X|$ 
and every log canonical place $v$ of $(X,B)$ centered at a point of $X$, we have $v(\mathcal{M})\leq mv(B)$. 
\end{proposition}

\begin{proof}
Let $G$ be a general element of the movable linear system $\mathcal{M}$.
Let $G^+$ be its strict transform in $Y^+$ and let $G_E$ be its restriction to $E$. 

Let $\eta\colon \hat{E}\rightarrow E$ be a crepant resolution of singularities and $\rho\colon \hat{E}\rightarrow F_0\simeq \mathbb{F}_1$ be a resolution of the pencil $\Lambda_\ell$. 
We let $\hat{C}=\eta^*(E\cap S^+)$
and $G_{\hat{E}}=\eta^*(G_E)$.
We let $h=\rho^*(s+2f)$. 
Since $S^+$ is smooth, we have $w(G_S)\leq 2bw(C)$ where $C=S^+\cap E$ for every log canonical place $w$ of $(S^+,C)$. 
For $0\leq D \sim As+Bf$ with $B\geq A\geq 0$ and $B> 0$, we have 
${\rm lct}(\mathbb{F}_1,D)\geq 1/B$. 
The lower bound for the log canonical threshold can be obtained by degenerating $D$ to a torus invariant divisor and using lower semicontinuity of log canonical thresholds (see, e.g.,~\cite[Theorem 4.9]{Mus02}).
Since $\mathbb{F}_1$ is projective, lower semicontinuity also
applies to the global threshold.
Indeed, we can degenerate successively under the two torus one parameter subgroups. In the invariant limit, section coefficients are at most $A$ and fiber coefficients at most $B$. The support is snc, proving the bound.
Since the nodes avoid the pencil base locus, for $k\leq m$ the preceding inequality gives 
\[
w(G_{\hat E})=w(G_{E_Y})\leq(m+k)w(\hat C)\leq2mw(\hat C).
\]
For $k\geq m$, we prove the same previous inequality. 
We set $c=k-m=m-b\geq 0$. 

First, we assume that $C_0$ has two horizontal components. 
Let their $h$-degrees be $(d,5-d)$ for $d\in \{1,2\}$.
Write $G_{\hat{E}}=\alpha C_1 +\beta C_2+R$.
The nef ruling gives $\alpha+\beta \leq k$. 
We can compute the intersections
\[
R\cdot C_1 = bd+2\alpha-2\beta 
\text{ and }
R\cdot C_2 = b(5-d)+2\beta-2\alpha. 
\]
At a node with weights $\mu$ and $\nu$, we set $x=\mu/(\mu+\nu)$.
By Lemma~\ref{lem:node-estimate}, $w(R)\leq\min\{\nu R\cdot C_1,\mu R\cdot C_2\}$. Taking the convex combination of these two bounds
with coefficients $(1+2x)/4$ and $(3-2x)/4$ yields 
\[
\frac{w(G_{\hat{E}})}{w(\hat C)} \leq \frac{\alpha+\beta}{2} + 
b\left(\frac{d}{4}+\frac{15-2d}{4}x-\frac{5}{2}x^2\right). 
\]
For $d=1,2$ the respective maxima of the function on $x$ are $209/160$ and $201/160$, both below $3/2$. Consequently, we have 
\[
w(G_{\hat{E}}) \leq \left(\frac{k}{2}+\frac{3}{2}b\right)w(\hat C) =
(m+b)w(\hat C)\leq 2mw(\hat C). 
\]

Now, assume that $C_0$ has one fiber and one bisection.
The statement is immediate if $b=0$. So, assume $b>0$.
The intersection matrices of these components  on the surfaces $\hat{E}$ and $F_0$ are 
\[
Q=\begin{bmatrix}
-2 & 2 \\
2 & -2 
\end{bmatrix} 
\text{ and } 
Q_0=\begin{bmatrix}
0 & 2 \\
2 & 4 
\end{bmatrix}.
\]
Set the vector $d=\begin{bmatrix}1\\4\end{bmatrix}$.
Again, write $G_{\hat{E}}=\alpha C_1+\beta C_2+R$. 
Set $z_1=(\alpha-c)/b$, $z_2=(\beta-c)/b$, and $T=(w(G_{\hat{E}})/w(\hat C)-c)/b$.
Both $R$ and $\rho_*R$ are effective and contain neither boundary component.
Writing $C_{i,0}=\rho_*C_i$, we get $\rho_*R=b(s+2f-z_1C_{1,0}-z_2C_{2,0})$, which justifies the inequalities on the second model. 
Intersection therefore gives $Qz\leq d$ and $Q_0z\leq d$. 
Applying Lemma~\ref{lem:node-estimate} to the
divisors on $\widehat E$ and $F_0$, with normalized
node weights $x,1-x$, gives
\begin{align}
T & \leq xz_1+(1-x)z_2+(1-x)(d_1-(Qz)_1)\\
T & \leq xz_1+(1-x)z_2+x(d_2-(Qz)_2), 
\end{align} 
and the same inequalities for $Q$ replaced with $Q_0$. 
We take the previous four inequalities, multiply them by 
$3/10, 1/30, 1/2$ and $1/6$, respectively, and add them.
By doing so, we obtain
\[
T \leq \frac{4}{5}+\frac{3}{5}z_1 - \frac{3}{5}z_2. 
\]
The inequality $Qz\leq d$ implies $\frac{3}{5}z_1 - \frac{3}{5}z_2 \leq 6/5$.
Thus, we conclude that $T\leq 2$. This implies that 
\[
w(G_{\hat{E}})\leq (c+2b)w(\hat C) = (m+b)w(\hat C) \leq 2mw(\hat C).
\]

Consequently, in any possible case, we conclude that 
$w(G_{\hat{E}})\leq 2mw(\hat C)$ for any log canonical place $w$. 
By Lemma~\ref{lem:criteria-under-adjunction}, we obtain
\[
v(G^+)\leq \min \{ 2mv(S^+),2bv(E)\} 
\leq mv(S^+)+bv(E). 
\]
Therefore, we conclude that $v(\mathcal{M})\leq mv(B)$ by equality~\eqref{eq:flop}. This finishes the proof.
\end{proof} 

\subsection{Cluster type condition} In this subsection, we finish the proof of Theorem~\ref{thm:1-6-not-cluster-type}. 

\begin{proof}[Proof of Theorem~\ref{thm:1-6-not-cluster-type}]
Let $X$ be a general smooth prime Fano threefold of genus seven.
Let $B\in |-K_X|$ be an slc boundary of coregularity zero. 
The surface $B$ is either normal or non-normal.
If $B$ is normal, then by Proposition~\ref{prop:normal-cluster-type-1.6} we conclude that $(X,B)$ is not of cluster type.
If $B$ is non-normal, then by Lemma~\ref{lem:sing-line} the singularities of $B$ are a single ordinary line $\ell$ and possibly some Du Val singularities. 
Let $\sigma \colon Y \rightarrow X$ be the blow-up of the line, 
with exceptional divisor $E_Y$. Let 
$\sigma\colon (Y,S+E_Y)\rightarrow (X,B)$ be the corresponding crepant morphism. Set $C_0:=S\cap E_Y$. Then, by Lemma~\ref{lem:singularities-blow-up-line}, there are three possible cases:
\begin{enumerate}
\item $S^+$ has Du Val singularities, $C_0$ is irreducible nodal, and its node is not a base point of $\Lambda_\ell$; 
\item $S^+$ is smooth, $C_0$ is irreducible nodal, and its node is a base point of $\Lambda_\ell$ and the other members of the pencil are smooth and transversal to both branches of the node; or
\item $S^+$ is smooth, $C_0$ has two rational components, and neither node is a base point of $\Lambda_\ell$. 
\end{enumerate} 
In Case (1), we know that $(X,B)$ is not of cluster type due to Proposition~\ref{prop:non-normal-case-1}.
In Case (2), we know that $(X,B)$ is not of cluster type due to Proposition~\ref{prop:non-normal-case-2}. Finally, in Case (3), we know that $(X,B)$ is not of cluster type due to Proposition~\ref{prop:non-normal-case-3}. Thus, we conclude that $(X,B)$ is not of cluster type for any $B\in |-K_X|$. 
Hence, the variety $X$ is not of cluster type.
\end{proof} 

\begin{proof}[Proof of Theorems~\ref{theorem:cluster-type-Fano-threefolds}, \ref{thm:complete-cluster-type}, and~\ref{thm:complete-cylindric}]
The constructions in Sections~\ref{sec:p3}--\ref{sec:other}, together with Lemma~\ref{lem:product-cluster-type}, cover the $87$ families listed positively in Table~\ref{table:Fanos}. They apply to every smooth member except for the stated generality qualifications in families \textnumero\,2.4 and \textnumero\,10.1. This proves the first two assertions. Theorem~\ref{theorem:ct-implies-cylindrical}, the known cylinders in family \textnumero\,2.4 and on del Pezzo surfaces~\cite[Example 3.15 and Theorem 2.10]{CPPZ21}, give the cylindricity assertion.
\end{proof}

\section{Applications to cylindrical varieties}
\label{sec:cylindrical}

In this section, we prove that cluster type varieties with mild singularities are cylindrical. A cylinder $U$ in a variety $X$ is said to be $H$-polar 
if there exists $0\leq B \sim_\qq H$ such that $U=X\setminus {\rm Supp}(B)$.
A cylinder $U\subset X$ is said to be an {\em anticanonical} cylinder if 
there exists an effective divisor $0\leq B \sim_\qq -K_X$ such that 
$U=X\setminus {\rm Supp}(B)$. If furthermore $-K_X$ is ample, then we may say that $U\subset X$ is a {\em polar anticanonical cylinder}. 

Although it is expected that many rational varieties are cylindrical, admitting an anticanonical polar cylinder is a much stronger condition.
For instance, every smooth del Pezzo surface is cylindrical 
while only del Pezzo surfaces of degree at least $4$ admit anticanonical polar cylinders (see~\cite[Theorem 1.2.3]{CPW16} and~\cite{CPWsing16}). Below, we give a criterion for a cluster type variety to admit an anticanonical polar cylinder which is compatible with the cluster type structure. We show that this is the case for del Pezzo surfaces of degree at least $4$. First, we formalize the concept of cylinders compatible with cluster type structures. 

\begin{definition}
{\em 
Let $(X,B)$ be a cluster type pair
with an embedding in codimension one $j\colon \mathbb{G}_m^n \dashrightarrow X\setminus B$. A {\em cluster cylinder} of $X$ is an open subset $U\subset X$ for which the following conditions are satisfied:
\begin{enumerate}
\item we have $U\simeq \mathbb{A}^1\times Z$ where $Z$ is an affine variety; 
\item the variety $Z$ is an open subset of $\mathbb{G}_m^{n-1}$; and 
\item $j(\mathbb{G}_m^n)\cap U$ is isomorphic in codimension one to $\mathbb{G}_m\times Z$.
\end{enumerate} 
If $U\subset X$ is moreover a polar cylinder (resp. anticanonical cylinder, polar anticanonical cylinder), then we say that $U\subset X$ is a {\em polar cluster cylinder} 
(resp. {\em anticanonical cluster cylinder}, {\em polar anticanonical cluster cylinder}). 
}
\end{definition} 

In the previous definition, $j(\mathbb{G}_m^n)$ stands for the rational image.
Roughly speaking, a cluster cylinder in $X$ 
comes from the choice of a toric morphism $\pi\colon \mathbb{G}_m^n \rightarrow \mathbb{G}_m^{n-1}$ for which there exists an open subset $Z\subset\mathbb G_m^{n-1}$ and the trivial $\mathbb{G}_m$-bundle $\pi^{-1}(Z)$
can be extended to an $\mathbb{A}^1$-bundle on $X$. 

\begin{proof}[Proof of Theorem~\ref{theorem:ct-implies-cylindrical}]
Let $(X,B)$ be a cluster type pair. 
As $X$ is projective and has klt singularities, the boundary divisor $B$ is non-trivial. 
Let $S\subseteq B$ be a prime component. 
Let 
\[
\phi\colon (\pp^n,H_0+\dots+H_n) \dashrightarrow (X,B) 
\]
be a crepant birational map such that 
${\rm Ex}(\phi)\cap \mathbb{G}_m^n$ has codimension at least two in $\mathbb{G}_m^n$ (see Proposition~\ref{prop:cluster-type-bir}).
We may replace $(\pp^n,H_0+\dots+H_n)$ with a higher birational toric model $(T,B_T)$ and assume that 
there is a $\pp^1$-fibration $\pi\colon T\rightarrow T_0$
such that the strict transform $S_T$ of $H_0$ in $T$
is a birational section for $\pi$. 
We may assume that the strict transform $S_{\pp^n}$ of $S$ in $\pp^n$ equals $H_0$. 
Let $S'_T$ be the second torus invariant section of $\pi$. 
Therefore, we have a crepant birational map 
\[
\psi\colon (T,B_T)\dashrightarrow (X,B) 
\]
which is small on the torus $\mathbb{G}_m^n$.
Let $E$ be the exceptional locus of $\psi$, i.e., the union of the locus
where $\psi$ is not defined and the locus where $\psi$ is not an isomorphism. 
By construction, the divisor $S_T$ is not contained in $E$.
Write $E_{\rm vert}$ for all the irreducible components of $E$ that are vertical over $T_0$. 
Recall that $\psi$ does not contract divisors in the algebraic torus.
Therefore, the only divisors that $\psi$ can contract, except from $S_T$ and $S'_T$, must be vertical over $T_0$.
We conclude that every prime component of $E_{\rm vert}$ either has codimension at least two or is a prime torus invariant divisor of $T$ that maps to a prime divisor on $T_0$. 
We conclude that the image of $E_{\rm vert}$ in $T_0$ must be proper.
Thus, we can write $E=E_{\rm vert}$ or 
$E=E_{\rm vert}+S'_T$, depending whether $S'_T$ is contained or not in the exceptional locus of $\psi$.
Let $V$ be the image of $E_{\rm vert}$ in $T_0$. By construction, $V$ is a proper subset of $T_0$.
Let $U$ be an open affine subset of $T_0$ which is disjoint of $V$.
Shrinking $U$ we may assume that the $\pp^1$-bundle is trivial over this open subset. 
Let $U_T:=\pi^{-1}(U)$.
Note that $U_T\setminus E$ is either isomorphic to 
$\pp^1\times U$ or $\mathbb{A}^1\times U$, depending whether $E$ contains or not $S'_T$ on its support. 
In any case, we get open embeddings
$\mathbb{A}^1\times U \hookrightarrow U_T\setminus E \hookrightarrow X$.
Thus, the variety $X$ admits a cluster cylinder. 
\end{proof} 

\begin{theorem}\label{thm:anticanonical-polar-cylinder}
Let $X$ be a projective normal klt variety.
Let $(X,B)$ be a log Calabi--Yau pair. 
Assume that there exists a crepant map
\[
\phi\colon (T,B_T) \dashrightarrow (X,B)
\]
from a toric
log Calabi--Yau pair $(T,B_T)$ satisfying the following conditions:
\begin{enumerate}
\item $\phi$ restricts to an isomorphism on $\mathbb{G}_m^n$;
\item there is an ample irreducible component $S_T$ of $B_T$; and 
\item for every prime divisor $P$ in $X$, exceptional over $T$, we have $c_P(T)\not \subseteq S_T$.
\end{enumerate} 
Then, the variety $X$ admits an anticanonical cluster cylinder. 
Furthermore, if $-K_X$ is ample, $X$ admits a polar anticanonical cluster cylinder.
\end{theorem}

\begin{proof}
By (2), choose a character $u$ with $\frac1N\operatorname{div}_T(u)=A_T-S_T$, where $N\in\zz_{>0}$ and $A_T$ is effective with support $B_T-S_T$. By (1) and (3), every prime divisor on $X$ exceptional over $T$ has positive $u$-order. Thus $S_T$ cannot be contracted. Otherwise $u$ would have no poles and be constant on projective normal $X$.
Let $S_X$ be the strict transform of $S_T$ in $X$.

Note that (1) implies that our pair $(X,B)$ is of cluster type.
Let $\rho_S$ be the ray in the fan of $T$ corresponding to the prime
torus invariant divisor $S_T$. 
Let $\psi\colon T'\dashrightarrow T$ be a toric birational map that extracts
the torus invariant divisor corresponding to $-\rho_S$ and 
the torus invariant divisors corresponding to prime components of $B$. We may take $T^{\prime}$ smooth. 
Let $(T',B_{T'})$ be the log pull-back of $(T,B_T)$ to $T'$. 
By construction, $T^{\prime}\dashrightarrow T$ is an isomorphism at the generic point of $S_T$. We let $S_{T'}$ be the strict transform of $S_T$ in $T'$. 
Since the rays $\rho_S$ and $-\rho_S$ are part of the fan of $T'$, then we have a rational fibration $\pi\colon T'\dashrightarrow \mathbb{G}_m^{n-1}$ 
such that $\pi$ is indeed a morphism near the general fiber which is isomorphic to $\pp^1$. We have a commutative diagram as follows 
\[
\xymatrix{
\mathbb{G}_m^{n-1} & (T',B_{T'})\ar@{-->}[l]^-{\pi} \ar@{-->}[d]_-{\psi} \ar@{-->}[r]^-{\phi'} & (X,B) \\ 
& (T,B_T) \ar@{-->}[ru]_-{\phi} & 
}
\]
Let $B'_1,\dots,B'_k$ be the prime components of the torus invariant boundary $B_{T'}$.
We may assume $B'_1=S_{T'}$.
By construction, we have $T'\setminus {\rm Supp}(\sum_{i=2}^k B_i')\simeq \mathbb{G}_m^{n-1}\times \mathbb{A}^1$.
Let $S_1,\dots,S_m$ be the prime components of $(X\setminus B)\setminus \phi(\mathbb{G}_m^n)$. 
By replacing $T'$ by a higher birational toric model, we may assume that for each $i\in \{1,\dots,m\}$ the center $c_{T'}(S_i)$ is a codimension two subvariety of $T'$ contained in a unique prime torus invariant divisor
(see, e.g.,~\cite[Proof of Lemma 3.1]{EFM24}). 
These refinements preserve the cylinder and make the centers contain no toric stratum. Adding a sufficiently small general ample divisor through them keeps the pair lc and makes the log discrepancies of the $S_i$ lie in $(0,1)$.
By~\cite[Theorem 1]{Mor19}, applied to this perturbed pair, we can find a projective birational morphism 
$\rho\colon Y\rightarrow T'$ of relative Picard rank $m$ extracting precisely $S_1,\dots,S_m$. We write $S_{Y,1},\dots,S_{Y,m}$ for the center of each $S_i$ on $Y$. 
Therefore, we have a birational contraction 
$\phi_Y \colon Y\dashrightarrow X$.
Taking character orders on $T^{\prime}$, we can write $B'_1\sim_\qq \sum_{i=2}^k \alpha_i B'_i$
where $\alpha_i=N^{-1}\operatorname{ord}_{B_i^{\prime}}(u)$, positive whenever $B_i^{\prime}$ survives on $X$ (see, e.g.,~\cite{CLS11}). 
Therefore, we have 
\[
-K_{T'} \sim_\qq \sum_{i=2}^k (1+\alpha_i) B'_i. 
\]
The center of each $S_i$ on $T'$ lies in one of the $B'_2,\dots,B'_k$ and not inside $B'_1$.
From this, we conclude that
\[
-K_Y \sim_\qq \sum_{i=2}^k (1+\alpha_i)B_{Y,i} + \sum_{i=1}^m \beta_i S_{Y,i}
\]
where $B_{Y,i}$ is the strict transform of $B'_i$ to $Y$
and each $\beta_i=N^{-1}\operatorname{ord}_{S_i}(u)$ is positive. 
Note that 
\[
Y\setminus {\rm Supp}\left( 
\sum_{i=2}^k B_{Y,i} + \sum_{i=1}^m S_{Y,i} 
\right) \simeq \mathbb{G}_m^{n-1} \times \mathbb{A}^1. 
\]
Push the displayed divisor to $D_X\sim_\qq-K_X$. Its surviving coefficients are positive outside $S_X$ and the torus, so $D_X\geq0$. Normality extends the toric coordinates to $g\colon U:=X\setminus\operatorname{Supp}D_X\to V:=\mathbb A^1\times\mathbb G_m^{n-1}$, identifying big opens $U^\circ\simeq V^\circ$. Choose a nonzero Laurent polynomial $f$ vanishing on the closure of the base projection of $(V\setminus V^\circ)\cup g(U\setminus U^\circ)$, whose dimension is at most $n-2$. Then $U\cap\{f\ne0\}\simeq\mathbb A^1\times D(f)$. Since $f$ is regular on $U$, for sufficiently small rational $\epsilon>0$, the divisor $D_X+\epsilon\operatorname{div}_X(f)$ is effective with precisely this complement.
We conclude that $X$ admits an anticanonical cylinder. 
Furthermore, by construction, the cylinder is a cluster cylinder. 
\end{proof} 

We conclude this subsection by applying the ideas of the previous theorem to the context of smooth del Pezzo surfaces. 

\begin{example}
{\em 
In this example, we give a short argument that a smooth del Pezzo
surface of degree at least $4$ admits a polar anticanonical cluster cylinder. The statement is trivial for $\pp^1\times\pp^1$ by considering the toric boundary, so we assume that the surface is a blow-up of $\pp^2$.

Let $T$ be the second Hirzebruch surface.
Let $f_0$ and $f_\infty$ be two fibers over $\pp^1$,
$s_0$ be the section with self-intersection $2$,
and $s_\infty$ be the section with self-intersection $-2$. 
Let $p_1,\dots,p_k$ be $k\leq 5$ be general points in $s_0$.
Let $q$ be a general point in $f_\infty$.
Blow up the points $p_1,\dots,p_k$, and $q$ and obtain 
$\pi\colon X\rightarrow T$.
By abuse of notation, we denote by
$f_0,f_\infty,s_0$, and $s_\infty$  
the strict transforms of these curves in $X$. 
Therefore, the pair $(X,f_0+f_\infty+s_0+s_\infty)$ is log Calabi--Yau. 
We obtain a situation as in the following figure:

\begin{figure}[htbp]
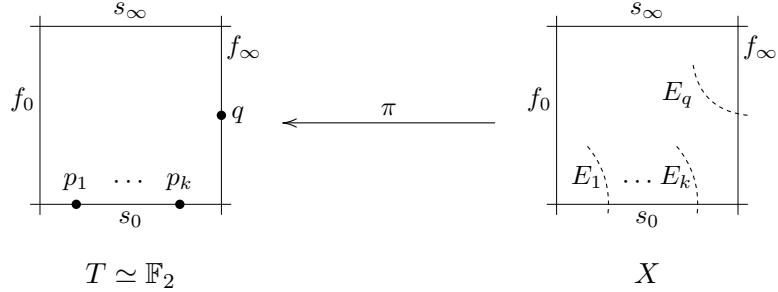

\centering
\[
\BlowupDiagram
\]
\caption{Blow-up of toric model $T\simeq \mathbb{F}_2$.}
\label{fig:my_label}
\end{figure}

Note that in $X$ the curves $f_\infty$ and $s_\infty$ satisfy
$f_\infty^2=-1$ and $s_\infty^2=-2$.
Thus, we can contract $f_\infty$ and then $s_\infty$. 
Let $\varphi \colon X \rightarrow Y$ be the induced birational contraction.
By abuse of notation, denote by $s_0$ and $f_0$ the images of such curves on $Y$. We may also identify the curves $E_1,\dots,E_k$ and $E_q$ with their images in $Y$. By construction, the pair $(Y,s_0+f_0)$ is a log Calabi--Yau pair
and $Y$ is a smooth del Pezzo surface. We can compute
\[
(-K_Y)^2 = (s_0+f_0)^2 = 9-k
\]
so $Y$ is a smooth del Pezzo surface of degree $9-k$. 
The blow-up $Y\to\pp^2$ at the $k$ points maps $s_0+f_0$ crepantly to $C+L$, sending 
$s_0$ to a conic through $k\leq 5$ general points and $f_0$ to a line.
Note that we cannot apply Theorem~\ref{thm:anticanonical-polar-cylinder} in a straightforward way as $f_0$ is not ample in the toric model $T\simeq \mathbb{F}_2$. 
However, in this case we can observe that 
\[
-K_T \sim \frac{2}{3}s_\infty + \frac{4}{3} (s_0+f_\infty), 
\]
where this relation comes from a rational function on the torus $\mathbb{G}_m^2$. 
Using the previous relations, we observe that 
\[
-K_X \sim_\qq \frac{2}{3}s_\infty + \frac{4}{3}\left(s_0+f_\infty\right)+\frac{1}{3}\left(\sum_{i=1}^k E_i+E_q\right). 
\]
Pushing forward the previous relation to $Y$, we see that the cylinder 
$\mathbb{G}_m\times \mathbb{A}^1 \simeq Y\setminus {\rm Supp}(s_0+E_1+\dots+E_k+E_q)$ is a polar anticanonical cluster cylinder of $Y$, since $-K_Y\sim_\qq\frac43s_0+\frac13(\sum_iE_i+E_q)$.
}
\end{example} 

\appendix 

\section{Use of AI disclosure}
\label{appendix-A}

In this appendix, we explain how the tools of Artificial Intelligence (AI), specifically ChatGPT 6.0 were used to understand the cluster type property for some families of smooth Fano threefolds. 

This project is one of the milestones of a long-term project by the authors to understand the cluster type condition for three-dimensional varieties. We started writing this paper in February 2026, however, the ideas within go back to previous papers by the authors, for instance~\cite{Mor24a,Mor24b,MY24,HMY25,AdSFM26,Mor26a}. Around May 2026, the authors had understood the cluster type condition for 81 families of smooth Fano threefolds. Some talks by the third author reported on the status of the project by then. By that point AI was only used for minor tasks unrelated to the intellectual content of the paper. In early June 2026, the authors tried to use AI tools to understand the missing families  \textnumero\, 1.8, \textnumero\,1.9, and \textnumero\,1.10. The ideas of these cases go back to Loginov and are contained in the log rationality proof in~\cite{LMV24}. AI suggested technical steps for the proofs presented here. More precisely, regarding the existence of some lines in Du Val del Pezzo surfaces.
Arguments for the cluster type property for \textnumero\, 2.18, \textnumero\, 3.2 and \textnumero\,3.4 were suggested by AI and subsequently revised by the authors. In all these three cases, AI suggested carefully chosen coregularity zero boundaries which were not the most natural ones considered by the authors. The special element in \textnumero\,2.4 was found by the first two authors as a potential example which is not cluster type. Some observations in this direction were made by the first two authors, but details of the proof including Wronskian computations were provided by AI. Finally, the strategy to prove that a general element in \textnumero\,1.6 is not of cluster type using Lemma~\ref{lem:movable-system} was due to the authors. Such lemma is motivated by the proof of the Noether-Fano inequality. However, to apply Lemma~\ref{lem:movable-system} one needs to understand all the coregularity zero anti-canonical boundaries of the variety. AI proposed a classification of these coregularity zero boundaries for smooth prime Fano threefolds of genus $7$. The argument uses ideas from Manivel, Kuznetsov, Prokhorov, Iliev, and Markushevich. Most of these arguments are related to the ten-dimensional spinor variety. The precise citations are given in the corresponding section.

All the AI-generated arguments were rather long and tedious. In many cases, the authors improved and cleaned the exposition. 
After the manuscript was completed, AI was used for proofreading, without altering the mathematical arguments.
The authors take full responsibility for the material in this paper. 

\section{Table of cluster type smooth Fano threefolds}\label{appendix-B}

In this appendix, we present a table that collects
all families of smooth Fano threefolds as classified by Iskovskikh, Mori, Mukai, and Prokhorov~\cite{MM81,MM03,IP99}. 
The first four columns follow~\cite{LMV24,ALP23}, based on the table of Cheltsov and Shramov~\cite{CS08}. We use this table to better help the reader navigate the paper and know precisely which statement of the paper proves that an element or a general element of such family is of cluster type or not. 
The first column lists all families of smooth Fano threefolds, the second column explains their volumes, the third column gives a brief explanation, and the fourth column states whether the general element of such family is rational.
The last column of the table is new and explains which elements of such family are proved to be of cluster type.

\label{sec-the-table}
\small{
\begin{longtable}{|c|c|p{9cm}|c|c|}
\caption{Cluster type smooth Fano threefolds}\label{table:Fanos}\\
\hline Family & $-K_X^3$ &  Brief description & Rational & Cluster type \\
\hline \textnumero\,$1.1$ & $2$ & A hypersurface in $\mathbb{P}(1,1,1,1,3)$ of
degree $6$. & No & No \\
\hline \textnumero\,$1.2$ & $4$ & A hypersurface in $\mathbb{P}^4$ of degree
$4$. & No & No  \\
\hline \textnumero\,$1.3$ & $6$ & A complete intersection of a quadric and a
cubic in
$\mathbb{P}^{5}$. & Gen. Not & Gen. Not \\
\hline \textnumero\,$1.4$ & $8$ & A complete intersection of three quadrics in
$\mathbb{P}^{6}$. & No &  No \\
\hline \textnumero\,$1.5$ & $10$ & A section of
$\mathrm{Gr}(2,5)\subset\mathbb{P}^9$ by
quadric and linear subspace of dimension~$7$
& Gen. Not &  Gen. Not \\
\hline \textnumero\,$1.6$ & $12$ & A section of the Hermitian symmetric space
$M=G/P\subset\mathbb{P}^{15}$\hfill\break of type DIII  by a
linear
subspace of dimension~$8$. & Yes & \makecell{Gen. Not\\ by Theorem~\ref{thm:1-6-not-cluster-type}} \\
\hline \textnumero\,$1.7$ & $14$ & A section of
$\mathrm{Gr}(2,6)\subset\mathbb{P}^{14}$
by a linear subspace of codimension~$5$. & No & No \\
\hline \textnumero\,$1.8$ & $16$ & A section of the Hermitian symmetric space
$M=G/P\subset \mathbb{P}^{19}$\hfill\break of type CI  by a linear
subspace
of dimension~$10$ & Yes & \makecell{Any element \\ by Proposition~\ref{prop:1-8} } \\
\hline \textnumero\,$1.9$ & $18$ & A section of the $5$-dimensional rational
homogeneous contact\hfill\break manifold
$G_2/P\subset\mathbb{P}^{13}$  by
a linear subspace of dimension~$11$ & Yes & \makecell{Any element \\ by Proposition~\ref{prop:1-9} } \\
\hline \textnumero\,$1.10$ & $22$ & A zero locus of three sections of the rank
$3$ vector bundle $\bigwedge^2\mathcal{Q}$,\hfill\break where
$\mathcal{Q}$ is
the universal quotient bundle on $\mathrm{Gr}(3,7)$ & Yes &
\makecell{Any element \\ by Proposition~\ref{prop:1-10}}  \\
\hline \textnumero\,$1.11$ & $8$ & $V_{1}$ that is a hypersurface in
$\mathbb{P}(1,1,1,2,3)$ of degree $6$ & No & No
 \\
\hline \textnumero\,$1.12$ & $16$ & $V_{2}$ that is a hypersurface in
$\mathbb{P}(1,1,1,1,2)$ of degree $4$ & No & No \\
\hline \textnumero\,$1.13$ & $24$ & $V_{3}$ that is a hypersurface in
$\mathbb{P}^{4}$ of degree $3$ & No & No \\
\hline \textnumero\,$1.14$ & $32$ & $V_{4}$ that is a complete intersection of two
quadrics in $\mathbb{P}^{5}$ & Yes & 
\makecell{Any element \\ by Proposition~\ref{prop:cluster-type-q1capq2}}
\\
\hline \textnumero\,$1.15$ & $40$ & $V_{5}$ that is a section of
$\mathrm{Gr}(2,5)\subset\mathbb{P}^9$ by linear subspace of
codimension $3$ & Yes & \makecell{Any element\\ by Proposition~\ref{prop:cluster-type-section-grass}}\\
\hline \textnumero\,$1.16$ & $54$ & $\mathbb{Q}^3$ that is a hypersurface in $\mathbb{P}^{4}$
of degree $2$ & Yes & \makecell{Any element\\ by
Theorem~\ref{thm:complexity-one-is-cluster-type}} \\
\hline \textnumero\,$1.17$ & $64$ & $\mathbb{P}^{3}$ & Yes & Proposition~\ref{prop:toric} \\
\hline \textnumero\,$2.1$ & $4$ & A blow-up of the Fano threefold $V_1$ along an
elliptic curve\hfill\break that is an intersection of two divisors
from $|-\frac{1}{2}K_{V_1}|$ & No &  No \\
\hline \textnumero\,$2.2$ & $6$ & A double cover of
$\mathbb{P}^1\times\mathbb{P}^2$
whose branch locus is a divisor of bidegree $(2, 4)$ & No & No \\
\hline \textnumero\,$2.3$ & $8$ & The blow-up of the Fano threefold $V_2$ along an
elliptic curve\hfill\break that is an intersection of two divisors
from $|-\frac{1}{2}K_{V_2}|$ & No & No \\
\hline \textnumero\,$2.4$ & $10$ & The blow-up of $\mathbb{P}^3$ along an
intersection of two cubics & Yes & \makecell{General element\\ by Theorem~\ref{thm:blow-up-P3}}\\ 
\hline \textnumero\,$2.5$ & $12$ & The blow-up of the threefold
$V_3\subset\mathbb{P}^4$ along a plane cubic & No & No \\
\hline \textnumero\,$2.6$ & $12$ & A divisor on
$\mathbb{P}^2\times\mathbb{P}^2$ of bidegree $(2, 2)$ & No & No \\
\hline \textnumero\,$2.7$ & $14$ & The blow-up of $Q$ along the intersection of
two divisors from $|\mathcal{O}_Q (2)|$ & Yes & \makecell{Any element\\ by Theorem~\ref{thm:blow-up-Q3}} \\
\hline \textnumero\,$2.8$ & $14$ & A double cover of $V_7$ whose branch locus
is a surface
in $|-K_{V_7}|$ & No & No \\
\hline \textnumero\,$2.9$ & $16$ & The blow-up of $\mathbb{P}^3$ along a curve
of degree $7$ and genus~$5$\hfill\break which is an intersection
of three cubics & Yes & \makecell{Any element\\ by Theorem~\ref{thm:blow-up-P3}} \\
\hline \textnumero\,$2.10$ & $16$ & The blow-up of $V_4\subset\mathbb{P}^5$
along an elliptic curve\hfill\break which is an intersection of
two hyperplane
sections & Yes & \makecell{Any element\\ by Proposition~\ref{prop:cluster-type-q1capq2}}\\
\hline \textnumero\,$2.11$ & $18$ & The blow-up of $V_3$ along a line & No & No
 \\
\hline \textnumero\,$2.12$ & $20$ & The blow-up of $\mathbb{P}^3$ along a curve
of degree $6$ and genus~$3$\hfill\break which is an intersection
of four cubics & Yes & \makecell{Any element\\ by Theorem~\ref{thm:blow-up-P3}} \\
\hline \textnumero\,$2.13$ & $20$ & The blow-up of $\mathbb{Q}^3 \subset\mathbb{P}^4$ along
a curve of degree $6$ and genus $2$ & Yes & \makecell{Any element\\ by Theorem~\ref{thm:blow-up-Q3}}\\
\hline \textnumero\,$2.14$ & $20$ & The blow-up of $V_5\subset\mathbb{P}^6$ along
an elliptic curve\hfill\break which is an intersection of two
hyperplane sections & Yes & \makecell{Any element\\ by Proposition~\ref{prop:cluster-type-section-grass}}\\
\hline \textnumero\,$2.15$ & $22$ & The blow-up of $\mathbb{P}^3$ along the
intersection of a quadric and a cubic surface & Yes & \makecell{Any element\\ by Theorem~\ref{thm:blow-up-P3}} \\
\hline \textnumero\,$2.16$ & $22$ & The blow-up of $V_4\subset\mathbb{P}^5$
along a conic & Yes & \makecell{Any element\\ by Proposition~\ref{prop:cluster-type-q1capq2}} \\
\hline \textnumero\,$2.17$ & $24$ & The blow-up of $\mathbb{Q}^3\subset\mathbb{P}^4$ along
an elliptic curve of degree~$5$ & Yes & \makecell{Any element\\ by Theorem~\ref{thm:blow-up-Q3}} \\
\hline \textnumero\,$2.18$ & $24$ & A double cover of
$\mathbb{P}^1\times\mathbb{P}^2$ whose branch locus is a divisor of
bidegree $(2, 2)$ & Yes & \makecell{Any element\\ by Proposition~\ref{prop:double-cover-p1xp2-cluster-type}} \\
\hline \textnumero\,$2.19$ & $26$ & The blow-up of $V_4\subset\mathbb{P}^5$ along a
line & Yes & \makecell{Any element\\ by Proposition~\ref{prop:cluster-type-q1capq2}} \\
\hline \textnumero\,$2.20$ & $26$ & The blow-up of $V_5\subset\mathbb{P}^6$
along a twisted cubic & Yes & \makecell{Any element\\ by Proposition~\ref{prop:cluster-type-section-grass}} \\
\hline \textnumero\,$2.21$ & $28$ & The blow-up of $\mathbb{Q}^3 \subset\mathbb{P}^4$ along
a twisted quartic & Yes & \makecell{Any element\\ by Theorem~\ref{thm:blow-up-Q3}} \\
\hline \textnumero\,$2.22$ & $30$ & The blow-up of $V_5\subset\mathbb{P}^6$
along a conic & Yes & \makecell{Any element\\ by Proposition~\ref{prop:cluster-type-section-grass}} \\
\hline \textnumero\,$2.23$ & $30$ & The blow-up of $\mathbb{Q}^3 \subset\mathbb{P}^4$ along a
curve of degree $4$ that is an intersection of a surface in
$|\mathcal{O}_{\mathbb{Q}^3}(1)|$ and a surface in
$|\mathcal{O}_{\mathbb{Q}^3}(2)|$ & Yes & \makecell{Any element\\ by Theorem~\ref{thm:blow-up-Q3}} \\
\hline \textnumero\,$2.24$ & $30$ & A divisor on $\mathbb{P}^2\times\mathbb{P}^2$
of bidegree $(1, 2)$ & Yes & \makecell{Any element\\ by Proposition~\ref{prop:(1,2)-in-p2xp2}} \\
\hline \textnumero\,$2.25$ & $32$ & The blow-up of $\mathbb{P}^3$ along an elliptic
curve which is an intersection of two quadrics & Yes & \makecell{Any element\\ by Theorem~\ref{thm:blow-up-P3}} \\
\hline \textnumero\,$2.26$ & $34$ & The blow-up of the threefold
$V_5\subset\mathbb{P}^6$ along a line & Yes & \makecell{Any element\\ by Proposition~\ref{prop:cluster-type-section-grass}} \\
\hline \textnumero\,$2.27$ & $38$ & The blow-up of $\mathbb{P}^3$ along a twisted
cubic  & Yes & \makecell{Any element\\ by Theorem~\ref{thm:blow-up-P3}} \\
\hline \textnumero\,$2.28$ & $40$ & The blow-up of $\mathbb{P}^3$ along a plane
cubic &Yes & \makecell{Any element\\ by Theorem~\ref{thm:blow-up-P3}} \\
\hline \textnumero\,$2.29$ & $40$ & The blow-up of $\mathbb{Q}^3 \subset\mathbb{P}^4$ along a
conic & Yes & \makecell{Any element\\ by Theorem~\ref{thm:blow-up-Q3}} \\
\hline \textnumero\,$2.30$ & $46$ & The blow-up of $\mathbb{P}^3$ along a conic &
Yes &\makecell{Any element\\ by Theorem~\ref{thm:blow-up-P3}} \\
\hline \textnumero\,$2.31$ & $46$ & The blow-up of $\mathbb{Q}^3 \subset\mathbb{P}^4$ along a
line & Yes &\makecell{Any element\\ by Theorem~\ref{thm:blow-up-Q3}} \\
\hline \textnumero\,$2.32$ & $48$ & $W$ that is a divisor on
$\mathbb{P}^2\times\mathbb{P}^2$ of bidegree $(1, 1)$ & Yes &
\makecell{Any element\\ by Proposition~\ref{prop:(1,1)-in-p2xp2}} \\
\hline \textnumero\,$2.33$ & $54$ & The blow-up of $\mathbb{P}^3$ along a line  &
Yes & \makecell{Any element\\ by Proposition~\ref{prop:toric}} \\
\hline \textnumero\,$2.34$ & $54$ & $\mathbb{P}^1\times\mathbb{P}^2$  & Yes &
Proposition~\ref{prop:toric} \\
\hline \textnumero\,$2.35$ & $56$ &
$V_7\cong\mathbb{P}(\mathcal{O}_{\mathbb{P}^2}\oplus\mathcal{O}_{\mathbb{P}^2}(1))$
 &Yes & Proposition~\ref{prop:toric} \\
\hline \textnumero\,$2.36$ & $62$ &
$\mathbb{P}(\mathcal{O}_{\mathbb{P}^2}\oplus\mathcal{O}_{\mathbb{P}^2}(2))$
 & Yes & Proposition~\ref{prop:toric} \\
\hline \textnumero\,$3.1$ & $12$ & A double cover of
$\mathbb{P}^1\times\mathbb{P}^1\times\mathbb{P}^1$ branched in a
divisor of type $(2, 2, 2)$ & No & No \\
\hline \textnumero\,$3.2$ & $14$ & A divisor on a $\mathbb{P}^{2}$-bundle
$\mathbb{P}(\mathcal{O}_{\mathbb{P}^1\times\mathbb{P}^1}\oplus\mathcal{O}_{\mathbb{P}^1\times\mathbb{P}^1}(-1,-1)\oplus\mathcal{O}_{\mathbb{P}^1\times\mathbb{P}^1}(-1,-1))$
such that $X\in|L^{{}\otimes
2}\otimes\mathcal{O}_{\mathbb{P}^{1}\times\mathbb{P}^{1}}(2,3)|$,
where $L$ is the tautological line bundle & Yes & \makecell{Any element\\ by Proposition~\ref{prop:3-2_cluster_type}}\\
\hline \textnumero\,$3.3$ & $18$ & A divisor on
$\mathbb{P}^1\times\mathbb{P}^1\times\mathbb{P}^2$ of type $(1,
1, 2)$ & Yes & \makecell{Any element\\ by Proposition~\ref{prop:(1,1,2)-multi}} \\
\hline \textnumero\,$3.4$ & $18$ & The blow-up of the Fano threefold $Y$ from the
family \textnumero\,2.18 along a smooth fiber of the
composition $Y\to\mathbb{P}^1\times\mathbb{P}^2\to\mathbb{P}^2$ of the
double cover with the projection & Yes & \makecell{Any element\\ by Corollary~\ref{coro:3.4}}
\\
\hline \textnumero\,$3.5$ & $20$ & The blow-up of $\mathbb{P}^1\times\mathbb{P}^2$
along a curve $C$ of bidegree $(5, 2)$\hfill\break such that the
composition  $C\hookrightarrow\mathbb{P}^1\times\mathbb{P}^2\to\mathbb{P}^2$
is an embedding & Yes & Proposition~\ref{prop:blow-up-p1xp2} \\
\hline \textnumero\,$3.6$ & $22$ & The blow-up of $\mathbb{P}^3$ along a disjoint
union of a line and an elliptic curve of degree~$4$ & Yes & \makecell{Any element\\
by Theorem~\ref{thm:blow-up-P3}} \\
\hline \textnumero\,$3.7$ & $24$ & The blow-up of the threefold $W$ along an
elliptic curve\hfill\break that is an intersection of two  divisors
from $|-\frac{1}{2}K_W|$ & Yes & \makecell{Any element\\ by Proposition~\ref{prop:blow-up-of-(1,1)-in-P2xP2}}
\\
\hline \textnumero\,$3.8$ & $24$ & A divisor in
$|(\alpha\circ\pi_1)^*(\mathcal{O}_{\mathbb{P}^2}(1))\otimes\pi_2^*(\mathcal{O}_{\mathbb{P}^2}(2))|$,
where $\pi_{1}\colon\mathbb{F}_1\times\mathbb{P}^2\to\mathbb{F}_1$
and $\pi_{2}\colon\mathbb{F}_1\times\mathbb{P}^2\to\mathbb{P}^2$ are
projections, and $\alpha\colon\mathbb{F}_1\to\mathbb{P}^2$ is a blow
up of a point & Yes & \makecell{Any element\\ by Proposition~\ref{prop:3-8}} \\

\hline \textnumero\,$3.9$ & $26$ & The blow-up of a cone $W_4\subset\mathbb{P}^6$
over the Veronese surface  $R_4\subset\mathbb{P}^5$ with
center in a disjoint union of the vertex and  a quartic \hfill\break on
$R_4\cong\mathbb{P}^2$ & Yes & \makecell{Any element\\ by Proposition~\ref{prop:3-9}} \\

\hline \textnumero\,$3.10$ & $26$ & The blow-up of $\mathbb{Q}^3 \subset\mathbb{P}^4$ along a
disjoint union of two conics & Yes & \makecell{Any element\\ by Theorem~\ref{thm:blow-up-Q3}} \\

\hline \textnumero\,$3.11$ & $28$ & The blow-up of the threefold $V_7$ along an
elliptic curve\hfill\break that is an intersection of  two divisors
from $|-\frac{1}{2}K_{V_7}|$ & Yes & \makecell{Any element\\ by Theorem~\ref{thm:blow-up-P3}}  \\

\hline \textnumero\,$3.12$ & $28$ & The blow-up of $\mathbb{P}^3$ along a disjoint
union of a line and a twisted cubic & Yes & \makecell{Any element\\ by Theorem~\ref{thm:blow-up-P3}} \\

\hline \textnumero\,$3.13$ & $30$ & The blow-up of
$W\subset\mathbb{P}^2\times\mathbb{P}^2$ along a curve $C$ of
bidegree $(2, 2)$\hfill\break such that
$\pi_{1}(C)\subset\mathbb{P}^2$ and
$\pi_{2}(C)\subset\mathbb{P}^{2}$ are irreducible
conics,\hfill\break where $\pi_{1}\colon W\to\mathbb{P}^2$ and
$\pi_{2}\colon W\to\mathbb{P}^2$ are
natural projections & Yes & \makecell{Any element\\ by Proposition~\ref{prop:3-13}} \\

\hline \textnumero\,$3.14$ & $32$ & The blow-up of $\mathbb{P}^3$ along a
disjoint union of a plane cubic curve that is
contained in a plane
$\Pi\subset\mathbb{P}^{3}$ and a point that is not contained in $\Pi$
& Yes & \makecell{Any element\\ by Theorem~\ref{thm:blow-up-P3}} \\

\hline \textnumero\,$3.15$ & $32$ & The blow-up of $\mathbb{Q}^3 \subset\mathbb{P}^4$ along a
disjoint union of a line and a conic & Yes & \makecell{Any element\\ by Theorem~\ref{thm:blow-up-Q3}} \\

\hline \textnumero\,$3.16$ & $34$ & The blow-up of $V_7$ along a proper
transform via the blow-up $\alpha\colon
V_7\to\mathbb{P}^3$ of a twisted cubic
passing through the center of the blow-up $\alpha$ & Yes & \makecell{Any element\\ by Theorem~\ref{thm:blow-up-P3}} \\

\hline \textnumero\,$3.17$ & $36$ & A divisor on
$\mathbb{P}^1\times\mathbb{P}^1\times\mathbb{P}^2$ of type $(1,
1, 1)$ & Yes & \makecell{Any element\\ by Proposition~\ref{prop:3-17}} \\

\hline \textnumero\,$3.18$ & $36$ & The blow-up of $\mathbb{P}^3$ along a disjoint
union of a line and a conic & Yes & \makecell{Any element\\ by Theorem~\ref{thm:blow-up-P3}}\\

\hline \textnumero\,$3.19$ & $38$ & The blow-up of $\mathbb{Q}^3 \subset\mathbb{P}^4$ at two
points whose joining line is not in $\mathbb{Q}^3$ & Yes & \makecell{Any element\\ by Theorem~\ref{thm:blow-up-Q3}} \\

\hline \textnumero\,$3.20$ & $38$ & The blow-up of $\mathbb{Q} \subset\mathbb{P}^4$ along a
disjoint union of two lines & Yes& \makecell{Any element\\ by Theorem~\ref{thm:blow-up-Q3}} \\

\hline \textnumero\,$3.21$ & $38$ & The blow-up of $\mathbb{P}^1\times\mathbb{P}^2$
along a curve of bidegree  $(2, 1)$ & Yes & \makecell{Any element\\ by Proposition~\ref{prop:3-21}}  \\

\hline \textnumero\,$3.22$ & $40$ & The blow-up of $\mathbb{P}^1\times\mathbb{P}^2$
along a conic in a fiber of the projection
$\mathbb{P}^{1}\times\mathbb{P}^2\to\mathbb{P}^1$ & Yes & \makecell{Any element\\ by Proposition~\ref{prop:3-22}} \\

\hline \textnumero\,$3.23$ & $42$ & The blow-up of $V_7$ along a proper transform
via the blow-up $\alpha\colon V_7\to\mathbb{P}^3$ of an
irreducible conic passing through the center of the blow-up $\alpha$ &
Yes & \makecell{Any element\\ by Theorem~\ref{thm:blow-up-P3}}  \\

\hline \textnumero\,$3.24$ & $42$ & $W\times_{\mathbb{P}^2}\mathbb{F}_1$, where
$W\to\mathbb{P}^2$ is a $\mathbb{P}^1$-bundle and
$\mathbb{F}_1\to\mathbb{P}^2$ is the
blow-up & Yes & \makecell{Any element\\ by Proposition~\ref{prop:3-24}}  \\

\hline \textnumero\,$3.25$ & $44$ & The blow-up of $\mathbb{P}^3$ along a disjoint
union of two lines & Yes & \makecell{Any element\\ by Proposition~\ref{prop:toric}}  \\

\hline \textnumero\,$3.26$ & $46$ & The blow-up of $\mathbb{P}^3$ with center in a
disjoint union of a point and a line  & Yes & \makecell{Any element\\ by Proposition~\ref{prop:toric}}  \\

\hline \textnumero\,$3.27$ & $48$ &
$\mathbb{P}^1\times\mathbb{P}^1\times\mathbb{P}^1$  & Yes &
Proposition~\ref{prop:toric} \\

\hline \textnumero\,$3.28$ & $48$ & $\mathbb{P}^1\times\mathbb{F}_1$  & Yes &
Proposition~\ref{prop:toric} \\

\hline \textnumero\,$3.29$ & $50$ & The blow-up of the Fano threefold $V_7$
along a line in $E\cong\mathbb{P}^2$,\hfill\break where $E$ is the
exceptional divisor of the
blow-up $V_7\to\mathbb{P}^3$  & Yes & Proposition~\ref{prop:toric} \\

\hline \textnumero\,$3.30$ & $50$ & The blow-up of $V_7$ along a proper transform
via the blow-up $\alpha\colon V_7\to\mathbb{P}^3$ of a
line that passes through the center of the blow-up $\alpha$  & Yes &
Proposition~\ref{prop:toric} \\

\hline \textnumero\,$3.31$ & $52$ & The blow-up of a cone over a smooth quadric in
$\mathbb{P}^3$ at the vertex & Yes & Proposition~\ref{prop:toric} \\

\hline \textnumero\,$4.1$ & $24$ & Divisor on
$\mathbb{P}^1\times\mathbb{P}^1\times\mathbb{P}^1\times\mathbb{P}^1$
of multidegree $(1, 1, 1, 1)$ & Yes & \makecell{Any element\\ by Proposition~\ref{prop:4.1}} \\

\hline \textnumero\,$4.2$ & $28$ & The blow-up of the cone over a smooth quadric
$S\subset\nolinebreak\mathbb{P}^3$\hfill\break along a disjoint union
of the vertex and an elliptic curve on $S$ & Yes &\makecell{Any element\\ by Proposition~\ref{prop:4.2}}\\

\hline \textnumero\,$4.3$ & $30$ & The blow-up of
$\mathbb{P}^1\times\mathbb{P}^1\times\mathbb{P}^1$ along a curve of
type $(1, 1, 2)$  & Yes & \makecell{Any element\\ by Proposition~\ref{prop:4.3-13}} \\

\hline \textnumero\,$4.4$ & $32$ & The blow-up of the smooth Fano threefold $Y$
from the family \textnumero\,3.19 along the proper
transform of a conic on the quadric $\mathbb{Q}^3 \subset\mathbb{P}^4$
that passes through both centers \hfill\break of the blow-up $Y\to Q$ & Yes &
\makecell{Any element\\ by Theorem~\ref{thm:blow-up-Q3}} \\

\hline \textnumero\,$4.5$ & $32$ & The blow-up of $\mathbb{P}^1\times\mathbb{P}^2$
along a disjoint union of\hfill\break two irreducible curves of
bidegree $(2, 1)$ and $(1, 0)$ & Yes & \makecell{Any element\\ by Proposition~\ref{prop:T-comp-one}} \\

\hline \textnumero\,$4.6$ & $34$ & The blow-up of $\mathbb{P}^3$ along a disjoint
union of three lines & Yes & \makecell{Any element\\ by Theorem~\ref{thm:blow-up-P3}} \\

\hline \textnumero\,$4.7$ & $36$ & The blow-up of
$W\subset\mathbb{P}^2\times\mathbb{P}^2$ along a disjoint union
of\hfill\break two curves of bidegree $(0, 1)$ and $(1, 0)$ & Yes &
\makecell{Any element\\ by Proposition~\ref{prop:T-comp-one}} \\

\hline \textnumero\,$4.8$ & $38$ & The blow-up of
$\mathbb{P}^1\times\mathbb{P}^1\times\mathbb{P}^1$ along a curve of
type $(0, 1, 1)$  & Yes & \makecell{Any element\\ by Proposition~\ref{prop:T-comp-one}} \\

\hline \textnumero\,$4.9$ & $40$ & The blow-up of the smooth Fano threefold $Y$
from the family
\textnumero\,3.25 along a curve that is contracted
by the blow
up $Y \to\mathbb{P}^3$ & Yes & Proposition~\ref{prop:toric} \\

\hline \textnumero\,$4.10$ & $42$ & $\mathbb{P}^1\times S_7$ & Yes &
Proposition~\ref{prop:toric} \\

\hline \textnumero\,$4.11$ & $44$ & The blow-up of
$\mathbb{P}^1\times\mathbb{F}_1$ along a curve $C\cong\mathbb{P}^1$ such
that $C$ is contained in a fiber $F\cong\mathbb{F}_{1}$ of
the projection
$\mathbb{P}^1\times\mathbb{F}_{1}\to\mathbb{P}^1$ and $C\cdot C=-1$ on
$F$ & Yes & Proposition~\ref{prop:toric} \\

\hline \textnumero\,$4.12$ & $46$ & The blow-up of the smooth Fano threefold
$Y$ from the family \textnumero\,2.33 along two curves that are
contracted by the blow-up
$Y\to\mathbb{P}^3$ & Yes & Proposition~\ref{prop:toric} \\

\hline \textnumero\,$4.13$ & $26$ & The blow-up of
$\mathbb{P}^1\times\mathbb{P}^1\times\mathbb{P}^1$ along a curve of
type $(1, 1, 3)$ & Yes & \makecell{Any element\\ by Proposition~\ref{prop:4.3-13}} \\

\hline \textnumero\,$5.1$ & $28$ & The blow-up of the smooth Fano threefold $Y$
from the family \textnumero\,2.29 along three curves that are
contracted by the
blow-up $Y\to Q$ & Yes & \makecell{Any element\\ by Theorem~\ref{thm:blow-up-Q3}} \\

\hline \textnumero\,$5.2$ & $36$ & The blow-up of the smooth Fano threefold $Y$
in the family $3.25$ along two curves $C_{1}\ne
C_{2}$ that are contracted by the blow-up $\phi\colon
Y\to\nolinebreak\mathbb{P}^3$ and that are contained in the
same exceptional divisor of the blow-up $\phi$  &
Yes & Proposition~\ref{prop:toric} \\

\hline \textnumero\,$5.3$ & $36$ & $\mathbb{P}^1\times S_6$  & Yes &
Proposition~\ref{prop:toric} \\
\hline \textnumero\,$6.1$ & $30$ & $\mathbb{P}^1\times S_5$  & Yes &
Lemma~\ref{lem:product-cluster-type} \\
\hline \textnumero\,$7.1$ & $24$ & $\mathbb{P}^1\times S_4$  & Yes &
Lemma~\ref{lem:product-cluster-type}  \\
\hline \textnumero\,$8.1$ & $18$ & $\mathbb{P}^1\times S_3$  & Yes &
Lemma~\ref{lem:product-cluster-type}  \\
\hline \textnumero\,$9.1$ & $12$ & $\mathbb{P}^1\times S_2$  & Yes &
Lemma~\ref{lem:product-cluster-type}  \\
\hline \textnumero\,$10.1$ & $6$ & $\mathbb{P}^1\times S_1$  & Yes &
\makecell{General element \\ by Lemma~\ref{lem:product-cluster-type}}  \\
\hline
\end{longtable}
}

\bibliographystyle{habbvr}
\bibliography{bib}

@incollection {Rei83,
    AUTHOR = {Reid, Miles},
     TITLE = {Minimal models of canonical {$3$}-folds},
 BOOKTITLE = {Algebraic varieties and analytic varieties ({T}okyo, 1981)},
    SERIES = {Adv. Stud. Pure Math.},
    VOLUME = {1},
     PAGES = {131--180},
 PUBLISHER = {North-Holland, Amsterdam},
      YEAR = {1983},
      ISBN = {0-444-86612-4},
   MRCLASS = {14E30 (14J30)},
  MRNUMBER = {715649},
       DOI = {10.2969/aspm/00110131},
       URL = {https://doi.org/10.2969/aspm/00110131},
}

@article {SaintDonat72,
    AUTHOR = {Saint-Donat, Bernard},
     TITLE = {Sur les \'equations d\'efinissant une courbe alg\'ebrique},
   JOURNAL = {C. R. Acad. Sci. Paris S\'er. A-B},
  FJOURNAL = {Comptes Rendus Hebdomadaires des S\'eances de l'Acad\'emie des
              Sciences. S\'eries A et B},
    VOLUME = {274},
      YEAR = {1972},
     PAGES = {A487--A489},
      ISSN = {0151-0509},
   MRCLASS = {14.20},
  MRNUMBER = {289517},
MRREVIEWER = {Y.\ Nakai},
}

@article{GP82,
author = {Gruson, Laurent and Peskine, Christian},
journal = {Annales scientifiques de l'École Normale Supérieure},
language = {fre},
number = {3},
pages = {401-418},
publisher = {Elsevier},
title = {Genre des courbes de l'espace projectif. II},
url = {http://eudml.org/doc/82100},
volume = {15},
year = {1982},
}

@incollection {AIPSV12,
    AUTHOR = {Altmann, Klaus and Ilten, Nathan Owen and Petersen, Lars and
              S\"{u}\ss , Hendrik and Vollmert, Robert},
     TITLE = {The geometry of {$T$}-varieties},
 BOOKTITLE = {Contributions to algebraic geometry},
    SERIES = {EMS Ser. Congr. Rep.},
     PAGES = {17--69},
 PUBLISHER = {Eur. Math. Soc., Z\"{u}rich},
      YEAR = {2012},
   MRCLASS = {14L30 (14M25)},
  MRNUMBER = {2975658},
MRREVIEWER = {Alessandro Ruzzi},
       DOI = {10.4171/114-1/2},
       URL = {https://doi.org/10.4171/114-1/2},
}

@article {AH06,
    AUTHOR = {Altmann, Klaus and Hausen, J\"{u}rgen},
     TITLE = {Polyhedral divisors and algebraic torus actions},
   JOURNAL = {Math. Ann.},
  FJOURNAL = {Mathematische Annalen},
    VOLUME = {334},
      YEAR = {2006},
    NUMBER = {3},
     PAGES = {557--607},
      ISSN = {0025-5831},
   MRCLASS = {14L30 (13A02 14L24 14R05 14R20)},
  MRNUMBER = {2207875},
MRREVIEWER = {Dmitry A. Timash\"{e}v},
       DOI = {10.1007/s00208-005-0705-8},
       URL = {https://doi.org/10.1007/s00208-005-0705-8},
}

@article {BCHM10,
    AUTHOR = {Birkar, Caucher and Cascini, Paolo and Hacon, Christopher D.
              and McKernan, James},
     TITLE = {Existence of minimal models for varieties of log general type},
   JOURNAL = {J. Amer. Math. Soc.},
  FJOURNAL = {Journal of the American Mathematical Society},
    VOLUME = {23},
      YEAR = {2010},
    NUMBER = {2},
     PAGES = {405--468},
      ISSN = {0894-0347},
   MRCLASS = {14E30 (14E05)},
  MRNUMBER = {2601039},
MRREVIEWER = {Mark Gross},
       DOI = {10.1090/S0894-0347-09-00649-3},
       URL = {https://doi.org/10.1090/S0894-0347-09-00649-3}  ,
}

@article {IM07,
    AUTHOR = {Iliev, Atanas and Markushevich, Dimitri},
     TITLE = {Parametrization of sing {$\Theta$} for a {F}ano 3-fold of
              genus 7 by moduli of vector bundles},
   JOURNAL = {Asian J. Math.},
  FJOURNAL = {Asian Journal of Mathematics},
    VOLUME = {11},
      YEAR = {2007},
    NUMBER = {3},
     PAGES = {427--458},
      ISSN = {1093-6106,1945-0036},
   MRCLASS = {14J60 (14J30 14J45)},
  MRNUMBER = {2372725},
MRREVIEWER = {Bernd\ Kreussler},
       DOI = {10.4310/AJM.2007.v11.n3.a4},
       URL = {https://doi.org/10.4310/AJM.2007.v11.n3.a4},
}

@misc{HMY25,
 author={Hacon, Christopher and Moraga, Joaqu\'in and Y\'a\~nez, Jos\'e Ignacio},
 title={{Bott--Chern} complexity of {K\"ahler} pairs},
 year={2025}, eprint={2505.04303}, archivePrefix={arXiv}, primaryClass={math.AG}
}

@misc{Vir26,
      title={Cylinders in {F}ano threefolds of genus 9 and 10},
      author={Nikita Virin},
      year={2026},
      eprint={2605.30875},
      archivePrefix={arXiv},
      primaryClass={math.AG}
}

@incollection {Isk04,
    AUTHOR = {Iskovskikh, Vasilij A.},
     TITLE = {On the {N}oether-{F}ano inequalities},
 BOOKTITLE = {The {F}ano {C}onference},
     PAGES = {25--35},
 PUBLISHER = {Univ. Torino, Turin},
      YEAR = {2004},
      ISBN = {88-900876-1-7},
   MRCLASS = {14E05 (14E07 14E30)},
  MRNUMBER = {2112565},
MRREVIEWER = {St\'ephane\ Lamy},
}

@article {KP18,
    AUTHOR = {Kuznetsov, Alexander and Prokhorov, Yuri},
     TITLE = {Prime {F}ano threefolds of genus 12 with a {$\Bbb G_{\rm
              m}$}-action and their automorphisms},
   JOURNAL = {\'Epijournal G\'eom. Alg\'ebrique},
  FJOURNAL = {\'Epijournal de G\'eom\'etrie Alg\'ebrique. EPIGA},
    VOLUME = {2},
      YEAR = {2018},
     PAGES = {Art. 3, 14},
      ISSN = {2491-6765},
   MRCLASS = {14J45 (14J40 14J50)},
  MRNUMBER = {3816899},
MRREVIEWER = {Arnaud\ Beauville},
       DOI = {10.46298/epiga.2018.volume2.4179},
       URL = {https://doi.org/10.46298/epiga.2018.volume2.4179},
}

@article {DFK26,
    AUTHOR = {Dubouloz, Adrien and Fujita, Kento and Kishimoto, Takashi},
     TITLE = {Another view on smooth prime {F}ano threefolds of degree
              {$22$} with infinite automorphism groups},
   JOURNAL = {Mat. Sb.},
  FJOURNAL = {Matematicheski\u i\ Sbornik},
    VOLUME = {217},
      YEAR = {2026},
    NUMBER = {6},
     PAGES = {3--33},
      ISSN = {0368-8666,2305-2783},
   MRCLASS = {14J45 (14J50)},
  MRNUMBER = {5119286},
       DOI = {10.4213/sm10368},
       URL = {https://doi.org/10.4213/sm10368},
}

@article {CF88,
    AUTHOR = {Coray, D. F. and Tsfasman, M. A.},
     TITLE = {Arithmetic on singular {D}el {P}ezzo surfaces},
   JOURNAL = {Proc. London Math. Soc. (3)},
  FJOURNAL = {Proceedings of the London Mathematical Society. Third Series},
    VOLUME = {57},
      YEAR = {1988},
    NUMBER = {1},
     PAGES = {25--87},
      ISSN = {0024-6115,1460-244X},
   MRCLASS = {11G35 (11E04 11E95 11G25 14J17 14J20 14J26)},
  MRNUMBER = {940430},
MRREVIEWER = {Noriko\ Yui},
       DOI = {10.1112/plms/s3-57.1.25},
       URL = {https://doi.org/10.1112/plms/s3-57.1.25},
}

@article {Der14,
    AUTHOR = {Derenthal, Ulrich},
     TITLE = {Singular del {P}ezzo surfaces whose universal torsors are
              hypersurfaces},
   JOURNAL = {Proc. Lond. Math. Soc. (3)},
  FJOURNAL = {Proceedings of the London Mathematical Society. Third Series},
    VOLUME = {108},
      YEAR = {2014},
    NUMBER = {3},
     PAGES = {638--681},
      ISSN = {0024-6115,1460-244X},
   MRCLASS = {14J26 (14C20 14G05 14J17)},
  MRNUMBER = {3180592},
MRREVIEWER = {Antonio\ Laface},
       DOI = {10.1112/plms/pdt041},
       URL = {https://doi.org/10.1112/plms/pdt041},
}

@article {KPS18,
    AUTHOR = {Kuznetsov, Alexander G. and Prokhorov, Yuri G. and Shramov,
              Constantin A.},
     TITLE = {Hilbert schemes of lines and conics and automorphism groups of
              {F}ano threefolds},
   JOURNAL = {Jpn. J. Math.},
  FJOURNAL = {Japanese Journal of Mathematics},
    VOLUME = {13},
      YEAR = {2018},
    NUMBER = {1},
     PAGES = {109--185},
      ISSN = {0289-2316,1861-3624},
   MRCLASS = {14J45 (14C05 14J30 14J50)},
  MRNUMBER = {3776469},
MRREVIEWER = {Alexandr\ V.\ Pukhlikov},
       DOI = {10.1007/s11537-017-1714-6},
       URL = {https://doi.org/10.1007/s11537-017-1714-6},
}

@article {CS18,
    AUTHOR = {Cheltsov, Ivan and Shramov, Constantin},
     TITLE = {K\"ahler-{E}instein {F}ano threefolds of degree 22},
   JOURNAL = {J. Algebraic Geom.},
  FJOURNAL = {Journal of Algebraic Geometry},
    VOLUME = {32},
      YEAR = {2023},
    NUMBER = {3},
     PAGES = {385--428},
      ISSN = {1056-3911,1534-7486},
   MRCLASS = {14J45 (32Q20 53C55)},
  MRNUMBER = {4622256},
MRREVIEWER = {Feng\ Shao},
       DOI = {10.1090/jag/812},
       URL = {https://doi.org/10.1090/jag/812},
}

@misc{IKTT24,
      title={Arithmetic finiteness of {M}ukai varieties of genus 7},
      author={Tetsushi Ito and Akihiro Kanemitsu and Teppei Takamatsu and Yuuji Tanaka},
      year={2024},
      eprint={2409.20046},
      archivePrefix={arXiv},
      primaryClass={math.AG}
}

@article {PW10,
    AUTHOR = {Park, Jihun and Won, Joonyeong},
     TITLE = {Log-canonical thresholds on del {P}ezzo surfaces of degrees
              {$\geq 2$}},
   JOURNAL = {Nagoya Math. J.},
  FJOURNAL = {Nagoya Mathematical Journal},
    VOLUME = {200},
      YEAR = {2010},
     PAGES = {1--26},
      ISSN = {0027-7630,2152-6842},
   MRCLASS = {14J26 (14J17 14J45)},
  MRNUMBER = {2747875},
MRREVIEWER = {Ivan\ Chel\cprime tsov},
       DOI = {10.1215/00277630-2010-009},
       URL = {https://doi.org/10.1215/00277630-2010-009},
}

@article {Pro25,
    AUTHOR = {Prokhorov, Yuri G.},
     TITLE = {Fano threefolds},
   JOURNAL = {Proc. Steklov Inst. Math.},
  FJOURNAL = {Proceedings of the Steklov Institute of Mathematics},
    VOLUME = {328},
      YEAR = {2025},
     PAGES = {S1--S130},
      ISSN = {0081-5438,1531-8605},
   MRCLASS = {14J45 (14E30 14J30)},
  MRNUMBER = {4922010},
       DOI = {10.1134/s0081543825020014},
       URL = {https://doi.org/10.1134/s0081543825020014},
}

@article {Kuz18b,
    AUTHOR = {Kuznetsov, A. G.},
     TITLE = {On linear sections of the spinor tenfold. {I}},
   JOURNAL = {Izv. Ross. Akad. Nauk Ser. Mat.},
  FJOURNAL = {Izvestiya Rossiiskoi Akademii Nauk. Seriya Matematicheskaya},
    VOLUME = {82},
      YEAR = {2018},
    NUMBER = {4},
     PAGES = {53--114},
      ISSN = {1607-0046,2587-5906},
   MRCLASS = {14M17 (14C05 14F05 14M15 18E30)},
  MRNUMBER = {3833474},
MRREVIEWER = {Nicolae\ Manolache},
       DOI = {10.4213/im8756},
       URL = {https://doi.org/10.4213/im8756},
}

@misc{MZ25,
 author={Moraga, Joaqu\'in and Z\'u\~niga, Juan Pablo},
 title={Degenerations of cluster type varieties},
 year={2025}, eprint={2511.14959}, archivePrefix={arXiv}, primaryClass={math.AG}
}

@article {Sho79,
    AUTHOR = {Shokurov, V. V.},
     TITLE = {Smoothness of a general anticanonical divisor on a {F}ano
              variety},
   JOURNAL = {Izv. Akad. Nauk SSSR Ser. Mat.},
  FJOURNAL = {Izvestiya Akademii Nauk SSSR. Seriya Matematicheskaya},
    VOLUME = {43},
      YEAR = {1979},
    NUMBER = {2},
     PAGES = {430--441},
      ISSN = {0373-2436},
   MRCLASS = {14J30},
  MRNUMBER = {534602},
MRREVIEWER = {Werner\ Kleinert},
}

@misc{Sus13,
      title={Fano threefolds with 2-torus action - a picture book}, 
      author={S\"u\ss, Hendrik},
      year={2013},
      eprint={1308.2379}        ,
      archivePrefix={arXiv},
      primaryClass={math.AG}
}

@misc{Mor26a,
      title={Cluster type varieties}, 
      author={Joaqu\'in Moraga},
      year={2026},
      eprint={2602.23584}        ,
      archivePrefix={arXiv},
      primaryClass={math.AG}
}

@incollection {KPZ11,
    AUTHOR = {Kishimoto, Takashi and Prokhorov, Yuri and Zaidenberg,
              Mikhail},
     TITLE = {Group actions on affine cones},
 BOOKTITLE = {Affine algebraic geometry},
    SERIES = {CRM Proc. Lecture Notes},
    VOLUME = {54},
     PAGES = {123--163},
 PUBLISHER = {Amer. Math. Soc., Providence, RI},
      YEAR = {2011},
      ISBN = {978-0-8218-7283-3},
   MRCLASS = {14R20 (14J26 14J50)},
  MRNUMBER = {2768637},
MRREVIEWER = {J\'er\'emy\ Blanc},
       DOI = {10.1090/crmp/054/08},
       URL = {https://doi.org/10.1090/crmp/054/08},
}

@article {KPZ13,
    AUTHOR = {Kishimoto, Takashi and Prokhorov, Yuri and Zaidenberg,
              Mikhail},
     TITLE = {{$\Bbb G_{\rm a}$}-actions on affine cones},
   JOURNAL = {Transform. Groups},
  FJOURNAL = {Transformation Groups},
    VOLUME = {18},
      YEAR = {2013},
    NUMBER = {4},
     PAGES = {1137--1153},
      ISSN = {1083-4362,1531-586X},
   MRCLASS = {14C20 (14J45 14R20)},
  MRNUMBER = {3127989},
MRREVIEWER = {Karol\ Palka},
       DOI = {10.1007/s00031-013-9246-5},
       URL = {https://doi.org/10.1007/s00031-013-9246-5},
}

@article {KPZ14,
    AUTHOR = {Kishimoto, Takashi and Prokhorov, Yuri and Zaidenberg,
              Mikhail},
     TITLE = {Affine cones over {F}ano threefolds and additive group
              actions},
   JOURNAL = {Osaka J. Math.},
  FJOURNAL = {Osaka Journal of Mathematics},
    VOLUME = {51},
      YEAR = {2014},
    NUMBER = {4},
     PAGES = {1093--1112},
      ISSN = {0030-6126},
   MRCLASS = {14J45 (14J50 14R05 14R20)},
  MRNUMBER = {3273879},
MRREVIEWER = {Andreas\ H\"oring},
       URL = {http://projecteuclid.org/euclid.ojm/1414761913},
}

@article {Saw24,
    AUTHOR = {Sawahara, Masatomo},
     TITLE = {Cylinders in canonical del {P}ezzo fibrations},
   JOURNAL = {Ann. Inst. Fourier (Grenoble)},
  FJOURNAL = {Universit\'e{} de Grenoble. Annales de l'Institut Fourier},
    VOLUME = {74},
      YEAR = {2024},
    NUMBER = {1},
     PAGES = {1--69},
      ISSN = {0373-0956,1777-5310},
   MRCLASS = {14R25 (14D06 14E30 14J17 14J26)},
  MRNUMBER = {4748167},
MRREVIEWER = {Jing\ Zhang},
       DOI = {10.5802/aif.3573},
       URL = {https://doi.org/10.5802/aif.3573},
}

@article {DK15,
    AUTHOR = {Dubouloz, Adrien and Kishimoto, Takashi},
     TITLE = {Log-uniruled affine varieties without cylinder-like open
              subsets},
   JOURNAL = {Bull. Soc. Math. France},
  FJOURNAL = {Bulletin de la Soci\'et\'e{} Math\'ematique de France},
    VOLUME = {143},
      YEAR = {2015},
    NUMBER = {2},
     PAGES = {383--401},
      ISSN = {0037-9484,2102-622X},
   MRCLASS = {14R25 (14D06 14E08 14J50)},
  MRNUMBER = {3351185},
MRREVIEWER = {J\'er\'emy\ Blanc},
       DOI = {10.24033/bsmf.2692},
       URL = {https://doi.org/10.24033/bsmf.2692},
}

@article {SY02,
    AUTHOR = {Shpilrain, Vladimir and Yu, Jie-Tai},
     TITLE = {Affine varieties with equivalent cylinders},
   JOURNAL = {J. Algebra},
  FJOURNAL = {Journal of Algebra},
    VOLUME = {251},
      YEAR = {2002},
    NUMBER = {1},
     PAGES = {295--307},
      ISSN = {0021-8693,1090-266X},
   MRCLASS = {14R10},
  MRNUMBER = {1900285},
MRREVIEWER = {Marcin\ Skrzy\'nski},
       DOI = {10.1006/jabr.2001.9124},
       URL = {https://doi.org/10.1006/jabr.2001.9124},
}

@article {KZ01,
    AUTHOR = {Kaliman, Shulim and Zaidenberg, Mikhail},
     TITLE = {Families of affine planes: the existence of a cylinder},
   JOURNAL = {Michigan Math. J.},
  FJOURNAL = {Michigan Mathematical Journal},
    VOLUME = {49},
      YEAR = {2001},
    NUMBER = {2},
     PAGES = {353--367},
      ISSN = {0026-2285,1945-2365},
   MRCLASS = {14R25 (14R10)},
  MRNUMBER = {1852308},
MRREVIEWER = {Tatiana\ M.\ Bandman},
       DOI = {10.1307/mmj/1008719778},
       URL = {https://doi.org/10.1307/mmj/1008719778},
}

@article {KKW24,
    AUTHOR = {Kim, In-Kyun and Kim, Jaehyun and Won, Joonyeong},
     TITLE = {K-unstable singular del {P}ezzo surfaces without anticanonical
              polar cylinder},
   JOURNAL = {Int. Math. Res. Not. IMRN},
  FJOURNAL = {International Mathematics Research Notices. IMRN},
      YEAR = {2024},

    VOLUME = {18},
     PAGES = {12599--12619},
      ISSN = {1073-7928,1687-0247},
   MRCLASS = {14J26 (32Q26)},
  MRNUMBER = {4798642},
MRREVIEWER = {James\ McKernan},
       DOI = {10.1093/imrn/rnae175},
       URL = {https://doi.org/10.1093/imrn/rnae175},
}

@article {KW25,
    AUTHOR = {Kim, Jaehyun and Won, Joonyeong},
     TITLE = {Cylinders in smooth del {P}ezzo surfaces of degree 2},
   JOURNAL = {Adv. Geom.},
  FJOURNAL = {Advances in Geometry},
    VOLUME = {25},
      YEAR = {2025},
    NUMBER = {1},
     PAGES = {71--91},
      ISSN = {1615-715X,1615-7168},
   MRCLASS = {14R25 (14J26)},
  MRNUMBER = {4854449},
MRREVIEWER = {Tatiana\ M.\ Bandman},
       DOI = {10.1515/advgeom-2024-0034},
       URL = {https://doi.org/10.1515/advgeom-2024-0034},
}

@article {KKW25,
    AUTHOR = {Kim, In-Kyun and Kishimoto, Takashi and Won, Joonyeong},
     TITLE = {Cylinders in weighted {F}ano hypersurfaces via projections and
              unprojections},
   JOURNAL = {Taiwanese J. Math.},
  FJOURNAL = {Taiwanese Journal of Mathematics},
    VOLUME = {29},
      YEAR = {2025},
    NUMBER = {6},
     PAGES = {1495--1505},
      ISSN = {1027-5487,2224-6851},
   MRCLASS = {14R20 (14R25)},
  MRNUMBER = {5003261},
       DOI = {10.11650/tjm/250502},
       URL = {https://doi.org/10.11650/tjm/250502},
}

@article {CPPZ21,
    AUTHOR = {Cheltsov, Ivan and Park, Jihun and Prokhorov, Yuri and
              Zaidenberg, Mikhail},
     TITLE = {Cylinders in {F}ano varieties},
   JOURNAL = {EMS Surv. Math. Sci.},
  FJOURNAL = {EMS Surveys in Mathematical Sciences},
    VOLUME = {8},
      YEAR = {2021},
    NUMBER = {1-2},
     PAGES = {39--105},
      ISSN = {2308-2151,2308-216X},
   MRCLASS = {14R25 (14E05 14E08 14E30 14J45 14J50 14R20)},
  MRNUMBER = {4307203},
       DOI = {10.4171/emss/44},
       URL = {https://doi.org/10.4171/emss/44},
}

@article {GKKP11,
    AUTHOR = {Greb, Daniel and Kebekus, Stefan and Kov\'acs, S\'andor J. and
              Peternell, Thomas},
     TITLE = {Differential forms on log canonical spaces},
   JOURNAL = {Publ. Math. Inst. Hautes \'Etudes Sci.},
  FJOURNAL = {Publications Math\'ematiques. Institut de Hautes \'Etudes
              Scientifiques},

    VOLUME = {114},
      YEAR = {2011},
     PAGES = {87--169},
      ISSN = {0073-8301,1618-1913},
   MRCLASS = {14F10 (14C20 14E15 14F05)},
  MRNUMBER = {2854859},
MRREVIEWER = {Anne-Sophie\ Kaloghiros},
       DOI = {10.1007/s10240-011-0036-0},
       URL = {https://doi.org/10.1007/s10240-011-0036-0},
}

@article {Ati57,
    AUTHOR = {Atiyah, M. F.},
     TITLE = {Complex analytic connections in fibre bundles},
   JOURNAL = {Trans. Amer. Math. Soc.},
  FJOURNAL = {Transactions of the American Mathematical Society},
    VOLUME = {85},
      YEAR = {1957},
     PAGES = {181--207},
      ISSN = {0002-9947,1088-6850},
   MRCLASS = {53.3X},
  MRNUMBER = {86359},
MRREVIEWER = {F.\ Hirzebruch},
       DOI = {10.2307/1992969},
       URL = {https://doi.org/10.2307/1992969},
}

@article {Bot57,
    AUTHOR = {Bott, Raoul},
     TITLE = {Homogeneous vector bundles},
   JOURNAL = {Ann. of Math. (2)},
  FJOURNAL = {Annals of Mathematics. Second Series},
    VOLUME = {66},
      YEAR = {1957},
     PAGES = {203--248},
      ISSN = {0003-486X},
   MRCLASS = {53.3X},
  MRNUMBER = {89473},
MRREVIEWER = {M.\ F.\ Atiyah},
       DOI = {10.2307/1969996},
       URL = {https://doi.org/10.2307/1969996},
}

@article {Mus02,
    AUTHOR = {Musta\c t\u a, Mircea},
     TITLE = {Singularities of pairs via jet schemes},
   JOURNAL = {J. Amer. Math. Soc.},
  FJOURNAL = {Journal of the American Mathematical Society},
    VOLUME = {15},
      YEAR = {2002},
    NUMBER = {3},
     PAGES = {599--615},
      ISSN = {0894-0347,1088-6834},
   MRCLASS = {14B05 (14B10 14E30)},
  MRNUMBER = {1896234},
MRREVIEWER = {Shihoko\ Ishii},
       DOI = {10.1090/S0894-0347-02-00391-0},
       URL = {https://doi.org/10.1090/S0894-0347-02-00391-0},
}

@misc{MY24,
 author={Moraga, Joaqu\'in and Y\'a\~nez, Jos\'e Ignacio},
 title={{Calabi--Yau} pairs of complexity two},
 year={2024}, eprint={2412.18830}, archivePrefix={arXiv}, primaryClass={math.AG}
}

@misc{ELY25,
 author={Enwright, Joshua and Li, Jennifer and Y\'a\~nez, Jos\'e Ignacio},
 title={Complexity one varieties are cluster type},
 year={2026}, eprint={2504.17369v2}, archivePrefix={arXiv}, primaryClass={math.AG},
 note={Version 2, 3 August 2026}
}

@article {AdSFM26,
    AUTHOR = {Alves da Silva, Eduardo and Figueroa, Fernando and Moraga,
              Joaqu\'in},
     TITLE = {Algebraic tori in the complement of quartic surfaces},
   JOURNAL = {J. Lond. Math. Soc. (2)},
  FJOURNAL = {Journal of the London Mathematical Society. Second Series},
    VOLUME = {113},
      YEAR = {2026},
    NUMBER = {1},
     PAGES = {Paper No. e70427, 26},
      ISSN = {0024-6107,1469-7750},
   MRCLASS = {14M25 (13F60 14E25)},
  MRNUMBER = {5021150},
       DOI = {10.1112/jlms.70427},
       URL = {https://doi.org/10.1112/jlms.70427},
}

@misc{EFMS25,
      title={Fano compactifications of mutation algebras}, 
      author={Joshua Enwright and Luca Francone and  Joaqu\'in Moraga and Hunter Spink},
      year={2025},
      eprint={2512.21839}        ,
      archivePrefix={arXiv},
      primaryClass={math.AG}
}

@misc{EFM24,
      title={Log {C}alabi-{Y}au pairs of birational complexity zero}, 
      author={Joshua Enwright and Fernando Figueroa and Joaqu\'in Moraga},
      year={2024},
      eprint={2404.05878}        ,
      archivePrefix={arXiv},
      primaryClass={math.AG}
}

@article{JM24,
 author={Ji, Lena and Moraga, Joaqu\'in},
 title={Toricity in families of {Fano} varieties},
 journal={Forum Math. Sigma}, volume={13}, year={2025}, pages={e39},
 note={Preprint arXiv:2409.03564}
}

@misc{Mor24a,
      title={Birational complexity and conic fibrations}, 
      author={Joaqu\'in Moraga},
      year={2024},
      eprint={2403.17251}        ,
      archivePrefix={arXiv},
      primaryClass={math.AG}
}

@article {Fil18,
AUTHOR = {Filipazzi, Stefano},
title={On a generalized canonical bundle formula and generalized adjunction},
journal = {Ann. Sc. Norm. Super. Pisa Cl. Sci.},
YEAR = {2020},
VOLUME = {21}, 
NUMBER ={5}, 
PAGES={1187--1221},
}

@misc{LMV24,
      title={Toric models of smooth {F}ano threefolds}, 
      author={Konstantin Loginov and Joaqu\'in Moraga and Artem Vasilkov},
      year={2024},
      eprint={2407.09420},
      archivePrefix={arXiv},
      primaryClass={math.AG},
}

@book {KM98,
    AUTHOR = {Koll\'{a}r, J\'{a}nos and Mori, Shigefumi},
     TITLE = {Birational geometry of algebraic varieties},
    SERIES = {Cambridge Tracts in Mathematics},
    VOLUME = {134},
      NOTE = {With the collaboration of C. H. Clemens and A. Corti,
              Translated from the 1998 Japanese original},
 PUBLISHER = {Cambridge University Press, Cambridge},
      YEAR = {1998},
     PAGES = {viii+254},
      ISBN = {0-521-63277-3},
   MRCLASS = {14E30},
  MRNUMBER = {1658959},
MRREVIEWER = {Mark Gross},
       DOI = {10.1017/CBO9780511662560},
       URL = {https://doi.org/10.1017/CBO9780511662560} ,
}

@book {K13,
    AUTHOR = {Koll\'{a}r, J\'{a}nos},
     TITLE = {Singularities of the minimal model program},
    SERIES = {Cambridge Tracts in Mathematics},
    VOLUME = {200},
      NOTE = {With a collaboration of S\'{a}ndor Kov\'{a}cs},
 PUBLISHER = {Cambridge University Press, Cambridge},
      YEAR = {2013},
     PAGES = {x+370},
      ISBN = {978-1-107-03534-8},
   MRCLASS = {14E30 (14B05)},
  MRNUMBER = {3057950},
MRREVIEWER = {Tommaso De Fernex},
       DOI = {10.1017/CBO9781139547895},
       URL = {https://doi.org/10.1017/CBO9781139547895}   ,
}

@article {KMM92,
    AUTHOR = {Koll\'{a}r, J\'{a}nos and Miyaoka, Yoichi and Mori, Shigefumi},
     TITLE = {Rationally connected varieties},
   JOURNAL = {J. Algebraic Geom.},
  FJOURNAL = {Journal of Algebraic Geometry},
    VOLUME = {1},
      YEAR = {1992},
    NUMBER = {3},
     PAGES = {429--448},
      ISSN = {1056-3911},
   MRCLASS = {14E30 (14E05 14E35)},
  MRNUMBER = {1158625},
MRREVIEWER = {Yuri G. Prokhorov},
}

@article {AHS08,
    AUTHOR = {Altmann, Klaus and Hausen, J\"urgen and S\"uss, Hendrik},
     TITLE = {Gluing affine torus actions via divisorial fans},
   JOURNAL = {Transform. Groups},
  FJOURNAL = {Transformation Groups},
    VOLUME = {13},
      YEAR = {2008},
    NUMBER = {2},
     PAGES = {215--242},
      ISSN = {1083-4362,1531-586X},
   MRCLASS = {14M25},
  MRNUMBER = {2426131},
MRREVIEWER = {Mina\ Teicher},
       DOI = {10.1007/s00031-008-9011-3},
       URL = {https://doi.org/10.1007/s00031-008-9011-3},
}

@article {CS08,
    AUTHOR = {Cheltsov, I. A. and Shramov, K. A.},
     TITLE = {Log-canonical thresholds for nonsingular {F}ano threefolds},
      NOTE = {With an appendix by J.-P. Demailly},
   JOURNAL = {Uspekhi Mat. Nauk},
  FJOURNAL = {Uspekhi Matematicheskikh Nauk},
    VOLUME = {63},
      YEAR = {2008},
    NUMBER = {5(383)},
     PAGES = {73--180},
      ISSN = {0042-1316,2305-2872},
   MRCLASS = {14J45 (14J30 32Q20)},
  MRNUMBER = {2484031},
MRREVIEWER = {Alexandr\ V.\ Pukhlikov},
       DOI = {10.1070/RM2008v063n05ABEH004561},
       URL = {https://doi.org/10.1070/RM2008v063n05ABEH004561},
}

@article {Zha06,
    AUTHOR = {Zhang, Qi},
     TITLE = {Rational connectedness of log {${\bf Q}$}-{F}ano varieties},
   JOURNAL = {J. Reine Angew. Math.},
  FJOURNAL = {Journal f\"{u}r die Reine und Angewandte Mathematik. [Crelle's
              Journal]},
    VOLUME = {590},
      YEAR = {2006},
     PAGES = {131--142},
      ISSN = {0075-4102},
   MRCLASS = {14E30 (14J45)},
  MRNUMBER = {2208131},
MRREVIEWER = {Stefan Kebekus},
       DOI = {10.1515/CRELLE.2006.006},
       URL = {https://doi.org/10.1515/CRELLE.2006.006},
}

@article {LS13,
    AUTHOR = {Liendo, Alvaro and S\"{u}ss, Hendrik},
     TITLE = {Normal singularities with torus actions},
   JOURNAL = {Tohoku Math. J. (2)},
  FJOURNAL = {The Tohoku Mathematical Journal. Second Series},
    VOLUME = {65},
      YEAR = {2013},
    NUMBER = {1},
     PAGES = {105--130},
      ISSN = {0040-8735},
   MRCLASS = {14J17 (14B05 14E15 14L30)},
  MRNUMBER = {3049643},
MRREVIEWER = {Sonia P\'{e}rez-D\'{\i}az},
       DOI = {10.2748/tmj/1365452628},
       URL = {https://doi.org/10.2748/tmj/1365452628} ,
}

@article {Mor19,
    AUTHOR = {Moraga, Joaqu\'{\i}n},
     TITLE = {Extracting non-canonical places},
   JOURNAL = {Adv. Math.},
  FJOURNAL = {Advances in Mathematics},
    VOLUME = {375},
      YEAR = {2020},
     PAGES = {107415, 12},
      ISSN = {0001-8708},
   MRCLASS = {14E30 (14B05)},
  MRNUMBER = {4170231},
       DOI = {10.1016/j.aim.2020.107415},
       URL = {https://doi.org/10.1016/j.aim.2020.107415} ,
}

@article {BMSZ18,
    AUTHOR = {Brown, Morgan V. and McKernan, James and Svaldi, Roberto and
              Zong, Hong R.},
     TITLE = {A geometric characterization of toric varieties},
   JOURNAL = {Duke Math. J.},
  FJOURNAL = {Duke Mathematical Journal},
    VOLUME = {167},
      YEAR = {2018},
    NUMBER = {5},
     PAGES = {923--968},
      ISSN = {0012-7094},
   MRCLASS = {14E30 (14M25)},
  MRNUMBER = {3782064},
MRREVIEWER = {Jaros\l aw A. Wi\'{s}niewski},
       DOI = {10.1215/00127094-2017-0047},
       URL = {https://doi.org/10.1215/00127094-2017-0047},
}

@book {CLS11,
    AUTHOR = {Cox, David A. and Little, John B. and Schenck, Henry K.},
     TITLE = {Toric varieties},
    SERIES = {Graduate Studies in Mathematics},
    VOLUME = {124},
 PUBLISHER = {American Mathematical Society, Providence, RI},
      YEAR = {2011},
     PAGES = {xxiv+841},
      ISBN = {978-0-8218-4819-7},
   MRCLASS = {14M25 (05A15 05E45 52B12)},
  MRNUMBER = {2810322},
MRREVIEWER = {Ivan Arzhantsev},
       DOI = {10.1090/gsm/124},
       URL = {https://doi.org/10.1090/gsm/124} ,
}

@misc{Mor22,
      title={Coregularity of {F}ano varieties}, 
      author={Joaquín Moraga},
      year={2022},
      eprint={2206.10834},
      archivePrefix={arXiv},
      primaryClass={math.AG}
}

@book {Con82,
    AUTHOR = {Conte, A.},
     TITLE = {Introduzione alle variet\`a algebriche a tre dimensioni},
    SERIES = {Quaderni dell'Unione Matematica Italiana [Publications of the
              Italian Mathematical Union]},
    VOLUME = {22},
 PUBLISHER = {Pitagora Editrice, Bologna},
      YEAR = {1982},
     PAGES = {133},
   MRCLASS = {14J30 (14-02)},
  MRNUMBER = {738782},
MRREVIEWER = {Luciana\ Picco Botta},
}

@article{MM24,
 author={Mauri, Mirko and Moraga, Joaqu\'in},
 title={Birational complexity and dual complexes},
 journal={Proc. Lond. Math. Soc. (3)}, volume={131}, number={2},
 year={2025}, pages={e70073}, doi={10.1112/plms.70073},
 note={Preprint arXiv:2402.10136}
}

@article {MM81,
    AUTHOR = {Mori, Shigefumi and Mukai, Shigeru},
     TITLE = {Classification of {F}ano {$3$}-folds with {$B\sb{2}\geq 2$}},
   JOURNAL = {Manuscripta Math.},
  FJOURNAL = {Manuscripta Mathematica},
    VOLUME = {36},
      YEAR = {1981/82},
    NUMBER = {2},
     PAGES = {147--162},
      ISSN = {0025-2611,1432-1785},
   MRCLASS = {14J30 (14J10)},
  MRNUMBER = {641971},
MRREVIEWER = {Mary\ Schaps},
       DOI = {10.1007/BF01170131},
       URL = {https://doi.org/10.1007/BF01170131},
}

@article {MS21,
    AUTHOR = {Moraga, Joaquin and Svaldi, Roberto},
     TITLE = {A geometric characterization of toric singularities},
   JOURNAL = {J. Math. Pures Appl. (9)},
  FJOURNAL = {Journal de Math\'ematiques Pures et Appliqu\'ees. Neuvi\`eme
              S\'erie},
    VOLUME = {195},
      YEAR = {2025},
     PAGES = {Paper No. 103620, 61},
      ISSN = {0021-7824,1776-3371},
   MRCLASS = {14E30 (14B05 14M25)},
  MRNUMBER = {4867812},
MRREVIEWER = {Daniel\ Duarte},
       DOI = {10.1016/j.matpur.2024.103620},
       URL = {https://doi.org/10.1016/j.matpur.2024.103620},
}

@article {AB92,
    AUTHOR = {Alzati, Alberto and Bertolini, Marina},
     TITLE = {On the rationality of {F}ano {$3$}-folds with {$B_2\geq 2$}},
   JOURNAL = {Matematiche (Catania)},
  FJOURNAL = {Le Matematiche},
    VOLUME = {47},
      YEAR = {1992},
    NUMBER = {1},
     PAGES = {63--74},
      ISSN = {0373-3505,2037-5298},
   MRCLASS = {14J45 (14E05)},
  MRNUMBER = {1229454},
MRREVIEWER = {Alessio\ Corti},
}

@article{ALP23,
 author={Avilov, Artem and Loginov, Konstantin and Przyjalkowski, Victor},
 title={Coregularity of smooth {Fano} threefolds},
 journal={Commun. Number Theory Phys.}, volume={18}, number={3},
 year={2024}, pages={509--577}, note={Preprint arXiv:2309.16784}
}

@article {Duc24,
    AUTHOR = {Ducat, Tom},
     TITLE = {Quartic surfaces up to volume preserving equivalence},
   JOURNAL = {Selecta Math. (N.S.)},
  FJOURNAL = {Selecta Mathematica. New Series},
    VOLUME = {30},
      YEAR = {2024},
    NUMBER = {1},
     PAGES = {Paper No. 2, 27},
      ISSN = {1022-1824,1420-9020},
   MRCLASS = {14E05 (14C20 14M25)},
  MRNUMBER = {4667347},
MRREVIEWER = {Alexandr\ V.\ Pukhlikov},
       DOI = {10.1007/s00029-023-00890-7},
       URL = {https://doi.org/10.1007/s00029-023-00890-7},
}

@misc{Mor24b,
      title={Birational complexity of log {C}alabi-{Y}au 3-folds}, 
      author={Joaqu\'in Moraga},
      year={2024},
      eprint={2405.18516} ,
      archivePrefix={arXiv},
      primaryClass={math.AG}
}

@misc{IP99,
      title={Algebraic Geometry V: Fano Varieties}, 
      author={Iskovskikh, V. A. and Prokhorov, Yu. G.},
      year={1999},
      eprint={Encyclopaedia of Mathematical Sciences, Book 47. Springer, 1999} ,
      archivePrefix={arXiv},
      primaryClass={math.AG}
}

@misc{Isk77,
      title={Fano 3-folds, I}, 
      author={V. A. Iskovskikh.},
      year={1977},
      eprint={Izv. Akad. Nauk SSSR Ser. Mat. 41 (1977) 516--
562 (Russian), English transl.: Math. USSR, Izv. 11, 485--527.} ,
      primaryClass={math.AG}
}

@misc{Isk78,
      title={Fano 3-folds, II}, 
      author={V. A. Iskovskikh.},
      year={1978},
      eprint={Izv. Akad. Nauk SSSR Ser. Mat. 42 (1978) 504--549 (Russian), English transl.: Math. USSR, Izv. 
12, 469--506.} ,
      primaryClass={math.AG}
}

@misc{Isk89,
      title={A double projection from a line on Fano threefolds of the first kind}, 
      author={V. A. Iskovskikh.},
      year={1989},
      eprint={Math. USSR Sb. 66 265--284 (1990); translation from Mat. Sb. 180(2) 260--278 (1989).} ,
      primaryClass={math.AG}
}

@article{GHK15,
 author={Gross, Mark and Hacking, Paul and Keel, Sean},
 title={Mirror symmetry for log {Calabi--Yau} surfaces {I}},
 journal={Publ. Math. Inst. Hautes \'Etudes Sci.}, volume={122},
 year={2015}, pages={65--168}, doi={10.1007/s10240-015-0073-1}
}

@misc{FFMP22,
      title={Complements and coregularity of {F}ano varieties}, 
      author={Fernando Figueroa and  Stefano Filipazzi and Joaquin Moraga and Junyao Peng},
        year={2022},
      eprint={2211.09187},
      archivePrefix={arXiv},
      primaryClass={math.AG}
}

@incollection {Cor25,
    AUTHOR = {Corti, Alessio},
     TITLE = {Cluster Varieties and Toric Specializations of {F}ano Varieties},
 BOOKTITLE = {Higher Dimensional Algebraic Geometry -- A Volume in Honor of V. V. Shokurov},
    SERIES = {London Mathematical Society Lecture Note Series},
    VOLUME = {489},
     PAGES = {264--285},
 PUBLISHER = {Cambridge University Press},
      YEAR = {2025},
       DOI = {10.1017/9781009396233.017},
       URL = {https://doi.org/10.1017/9781009396233.017},
}

@article{CPW16,
 author = {Cheltsov, Ivan and Park, Jihun and Won, Joonyeong},
 title = {Affine cones over smooth cubic surfaces},
 fjournal = {Journal of the European Mathematical Society (JEMS)},
 journal = {J. Eur. Math. Soc. (JEMS)},
 issn = {1435-9855},
 volume = {18},
 number = {7},
 pages = {1537--1564},
 year = {2016},
 language = {English},
 doi = {10.4171/JEMS/622},
 zbMATH = {6603460},
 Zbl = {1386.14068}
}

@article{Kaw07,
 author = {Kawakita, Masayuki},
 title = {Inversion of adjunction on log canonicity},
 fjournal = {Inventiones Mathematicae},
 journal = {Invent. Math.},
 issn = {0020-9910},
 volume = {167},
 number = {1},
 pages = {129--133},
 year = {2007},
 language = {English},
 doi = {10.1007/s00222-006-0008-z},
 zbMATH = {5118884},
 Zbl = {1114.14009}
}

@book{huy23,
 author = {Huybrechts, Daniel},
 title = {The geometry of cubic hypersurfaces},
 fseries = {Cambridge Studies in Advanced Mathematics},
 series = {Camb. Stud. Adv. Math.},
 volume = {206},
 isbn = {978-1-00-928000-6; 978-1-00-928002-0},
 year = {2023},
 publisher = {Cambridge: Cambridge University Press},
 language = {English},
 doi = {10.1017/9781009280020},
 zbMATH = {7671697},
 Zbl = {1561.14074}
}

@article{BW79,
  author  = {Bruce, J. W. and Wall, C. T. C.},
  title   = {On the classification of cubic surfaces},
  journal = {Journal of the London Mathematical Society},
  series  = {2},
  volume  = {19},
  number  = {2},
  year    = {1979},
  pages   = {245--256},
  doi     = {10.1112/jlms/s2-19.2.245}
}

@article{MM03,
 author={Mori, Shigefumi and Mukai, Shigeru},
 title={Erratum: ``Classification of {Fano} {$3$}-folds with {$B_2\geq2$}''},
 journal={Manuscripta Math.}, volume={110}, number={3}, year={2003}, pages={407},
 doi={10.1007/s00229-002-0336-2}
}

@article{Bat82,
 author={Batyrev, Victor V.}, title={Toroidal {Fano} {$3$}-folds},
 journal={Math. USSR-Izv.}, volume={19}, number={1}, year={1982}, pages={13--25}
}

@article{WW82,
 author={Watanabe, Keiichi and Watanabe, Masayuki},
 title={The classification of {Fano} {$3$}-folds with torus embeddings},
 journal={Tokyo J. Math.}, volume={5}, number={1}, year={1982}, pages={37--48}
}

@article{KO73,
 author={Kobayashi, Shoshichi and Ochiai, Takushiro},
 title={Characterizations of complex projective spaces and hyperquadrics},
 journal={J. Math. Kyoto Univ.}, volume={13}, number={1}, year={1973}, pages={31--47}
}

@article{CPWsing16,
 author={Cheltsov, Ivan and Park, Jihun and Won, Joonyeong},
 title={Cylinders in singular del {Pezzo} surfaces},
 journal={Compos. Math.}, volume={152}, number={6}, year={2016}, pages={1198--1224},
 doi={10.1112/S0010437X16007284}
}

@incollection{Mukai88,
 author={Mukai, Shigeru}, title={Curves, {$K3$} surfaces and {Fano} {$3$}-folds of genus {$\leq10$}},
 booktitle={Algebraic geometry and commutative algebra, Vol. I},
 publisher={Kinokuniya}, address={Tokyo}, year={1988}, pages={357--377}
}

@article{Mukai89,
 author={Mukai, Shigeru},
 title={Biregular classification of {Fano} {$3$}-folds and {Fano} manifolds of coindex 3},
 journal={Proc. Natl. Acad. Sci. USA}, volume={86}, number={9}, year={1989}, pages={3000--3002},
 doi={10.1073/pnas.86.9.3000}
}

@article{FZ02,
 author={Fomin, Sergey and Zelevinsky, Andrei},
 title={Cluster algebras {I}: Foundations}, journal={J. Amer. Math. Soc.},
 volume={15}, number={2}, year={2002}, pages={497--529}, doi={10.1090/S0894-0347-01-00385-X}
}

@article{Scott06,
 author={Scott, Joshua S.}, title={Grassmannians and cluster algebras},
 journal={Proc. Lond. Math. Soc. (3)}, volume={92}, number={2}, year={2006}, pages={345--380},
 doi={10.1112/S0024611505015571}
}

@article{DK19,
 author={Dubouloz, Adrien and Kishimoto, Takashi},
 title={Cylindres dans les fibrations de {Mori}: formes du volume quintique de del {Pezzo}},
 journal={Ann. Inst. Fourier (Grenoble)}, volume={69}, number={6},
 year={2019}, pages={2377--2393}, doi={10.5802/aif.3297}
}

\vspace{0.5cm}
\end{document}